\documentclass[11pt,a4paper]{article}

\usepackage[margin=1.0in]{geometry}

\renewcommand{\theequation}{\thesection.\arabic{equation}}

\renewcommand\thesection{\arabic{section}}
\usepackage{amsthm}
\usepackage{amsmath,amssymb,cite}
\usepackage{graphics}                 
\usepackage{graphicx}
\usepackage{empheq}
\usepackage{float}
\usepackage{array}
\usepackage{epsf}
\usepackage{color}                
\usepackage{accents}
\usepackage{hyperref}
\usepackage{relsize}
\usepackage{phonetic}
\usepackage{boldline}
\usepackage{bm}
\usepackage{epsfig}
\usepackage{caption}
\usepackage{longtable}
\usepackage{pgfplots}
\usepackage[utf8]{inputenc}
\usepackage[T1]{fontenc}
\usetikzlibrary{arrows.meta}

\newtheorem{theorem}{Theorem}[section]
\newtheorem{remark}{\bf Remark}[section]
\newtheorem{lemma}{\bf Lemma}[section]
\newtheorem{definition}{\bf Definition}[section]
\newtheorem{example}{\bf Example}[section]
\newtheorem{proposition}[theorem]{Proposition}

\newcommand{\nn}{\nonumber}

\newcommand{\D}{\mathrm{div}}

\def\bv{{\bf v}} 
\def\u{{\bf u}}
\def\e{{\bf e}} 
\def\w{{\bf w}}

\def\bg{{\bf g}}
\def\f{{\bf f}}
\def\s{{\bf s}}

\def\n{{\bf n}} 
 
\def\S{{\bf S}}
\def\q{{\bf q}}
\def\bzeta{\bb {\zeta}}
\def\bvarphi{\bb {\varphi}}
\def\bxi{\bb {\xi}}
\def\bomega{\bb{\omega}}
\def\DD{{\bf D}}
 \def\refe#1{(\ref{#1})}
\def\x{{\bf x}} 
\def\X{{\bf X}}
\def \endproof{\vrule height8pt width 5pt depth 0pt}
\def \bb#1{\setbox0=\hbox{$#1$}
           \kern-.0em\copy0\kern-\wd0
            \kern.0em\copy0\kern-\wd0
           \kern-.0em\raise.04em\box0}
\def\e{\bb{\eta}} 
\def\sig{\bb{\sigma}}
\numberwithin{equation}{section}
\numberwithin{definition}{section}
\numberwithin{theorem}{section}

\allowdisplaybreaks

\usepackage{lipsum}

\let\OLDthebibliography\thebibliography
\renewcommand\thebibliography[1]{
  \OLDthebibliography{#1}
  \setlength{\parskip}{5pt}
  \setlength{\itemsep}{0pt plus 0.3ex}
}

\begin{document}

\date{}

\title{Optimal $L^2$ error analysis of a loosely coupled finite element scheme for thin-structure interactions
} 

\author{
\setcounter{footnote}{0}
Buyang Li  
\footnote{
Department of Applied Mathematics, The Hong Kong Polytechnic University, Hung Hom, Hong Kong. E-mail address: buyang.li@polyu.edu.hk and yu-pei.xie@connect.polyu.hk. 
  }
\,\,\, 
Weiwei Sun
\footnote{Advanced Institute of Natural Sciences, Beijing Normal University, Zhuhai,  519087, P.R. China. E-mail address: maweiw@uic.edu.cn. 
  }
\footnote{
   Guangdong Provincial Key Laboratory of Interdisciplinary Research and Application for Data Science, BNU-HKBU United International College, Zhuhai, 519087, P.R.China.} 
\,\,\, Yupei Xie \footnotemark[1] 
\,\,\, 
and
\,\, 
Wenshan Yu \footnotemark[3] 
   \footnote{Hong Kong Baptist University, Kowloon Tong, Hong Kong. E-mail address: yuwenshan@uic.edu.cn
} 
}
\maketitle

\begin{abstract} 
Finite element methods and kinematically coupled schemes that decouple the fluid velocity and structure displacement have been extensively studied for incompressible fluid-structure interaction (FSI) over the  past decade. While these methods are known to be stable and easy to implement, optimal error analysis has remained challenging. Previous work has primarily relied on the classical elliptic projection technique, which is only suitable for parabolic problems and does not lead to optimal convergence of numerical solutions to the FSI problems in the standard $L^2$ norm. In this article, we propose a new stable fully discrete kinematically coupled scheme for incompressible FSI thin-structure model and establish a new approach for the numerical analysis of FSI problems in terms of a newly introduced coupled non-stationary Ritz projection, which allows us to prove the optimal-order convergence of the proposed method in the $L^2$ norm. The methodology presented in this article is also applicable to numerous other FSI models and serves as a fundamental tool for advancing research in this field. \\ 

\end{abstract} 




\setlength\abovedisplayskip{4pt}
\setlength\belowdisplayskip{4pt}

\section{Introduction} 
\setcounter{equation}{0}

There has been increasing interest in studying fluid-structure interaction due to its diverse applications in many areas \cite{Dowell-FSI-reviewmodel, Formaggia-Cardio-book-2009, Hou-Wang-Layton-reviewcomp, JuLIU-2021-Notices-AMS,FNobile-PhDThesis}. Numerical simulations are crucial in this field, and over the past two decades, numerous efforts have been devoted to developing efficient numerical algorithms and analysis methods.

This paper focus on a commonly-used academic model problem, 
 where an incompressible fluid interacts with thin structure described by some lower-dimensional, linearly elastic model (for example, membrane, shell and etc).
This thin-structure interaction 
model is  described by the following equations 
\begin{empheq}[left=\empheqlbrace]{align}
\rho_f \partial_t \u - \D \,\sig(\u, p) & = 0, \qquad \qquad \mbox{in } (0,T) \times \Omega,
\nn \\  
\D \, \u & = 0 ,  \qquad \qquad \mbox{in } (0,T) \times \Omega,
 \label{f-e}  \\ 
 \u(0, \cdot) & = \u_0(x) , \qquad \mbox{on } \Omega,
 \nn
\end{empheq}
\begin{empheq}[left=\empheqlbrace]{align}
\rho_s \epsilon_s \partial_{tt} \, \e - {\cal L}_s \e &= - \sig(\u,p)\n , \, \, \qquad \mbox{in } (0,T) \times \Sigma,
\nn \\  
\e(0, x) &= \e_0(x),   \qquad \qquad \mbox{on }  \Sigma, 
 \label{s-e}  
 \\ 
 \partial_t \, \e(0, x) &= \u_0(x), \qquad  \qquad \mbox{on }  \Sigma
 \nn 
\end{empheq}
with the kinematic interface condition 
\begin{align} 
\partial_t \e = \u  \qquad \mbox{on } (0,T) \times \Sigma 
\label{bc} 
\end{align} 
and certain inflow and outflow conditions at $\Sigma_l$ and $\Sigma_r$; see Figure \ref{figure-1}. 
The unknown solutions in \eqref{f-e}--\eqref{bc} are fluid velocity $\u$, fluid pressure $p$ and structure displacement $\e$, and the following notations are used:
\begin{longtable}{p{4.5cm}p{9cm}}
	$\epsilon_s$:
	&The thickness of the structure.\\
	
	$\mu$:
	&The fluid viscosity.\\
	
	$\rho_f$:
	&The fluid density.\\
	
	$\rho_s$:
	&The structure density.\\
	
	$\n$:
	&The outward normal vector on $\partial \Omega$.\\
	
	$\sig(\u, p) = -pI + 2 \mu \DD(\u)$: 
	& The fluid stress tensor. \\
	
	$\DD(\u) = \frac{1}{2} ( \nabla \u + (\nabla \u)^T)$: 
	& The strain-rate tensor. \\	
	
	${\cal L}_s$: 
	&An elliptic differential operator on $\Sigma$, such as ${\cal L}_s= - I + \Delta_s$, where $\Delta_s$ is the Laplace-Beltrami operator on $\Sigma$. 
\end{longtable}

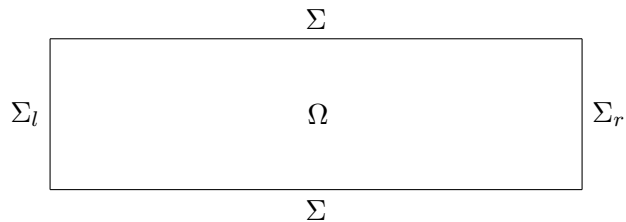
\begin{figure}[htp!]
\begin{center}
\begin{tikzpicture}
\draw[-,] (0,0) -- (7,0) ;
\draw[-,] (7,0) -- (7,2) ;
\draw[-, ] (7,2) -- (0,2) ;
\draw[-,] (0,2) -- (0,0) ;
\draw (3.8,1) node[left]{$\Omega$} ; 
\draw (3.5,2) node[above]{$\Sigma$} ; 
\draw (3.5,0) node[below]{$\Sigma$} ; 
\draw (0,1) node[left]{$\Sigma_l$} ; 
\draw (7,1) node[right]{$\Sigma_r$} ; 
\end{tikzpicture}
\end{center} \vspace{-10pt}
\caption{The computational domain in the thin-structure interaction problem}
\label{figure-1}
\end{figure}

In general, two strategies can be employed to construct numerical schemes for solving fluid-structure interaction problems. Monolithic algorithms solve a fully coupled system, which can be expensive for complex fluid-structure problems. Various studies have focused on the numerical simulation and analysis of monolithic algorithms, as can be found in \cite{Sun-Xu-2021-monolithic, JS-monolithic-lecturenotes, GKW-monolithic, PSun-2020_monolithic-simplified, JuLIU-2022-CMAME, LiuJ-JCP-2014,FNobile-PhDThesis}. 
Alternatively, 
the fluid and structure sub-problems can be solved separately by 
partitioned type schemes.
A strongly-coupled partitioned scheme often requires extra iterations 
for the sub-problems at each time step to obtain the solution which at convergence coincides with the monolithic one \cite{FNobile-PhDThesis, Fer-Ger-G-2007-ProjectionSemi-implicitScheme}, while 
the extra iterations are not needed in loosely coupled partitioned schemes.
However, the stability is a key issue 
for loosely coupled partitioned schemes,  
which may be hard to be ensured for highly added mass effect problems such as hemodynamics (e.g.\cite{AddedMass-NumericalInstability}). 
The development and study of stable loosely coupled partitioned schemes have been  an active area of research 
(e.g. 	\cite{Fer-B-2009-Stabilization-Nitsche, burman14, banks2014, gigante21, GGCC-2009}). 

Among those loosely coupled partitioned schemes, the kinematically coupled scheme is the most popular one due to its modularity, stability, and ease of implementation. The scheme was first studied in \cite{GGCC-2009} for the fluid-structure interaction problems and subsequently by numerous researchers \cite{BCGTQ,BM-2016SINUM, BukacT-2022-thin, LRH-nonNewtonian, Bukac2018SINUM}. However, the analysis of kinematically coupled schemes has been challenging due to the specific coupling of two distinct physical phenomena. In \cite{Fer-2013NumMath}, Fernandez proposed an incremental displacement-correction scheme, which proved to be stable, and the following energy-norm error estimate was established using piecewise polynomials of degree $k$ for both $\u_h^n$ and $\e_h^n$ in \refe{error-0}, i.e.,
\begin{align} 
 \| \u^n - \u_h^n \|_{L^2(\Omega)}
 + \Big ( \sum_{m=1}^n \tau \| \u^m - \u_h^m \|_f^2 \Big )^{\frac12} 
 +  \| \u^n - \u_h^n \|_{L^2(\Sigma)}  
 + \| \e^n - \e_h^n \|_s 
 \le C(\tau + h^k ) . 
 \label{error-0} 
\end{align} 
The above estimate is optimal only for the velocity in the weak $H^1$-norm (more precisely, $L^2(H^1)$-norm) and not optimal in $L^2$-norm. 
Several different schemes were investigated, and similar error estimates, such as those given in \cite{BM-2016SINUM, Bukac2018SINUM}, were provided. The kinematic coupling has been extended to other applications, such as composite structures and non-Newtonian flow \cite{BCM-composite-JCP-2015, LRH-nonNewtonian}, by many researchers. 
Recently, a fully discrete loosely coupled Robin-Robin scheme for thick structures was proposed in \cite{BDFG-2022}, where they showed that the error estimate in the same energy norm as in \refe{error-0} is in the order of $O(\sqrt{\tau} + h)$ for $k=1$. Additionally, a splitting scheme was proposed in \cite{AFL-2023} for the fluid-structure interaction problem with immersed thin-walled structures. The scheme was proved to be unconditionally stable, and a suboptimal $L^2$-norm error estimate was presented.
 
Optimal $L^2$-norm error estimates play a crucial role in both theoretical analysis of algorithms and development of novel algorithms for practical applications. However, to our knowledge, such results have not been established due to the lack of properly defined Ritz projections for fluid-structure interaction problems. This is in contrast to the error analysis of finite element methods for parabolic equations, where the Ritz projections have been well defined since the early work of Wheeler \cite{Whe}. For instance, for the heat equation  
$\partial_tu-\Delta u=f$, the Ritz projection is a finite element function $R_hu$ that satisfies the weak formulation:
\begin{align}
\int_\Omega \nabla (u - R_h u)\cdot \nabla v_h d x = 0 \quad\mbox{for all finite element functions $v_h$}.
\label{e-p-0}
\end{align}
With this projection $R_h$, the error of the finite element solution can be decomposed into two parts:
$$
u - u_h = (u -R_hu) + (R_hu - u_h). 
$$
In the analysis of the second part, the pollution from the approximation of the diffusion term is not involved, thus enabling the establishment of an optimal-order error estimate for $\| R_h u - u_h \|_{L^2(\Omega)}$. The optimal estimate for $\|u - u_h \|_{L^2(\Omega)}$ can be derived from the fact that the projection error $\|u -R_hu \|_{L^2(\Omega)}$ is also of optimal order. However, formulating and determining optimal $L^2$-norm error estimates for a suitably defined Ritz projection in fluid-structure interaction systems remains a challenge. The standard elliptic Ritz projection for the Stokes equations, while widely employed for obtaining error estimates in the energy norm, no longer produces optimal $L^2$-norm error estimates for such fluid-structure interaction systems; see \cite{AFL-2023,BM-2016SINUM, Fer-2013NumMath, PSun-2020_monolithic-simplified,Bukac2018SINUM}.

In this article, we propose a new kinematically coupled scheme which decouples $(\u,p)$ and $\e$ for solving the thin-structure interaction problem, and demonstrate its unconditional stability for long-time computation. More importantly, we establish an optimal $L^2$-norm error estimate for the proposed method, i.e., 
\begin{align} 
 \| \u^n - \u_h^n \|_{L^2(\Omega)}
+  \| \u^n - \u_h^n \|_{L^2(\Sigma)}  
 + \| \e^n - \e_h^n \|_{L^2(\Sigma)} 
 \le C( \tau + h^{k+1} ) \, ,
\end{align} 
by developing a new framework for the numerical analysis of fluid-structure interaction problems in terms of a newly introduced coupled non-stationary Ritz projection, which is defined as a triple of finite element functions $(R_h\u, R_h p, R_h\e) $ satisfying a weak formulation plus a constraint condition $(R_h \u)|_\Sigma = \partial_t R_h \e$ on $\Sigma\times[0,T]$. This is equivalent to solving an evolution equation of $R_h \e$ under some initial condition $R_h \e(0)$. 
Moreover, the dual problem of the non-stationary Ritz projection, required in the optimal $L^2$-norm error estimates for the fluid-structure interaction problem, is a backward initial-boundary value problem
	\begin{subequations}\label{dual-phi-eq-0}
		\begin{align}
			- \mathcal{L}_s {\bb \phi} + {\bb \phi} &= \partial_t \sig ({\bb \phi}, q) \n +\f && \text{on}\,\,\, \Sigma\times [0,T)
			 &&\mbox{\rm(the boundary condition)}
			\\
			-\nabla \cdot \sigma ({\bb \phi}, q) + {\bb \phi} &=  0 && \text{in}\,\,\, \Omega\times [0,T) \\
			\nabla \cdot {\bb \phi} &=  0 && \text{in}\,\,\, \Omega \times [0,T) \\
			\sig ({\bb \phi}, q) \n &=0 &&\mbox{at}\,\,\, t=T &&\mbox{\rm(the initial condition)} .
		\end{align} 
	\end{subequations}
which turns out to be equivalent to a backward evolution equation of ${\bb \xi}=\sig ({\bb \phi}, q) \n$, i.e., 
	\begin{equation}\label{xi-strong-eq-0} 
		- \mathcal{L}_s \mathcal{N} {\bb \xi} + \mathcal{N} {\bb \xi} - \partial_t \mathcal{}
		{\bb \xi} = \f \,\,\, \text{on}\,\,\, \Sigma\times[0,T), \,\,\,\mbox{with initial condition}\,\,\, {\bb \xi} (T) = 0 , 
	\end{equation}
where $\mathcal{N}: H^{-\frac12}(\Sigma)^d\rightarrow H^{\frac12}(\Sigma)^d$ is the Neumann-to-Dirichlet map associated to the Stokes equations. By choosing a well-designed initial value $R_h\e(0)$ and utilizing the regularity properties of the dual problem \eqref{dual-phi-eq-0}, which are shown by analyzing the equivalent formulation in \eqref{xi-strong-eq-0}, we are able to establish optimal $L^2$ error estimates for the non-stationary Ritz projection and, subsequently, optimal $L^2$-norm error estimates for the finite element solutions of the thin-structure interaction problem. 


The rest of this article is organized as follows. In Section 2, we introduce a kinematically coupled scheme and present our main theoretical results on the unconditional stability and optimal $L^2$-norm error estimates of the scheme. We focus on a first-order kinematically coupled time-stepping method and the class of $H^1$-conforming inf-sup stable finite element spaces, including the classical Taylor--Hood and MINI elements. 
In Section 3, we introduce a new non-stationary coupled Ritz projection and present the corresponding projection error estimates (with its proof deferred to Section 4). Then we establish unconditionally stability and optimal $L^2$-norm error estimates for the fully discrete finite element solutions by utilizing the error estimates for the non-stationary coupled Ritz projection. 
Section 4 is devoted to the proof of the error estimates of the non-stationary coupled Ritz projection. We present a well-designed initial value of the projection and the corresponding error estimates based on duality arguments on the thin solid structure. 
In Section 5, we provide three numerical examples to support the theoretical analysis presented in this article. The first example illustrates the optimal $L^2$-norm convergence of the proposed fully discrete kinematically coupled scheme. The second example demonstrates the simulation of certain physical features, which are consistent with previous works. The third example is the 3D simulation of common cardiac arteries in hemodynamics.


\section{Notations, assumptions and main results} 
\setcounter{equation}{0}
In this section, we propose a fully discrete stabilized kinematically coupled FEM for the fluid-structure problem \refe{f-e}--\refe{bc}, as well as the main theoretical results of unconditional stability and optimal-order convergence in the $L^2$ norm. 

\subsection{Notation and weak formulation}
Some standard notations and operators are defined below. For any two function $u$, $v \in L^{2}(\Omega)$, we denote the inner products and norms of $L^2(\Omega)$ and $L^2(\Sigma)$ by
\begin{align}
(u,v)  &= \int_{\Omega} u(\x) v(\x)\, {\mathrm{d}} \x,
\qquad\, \left \|u\right \|^2:= (u,u) , 
\nn\\
(w,  \xi)_\Sigma  &= \int_{\Sigma} w(\x) \xi(\x)\, {\mathrm{d}} \x,
\qquad \| w\|_\Sigma^2 := (w, w)_\Sigma.
\nn 
\end{align} 
We assume that $\Omega \subset \mathbb{R}^d$ ($d=2,3$) is a bounded domain with $\partial \Omega = \Sigma_l \cup \Sigma_r \cup 
\Sigma$, where $\Sigma$ denotes the fluid-structure interface,  
$\Sigma_l$ and  $\Sigma_r$ are two disks (or lines in 2-dimensional case) denoting the inflow and outflow boundary and $\Sigma_r = \{ (x,y,z+ L) : (x,y,z) \in \Sigma _l \, \mbox{ for some } L>0\}$. 

For the simplicity of analysis, we consider the problem with the periodic boundary condition on $\Sigma_l$ and $\Sigma_r$, and assume that the extended domains $\Omega_\infty$ and $\Sigma_\infty$ are smooth, where  
\begin{align*}
	\Omega_\infty&:=\{(x,y,z):\exists k\in\mathbb{Z} \text{ such that }(x,y,z+Lk)\in \Omega\cup\Sigma_l \} , \\
	\Sigma_\infty&:=\{(x,y,z):\exists k\in\mathbb{Z}\text{ such that }(x,y,z+Lk)\in \overline\Sigma\} . 
\end{align*} 
The structure is assumed to be a linear thin-solid (e.g., string in two-dimensional model and membrane for three-dimensional model). Nonetheless, the algorithm presented in this paper is also applicable to problems with more general boundary conditions and domains. 

We say a function $f$ defined in $\Omega_\infty$ is periodic if 
\begin{align*}
	f(x,y,z)=f(x,y,z+kL) \quad \forall (x,y,z)\in \Omega\cup\Sigma_l \quad\forall\, k\in\mathbb{Z} . 
\end{align*}
The space of periodic smooth functions on $\Omega_\infty$ is denoted as $C^\infty(\Omega_\infty)$.
The periodic Sobolev spaces $H^s(\Omega)$ and $H^s(\Sigma)$, with $s\ge 0$, are defined as
\begin{align*}
	H^s(\Omega)&:=\text{The closure of $C^\infty(\Omega_\infty)$ under the conventional norm of $H^s(\Omega)$}  , \\ 
	H^s(\Sigma)&:=\text{The closure of $C^\infty(\Sigma_\infty)$ under the conventional norm of $H^s(\Sigma)$}  . 
\end{align*} 
which are equivalent to the Sobolev spaces by considering $\Omega$ and $\Sigma$ as tori in the $z$ direction. 
The dual spaces of $H^s(\Omega)$ and $H^s(\Sigma)$ are denoted by $H^{-s}(\Omega)$ and $H^{-s}(\Sigma)$, respectively. 

We define the following function spaces associated to velocity, pressure and thin structure, respectively: 
\begin{align} 
& {\bf X}(\Omega): = H^1(\Omega)^d 
,\quad 
Q(\Omega): = L^2(\Omega) ,\quad 
\S(\Sigma): = H^1(\Sigma)^d . 
\nn 
\end{align} 
Correspondingly, we define the following bilinear forms:
\begin{align} 
a_f(\u, \bv): &= 2 \mu ( \DD(\u), \DD(\bv)) &&\mbox{for}\,\,\, \u,\bv\in {\bf X}(\Omega), \\
b(p, \bv): &= (p, \, \nabla \cdot \bv)  &&\mbox{for}\,\,\, \bv\in {\bf X}(\Omega)\,\,\mbox{and}\,\, p\in Q(\Omega) , \\
\nn 
a_s(\e, \w): &= ( {-\cal L}_s \e, \w )_\Sigma &&\mbox{for}\,\,\, \e,\w\in \S(\Sigma) .
\nn 
\end{align} 

We assume that ${\cal L}_s$ is a second-order differential operator on $\Sigma$ satisfying the following conditions: 
\begin{align}\label{Ls-second-order}
	&\|\mathcal{L}_s\w\|_{H^{k}(\Sigma)} \leq C\|\w\|_{H^{k+2}(\Sigma)} 
	&&\forall\, \w\in H^{k}(\Sigma)^d,\,\,\, \forall k\geq -1,\,\, k\in \mathbb{R}, \\
\label{as-ass}
	&a_s(\e,\w)=a_s(\w,\e)\quad\mbox{and}\quad a_s(\e,\e)\geq 0
	&& \forall \e\in H^1(\Sigma)^d, \\
	&\|\e\|_s+\|\e\|_\Sigma \sim \|\e\|_{H^1(\Sigma)} &&\mbox{for}\,\,\, \| \e \|_s: = \sqrt{a_s(\e, \e)} . \label{Korn-ineq}
\end{align}  
In addition, we denote $\| \u \|_f := \sqrt{( \DD(\u), \DD(\u))}$ and mention that the following norm equivalence holds (according to Korn's inequality):
\begin{equation*}
	\|\u\|_f+\|\u\|\sim \|\u\|_{H^1(\Omega)} . 
\end{equation*}
For the simplicity of notations, we denote by $\|\bv\|_{L^p X}$ the Bochner norm (or semi-norm) defined by 
\begin{align*}
	&\|\bv\|_{L^p X}:=\begin{cases}
		&\left(\int_{t=0}^{t=T}\|\bv(t,\cdot)\|^p_X dt\right)^{1/p}\quad 1\leq p<\infty\\
		&\sup_{t\in [0,T]}\|\bv(t,\cdot)\|_X \quad \qquad p=\infty , 
	\end{cases} 
	\qquad 
\end{align*}
where $\|\cdot\|_{X}$ is any norm or semi-norm in space, such as $\|\cdot\|_{f}$, $\|\cdot\|_{s}$ or $\|\cdot\|_{L^2(\Sigma)}$. 
The following conventional notations will be used: 
$ \| \cdot\|_{X}: = \| \cdot \|_{X(\Omega)}$,  
$ \| \cdot \|: =  \| \cdot \|_{L^2(\Omega)}$, $ \| \cdot \|_{\Sigma}: =  \| \cdot \|_{L^2(\Sigma)}$ and {$\|\cdot\|_{H_f}:=\|\cdot\|_f$, $\|\cdot\|_{H_s}:=\|\cdot\|_s$}.  

For smooth solutions of \eqref{f-e}--\eqref{bc}, one can verify that (via integration by parts) the following equations hold for all test functions $(\bv, q, \w) \in {\bf X} \times Q \times \S$ with $\bv|_\Sigma = \w$:
\begin{align} 
&\partial_t \e = \u  \qquad \mbox{on }  \Sigma, \nn\\
&\rho_f (\partial_t \u, \bv)  + a_f(\u, \bv) - b(p, \bv) + b(q, \u) + \rho_s \epsilon_s ( \partial_{tt} \e, \, \w )_\Sigma 
+ a_s(\e, \w) = 0 .
\label{weak} 
\end{align}

\subsection{Regularity assumptions}\label{section:regularity}

To establish the optimal error estimates for the finite element solutions to the thin-structure interaction problem, we need to use the following regularity results. 

\begin{itemize}
	\item 
	We assume that the domain $\Omega$ is smooth so that the the solution $(\u, p, \e)$ of the fluid-structure interaction problem \refe{f-e}--\refe{bc} is sufficiently smooth. 
	
	\item
	The weak solution $({\bb \omega},\lambda)\in  H^1(\Omega)^d\times L^2(\Omega)$ of the Stokes equations
	\begin{align*}
		-\nabla\cdot\sig({\bb \omega},\lambda)+{\bb \omega}&={\bf f} \\
		 \nabla\cdot{\bb \omega}&=0 
	\end{align*}
	has the following regularity estimates:
	\begin{align}
		&\|{\bb \omega}\|_{H^{k+3/2}}+\|\lambda\|_{H^{k+1/2}}\leq C\|{\bf f}\|_{H^{k-1/2}}+\|\sig({\bb \omega},\lambda)\cdot\n\|_{H^{k}(\Sigma)} && \mbox{for}\,\,\, k\geq -1/2, &&k\in \mathbb{R} , \label{reg-ass1}\\
		&\|{\bb \omega}\|_{H^{k+1/2}}+{\|\lambda-\bar{\lambda}\|_{H^{k-1/2}}}\leq C\|{\bf f}\|_{H^{k-3/2}}+\|{\bb \omega}\|_{H^{k}(\Sigma)} && \mbox{for}\,\,\, k\geq 1/2, &&k\in \mathbb{R} , \label{reg-ass2}
	\end{align}
	where $\bar{\lambda}:=\frac{1}{|\Omega|}\int_\Omega \lambda$ is the mean value of $\lambda$ over $\Omega$. 
	The estimates in \eqref{reg-ass1} and \eqref{reg-ass2} correspond to the Neumann and Dirichlet boundary conditions, respectively; see \cite[Theorem IV.6.1]{Galdi} for a proof of \eqref{reg-ass2} in smooth domains, with a similar approach as in \cite[Chapter IV]{Galdi} one can prove \eqref{reg-ass1}. We also refer to \cite[Theorem 4.15]{Schwab} for a proof of \eqref{reg-ass1} in the case of polygonal domain. 
\end{itemize}
\begin{itemize}
	\item 
	We assume that operator $\mathcal{L}_s$ possesses the following elliptic regularity: The weak solution ${\bb \xi}\in H^1(\Sigma)^d$ of the equation (in the weak formulation) 
	\begin{align*}
		a_s({\bb \xi},\w)+({\bb \xi},\w)_\Sigma=( {\bf g}, \w )_\Sigma\quad \forall \w\in H^1(\Sigma)^d , 
	\end{align*}
	has the following regularity estimate:
	\begin{align}\label{reg-ass3}
		\|{\bb \xi}\|_{H^{2+k}(\Sigma)}\leq C\|{\bf g}\|_{H^k(\Sigma)}\quad \mbox{for}\,\,\, k\geq -1,\,\, k\in \mathbb{R}. 
	\end{align}
\end{itemize}

\subsection{Assumptions on the finite element spaces}\label{FE-space}
Let  $\mathcal{T }_h$ denote a quasi-uniform partition on 
$\Omega$ with $\overline{\Omega} = \bigcup_{K \in \mathcal{T}_h} K$. Each $K$ is a curvilinear polyhedron/polygon with ${\rm diam}(K)\leq h$. 
All of the boundary faces on 
	$\Sigma$ consist of a partition $\mathcal{T}_h (\Sigma)$, $\Sigma =
	\bigcup_{D \in \mathcal{T}_h (\Sigma)} D$,  
	all of the boundary faces on ${\Sigma_l} $ or $\Sigma_r$
	consist of a partition of $\Sigma_l$ or $\Sigma_r$,  respectively, and these two partitions coincide after shifting $L$ in $z$-direction.	
To approximate the weak form \eqref{weak} by finite element method, we assume that there are finite element spaces $(\X^r_h, \S^r_h, Q^{r-1}_h)$ on $\mathcal{T }_h$ (where $r\geq 1$) with the following properties. 
\vskip0.1in 

\begin{itemize}
	\item {\bf (A1)} $\X^r_h \subseteq \X$, $\S^r_h \subseteq \S$ and $\mathbb{R}\subseteq Q^{r-1}_h \subseteq Q$, with $\S^r_h = \{ \bv_h |_{\Sigma} : \bv_h \in \X^r_h  \}$.
	
	
	
	\item {\bf (A2)} For $\X_h^r$ and $Q_h^{r-1}$, the following local inverse estimate holds on each $K\in \mathcal{T}_h$ for $ 0\leq l\leq k,1\leq p,q\leq \infty$:
	\begin{align}\label{inverse-Xh-Qh-local}
		&\| \bv_h \|_{W^{k, p} (K)} \leq C h^{- (k - l) + (d/ p - d/ q)} \| \bv_h
			\|_{W^{l, q} (K)} \quad \forall \bv_h \in \X^r_h \mbox{ or } 
			Q^{r-1}_h,
	\end{align}
	For $\S^r_h$, the following global inverse estimate holds: 
	\begin{equation}\label{inverse-Sh}
		\|\w_h \|_{H^s(\Sigma)} \leq C h^{k-s} \| \w_h \|_{H^{k}(\Sigma)} \quad \forall
		\w_h \in \S^r_h;\;\forall\, k,s\in \mathbb{R}\text{ with $0\leq k\leq s\leq 1$} .
	\end{equation}
	
	\item {\bf (A3)} There are interpolation/projection operators $I^X_h : \X \rightarrow \X^r_h$ and $I_h^Q : Q \rightarrow Q^{r-1}_h$ which have the following local $L^p$ approximation properties on each $K\in \mathcal{T}_h$, for all $1\leq p\leq \infty$:
	\begin{subequations}\label{interpolation-error-Xh-Qh}
		\begin{align}
			&\| I_h^X \u - \u \|_{L^p(K)}+h\| I_h^X \u - \u \|_{W^{1,p}(K)} \leq C h^{k + 1} \| \u \|_{W^{k + 1,p}(\Delta_K)}  && \forall\, 0\leq k \leq r , \\
			&\| I_h^Q p - p \|_{L^p(K)} \leq C h^{k + 1} \| p \|_{W^{k +1,p} (\Delta_K)} && \forall\, 0 \leq k \leq r - 1,
		\end{align}
	\end{subequations}
	where $\Delta_K$ is the macro element including all the elements which have a common vertex with $K$. And there is an interpolation/projection operator $I_h^S : \S \rightarrow \S^r_h$ satisfying $I_h^X \u |_{\Sigma} = I_h^S (\u |_{\Sigma})$ for all $\u\in \X$ with $\u|_\Sigma\in \S$. Moreover, we require the following optimal order error estimate
	\begin{equation}
		\| I_h^S \w - \w \|_{\Sigma} + h \|I_h^S \w - \w\|_{H^1(\Sigma)} \leq C h^{k + 1} \| \w \|_{H^{k + 1}_h (\Sigma)} \quad \forall 0 \leq
		k \leq r,
	\end{equation} 
	where $\|\cdot\|_{H^{k+1}_h(\Sigma)}$ is the piecewise $H^{k+1}$-norm associated with partition $\mathcal{T}_h(\Sigma)$. We will abuse notation and use $I_h$ to denote one of the operators $I_h^X$, $ I_h^S$ and $ I_h^Q$ when there is no confusion.
	\item {\bf (A4)} Let $\mathring \X^r_h := \{ \bv_h \in \X^r_h : \bv_h |_{\Sigma} = 0 \}$ and $Q_{h,0}^{r-1}:=\{q_h\in Q_h^{r-1}:q_h\in L^2_0(\Omega)\}$. The following inf-sup condition holds:
	\begin{equation}\label{inf-sup-condition0}
		\| q_h \| {\leq C \sup_{0 \neq \bv_h \in \mathring\X^r_h}}  \frac{({\rm div}\, \bv_h,
			q_h)}{\| \bv_h \|_{H^1 }} \quad
		\forall q_h \in Q^{r-1}_{h,0}  
	\end{equation} 
\end{itemize}

\begin{remark}\label{remark-FE-assm}\upshape 
Examples of finite element spaces which satisfy Assumptions (A1)--(A4) include the Taylor--Hood finite element space with $I_h^X$, $I_h^Q$ and $I_h^S$ being the Scott--Zhang interpolation operators onto $\X^r_h$, $Q_h^{r-1}$ and $\S^r_h$ respectively. We refer to \cite[Section 4.8]{brenner08} and the references therein for the details on construction and properties of Scott-Zhang interpolation, and refer to \cite[Section 8.8]{boffi-13} for a proof of \eqref{inf-sup-condition0} for the Taylor-Hood finite element spaces. The following properties are consequences of the assumptions on the finite element spaces in {\bf(A1)}--{\bf(A4)}.
\begin{enumerate}
	\item From {\bf (A2)} 
	and {\bf (A3)} we can derive the following estimate for $\bv_h\in \X^r_h$:
	\begin{align*}
		\|\DD(\bv_h)\n\|_{\Sigma}&=\Big(\sum_{D\in \mathcal{T}_h(\Sigma)}\|\DD(\bv_h)\n\|^2_{L^2(D)}\Big)^{1/2}\\
		&\leq C\Big(\sum_{D\in \mathcal{T}_h(\Sigma)}h^{d-1}\|\bv_h\|^2_{W^{1,\infty}(K)}\Big)^{1/2}\quad\mbox{($K\in \mathcal{T}_h$ contains $D$)}\\
		&\leq C\Big(\sum_{D\in \mathcal{T}_h(\Sigma)}h^{-1}\|\bv_h\|^2_{H^1(K)}\Big)^{1/2}\leq Ch^{-1/2}\|\bv_h\|_{H^1} . 
	\end{align*}
Therefore, we can obtain the following inverse estimate for the boundary term $\sig(\bv_h,q_h)\n$:
	\begin{equation}\label{inverse-sig}
		\|\sig(\bv_h,q_h)\n\|_\Sigma\leq Ch^{-1/2}(\|\bv_h\|_{H^1}+\|q_h\|) . 
	\end{equation}
	\item From {\bf (A3)} and {\bf (A4)} we can see that when $r\geq 2$, the mixed finite element space $(\X_h^r,Q_h^{r-1})$ can be realized by the $(r,r-1)$ Taylor-Hood finite element space. When $r=1$, $(\X_h^1,Q_h^{0})$ can be realized by the MINI element space. 
	
	\item From inf-sup condition \eqref{inf-sup-condition0}, we can deduce the following alternative version of inf-sup condition (involving $H^1(\Sigma)$-norm in the denominator)
	\begin{equation}\label{inf-sup-condition}
		\| q_h \| {\leq C \sup_{0 \neq \bv_h \in \X^r_h}}  \frac{({\rm div} \bv_h,
			q_h)}{\| \bv_h \|_{H^1 }+\|\bv\|_{H^1(\Sigma)}} \quad
		\forall q_h \in Q^{r-1}_{h}. 
	\end{equation}
	An inf-sup condition similar to \eqref{inf-sup-condition} was proved in \cite[Lemma 2]{xujinchao-15}, though thick structure problem is considered there. For the reader's convenience, we present a proof of \eqref{inf-sup-condition} in the Appendix C.
	\item For each $\w_h\in \S^r_h$, we denote by $E_h\w_h\in \X_h^r$  an extension such that $E_h\w_h: = I_h^X\bv$, where $\bv \in H^1(\Omega)^d$ is the extension of $\w_h$ by trace theorem, satisfying $\|\bv\|_{H^1}\leq C\|\w_h\|_{H^{1/2}(\Sigma)}$ and  $\bv |_\Sigma  = \w_h$. Combining \eqref{interpolation-error-Xh-Qh} with \eqref{inverse-Sh}  we see that 
	\begin{equation}\label{Eh-estimate}
		\|E_h\w_h\|_{H^1}\leq Ch^{-1/2}\|\w_h\|_\Sigma.
	\end{equation}
	\item  Combining \eqref{interpolation-error-Xh-Qh} with \eqref{inverse-sig} we have for any $\u_h\in \X_h^r$, $p_h\in Q_h^{r-1}$
	\begin{align}\label{sig-error-inverse}
		&\|\sig(\u-\u_h,p-p_h)\n\|_{\Sigma} \nn\\
		&\leq \|\sig(\u-I_h\u,p-I_hp)\n\|_{\Sigma}+\|\sig(I_h\u-\u_h,I_hp-p_h)\n\|_{\Sigma}\nn\\
		&\leq C(\|\u-I_h\u\|_{W^{1,\infty}}+\|p-I_hp\|_{L^{\infty}})+\|\sig(I_h\u-\u_h,I_hp-p_h)\n\|_{\Sigma}\nn\\
		&\leq Ch^r+Ch^{-1/2}(\|I_h\u-\u_h\|_{H^1}+\|I_hp-p_h\|) \nn\\
		&\leq Ch^{r-1/2}+Ch^{-1/2}(\|\u-\u_h\|_{H^1}+\|p-p_h\|) , 
	\end{align}
	where we have used \eqref{interpolation-error-Xh-Qh} with $p=\infty$  and \eqref{inverse-sig} in the second to last inequality.
\end{enumerate}
\end{remark}


\subsection{A new kinematically coupled scheme and main theoretical results}

Let ${\left \{ t_{n} \right \}}_{n=0}^{N}$ be a uniform partition of the time interval $[0,T]$ with stepsize $\tau = T/N$. 
For a sequence of functions $\{\u^{n}\}_{n=0}^{N}$ 
we denote
\begin{eqnarray}
{D_{\tau}} \u^{n} = \frac{\u^{n}-\u^{n-1}}{\tau}, \quad
\textrm{for $n=1$, $2$, $\ldots$, $N$}.
\nn
\end{eqnarray}
With the above notations, we present a fully discrete kinematically coupled algorithm below.  
\vskip0.1in 

{\it Step 1}:  For given $\u_h^{n-1}, p_h^{n-1},  \e_h^{n-1}$, 
find $\s_h^n \in \S_h^{r}$ such that 
\begin{align}
&  \rho_s \epsilon_s \left  ( \frac{\s_h^n-\u_h^{n-1}}{\tau}, \, \w_h  
\right )_\Sigma 
+ a_s(\e_h^{n}, \, \w_h) 
= - ( \sig^{n-1}_h \cdot {\bf n}, \, \w_h )_\Sigma, \qquad 
\forall \w_h \in \S_h^{r}
\label{s1} 
\\ 
& \e_h^n = \e_h^{n-1} + \tau \s_h^n \, .
\nn 
\end{align} 

{\it Step 2}: Then find $(\u_h^n, p_h^n) \in \X_h^r\times Q_h^{r-1}$
satisfying 
\begin{align}
& \rho_f (D_{\tau} \u_h^n, \, \bv_h)   + a_f(\u_h^n, \, \bv_h) - b(p_h^n, \, \bv_h) 
+ b(q_h, \, \u_h^n) -  ( \sig^{n}_h \cdot {\bf n}, \, \bv_h )_\Sigma  
\label{s2} \\ 
& + \rho_s \epsilon_s \left ( \frac{\u_h^n - \s_h^n}{\tau}, \, \bv_h + \frac{\tau}{\rho_s \epsilon_s} \sig(\bv_h, q_h)\cdot \n  \right  )_\Sigma \nn\\
&+ \left ( (\sig^n_h-\sig^{n-1}_h) \cdot {\bf n}, \, \bv_h + \frac{\tau(1+\beta)}{\rho_s \epsilon_s} \sig(\bv_h, q_h)\cdot \n  \right )_\Sigma  = 0 
\nn 
\end{align} 
for all $(\bv_h, q_h) \in \X_h^r\times Q_h^{r-1}$, 
where $\sig_h^n = \sig(\u_h^n, p_h^n)$ and $\beta \geq 0$ denotes a stabilization parameter. 

{\it Initial values}: 
Since $\sig^{n-1}_h$ depends on both $\u_h^{n-1}$ and $p_h^{n-1}$, the numerical scheme in \eqref{s1}--\eqref{s2} requires the initial value $(\u_h^{0},p_h^0,\e_h^0)$ to be given. We simply assume that the initial value $(\u_h^{0},p_h^0,\e_h^0)$ are given sufficiently accurately, satisfying the following conditions:
\begin{align}\label{initial-value}
	\begin{aligned}
		\| \u_h^{0} - R_h \u^0 \| + \| \u_h^{0} - R_h \u^0 \|_\Sigma + \| \e_h^{0} - R_h \e^0 \|_{H^1(\Sigma)} &\le Ch^{r+1} , \\
		\| p_h^{0} - R_hp^0 \|_\Sigma &\le C,
	\end{aligned}
\end{align} 
where $(R_h\u^0, R_h p^0, R_h\e^0)$ satisfies a couple non-stationary Ritz projection defined in Section \ref{section:Ritz-def}. 

\begin{remark}\upshape 
Kinematically coupled schemes were firstly proposed in  \cite{BCGTQ, GGCC-2009, BM-2016SINUM} with the following time discretization: 
Find $(\s^n, \e^n)$ such that 
\begin{align} 
&  \rho_s \epsilon_s \frac{\s^n-\u^{n-1}}{\tau}
- {\cal L}_s (\e^{n}) 
= - \sig^{n-1} \cdot {\bf n} &&\mbox{on}\,\,\,\Sigma
\label{s-1} 
\\ 
& \e^n = \e^{n-1} + \tau \s^n  &&\mbox{on}\,\,\,\Sigma
\nn 
\end{align} 
and then find
$(\u^n, p^n)$
satisfying 
\begin{align} 
& \rho_f D_{\tau} \u^n  +  \nabla \cdot \sig^n = 0  \quad\mbox{and}\quad \nabla \cdot \u^n =0 &&\mbox{in}\,\,\,\Omega,
\label{s-2} \\ 
& \rho_s \epsilon_s  \frac{\u^n - \s^n}{\tau}  
+ (\sig^n-\sig^{n-1})\cdot {\bf n}  = 0 &&\mbox{on}\,\,\,\Sigma .
\nn 
\end{align} 
The extension to full discretization was considered by several authors \cite{BM-2016SINUM, Bukac2018SINUM}, while the analysis for full discretization is incomplete and  the energy stability is proved only for time-discrete schemes.
\end{remark}

\begin{remark}\upshape 
Our scheme in \eqref{s1}--\eqref{s2} is designed with two new ingredients. First, we have added two stabilization terms 
$$
\rho_s \epsilon_s \left ( \frac{\u_h^n - \s_h^n}{\tau}, \,  \frac{\tau}{\rho_s \epsilon_s} \sig(\bv_h, q_h)\cdot \n  \right  )_\Sigma 
\quad\mbox{and}\quad 
\left ( (\sig^n_h-\sig^{n-1}_h) \cdot {\bf n}, \,  \frac{\tau(1+\beta)}{\rho_s \epsilon_s} \sig(\bv_h, q_h)\cdot \n  \right )_\Sigma ,
$$
which guarantee unconditional energy stability of the scheme in \eqref{s1}--\eqref{s2}. Otherwise the unconditional energy stability cannot be proved in the fully discrete finite element setting. 
Second, we have introduced an additional parameter $\beta>0$ to the scheme, and this additional parameter allows us to prove optimal-order convergence in the $L^2$ norm (especially optimal order in space). 
More specifically, this parameter $\beta>0$ leads to the following term in the $E_1$ of \eqref{E1}: 
$$\beta_0\frac{\rho_s \epsilon_s}{2\tau}  \| \s^n_h-\u_h^{n} \|_\Sigma^2 \quad\mbox{with}\,\,\, \beta_0 = 1-(\sqrt{4+\beta^2}-\beta)/2,$$
which is used to absorb other undesired terms on the right-hand side of the inequalities in our error estimation. 
Therefore, the optimal-order $L^2$ error estimate does benefits from our scheme (with the parameter $\beta>0$). 
\end{remark} 

\begin{remark}\upshape 
For the Taylor--Hood finite element spaces, the conditions in \eqref{initial-value} on the initial values can be satisfied if one chooses $\u_h^0$ and $p_h^0$ to be the Lagrange interpolations of $\u^0$ and $p^0$, respectively, and chooses $\e_h^0=R_{sh}\e(0)$, where $R_{sh}\e(0)$ is defined in Section \ref{sec-ritz}; see Definition \ref{def-Rshe0} and estimate \eqref{super-approx}. 
\end{remark} 

The main theoretical results of this article are the following two theorems. 

\begin{theorem} \label{main-1}
Under the assumptions in Section \ref{FE-space} (on the finite element spaces), the finite element system in \refe{s1}--\refe{s2} is uniquely solvable, and the following inequality holds: 
\begin{align} 
E_0(\u_h^n, p_h^n, \e_h^n) + \sum_{m=1}^n \tau E_1(\u_h^m, \s_h^m, \e_h^m) 
\le E_0(\u_h^0, p_h^0, \e_h^0), \quad n=1,2,..., N ,
\label{stability} 
\end{align} 
where 
\begin{align} 
 E_0(\u_h^n, p_h^n, \e_h^n) 
& = 
\frac{\rho_f}{2}  \| \u_h^n \|^2 
+ \frac{1}{2}  \| \e_h^n \|_s^2 
+ \frac{\tau^2(1+\beta)}{2\rho_s \epsilon_s} \| \sig^n_h \cdot \n\|_\Sigma^2+\frac{\rho_s\epsilon_s}{2}\|\u_h^n\|_\Sigma^2 , 
\label{E0} \\ 
 E_1(\u_h^n, \s_h^n, \e_h^n) 
& =  2 \mu\| \u_h^n\|_f^2 
+ \frac{\rho_f }{2\tau}  \|  \u_h^n - \u_h^{n-1}\|^2 
+ \frac{\rho_s \epsilon_s}{2\tau} \| \s_h^n - \u_h^{n-1} \|_\Sigma^2 
+ \frac{\rho_s \epsilon_s\beta_0}{2\tau}  \| \s^n_h-\u_h^{n} \|_\Sigma^2 
\nn \\ 
&\quad\, +  \frac{\tau \beta_0}{2\rho_s \epsilon_s} 
 \| (\sig_h^n - \sig_h^{n-1}) \cdot \n \|_\Sigma^2 
 +
 \frac{\tau}{2}  \| D_\tau \e^n_h \|_s^2  ,
\label{E1} 
\end{align}  
with $\beta_0 = 1-(\sqrt{4+\beta^2}-\beta)/2$ and $\beta \ge 0$. 
\end{theorem}

\begin{theorem} \label{main-2} 
For finite elements of degree $r\ge 2$, under the assumptions in Sections \ref{section:regularity}--\ref{FE-space} (on the regularity of solutions and finite element spaces), 
there exist positive constants $\tau_0$ and $h_0$ such that, for sufficiently small stepsize and mesh size $\tau\le\tau_0$ and $h\le h_0$, the finite element solutions given by \refe{s1}--\refe{s2} with initial values satisfying \eqref{initial-value} and $\beta>0$ has the following error bound: 
\begin{align} 
\max_{1\le n\le N} 
\big( \| \u(t_n, \cdot) - \u_h^n \|  + \| \e(t_n, \cdot) - \e_h^n \|_\Sigma  + \|\u(t_n,\cdot)-\u_h^n\|_{\Sigma} \big) 
& \le C ( \tau + h^{r+1}), 
\label{main-error} 
\end{align} 
where $C$ is some positive constant independent of $n$, $h$ and $\tau$.
\end{theorem} 

The proofs of Theorem \ref{main-1} and Theorem \ref{main-2} are presented in the next section. 

%
%
\section{Analysis of the proposed algorithm} 
\setcounter{equation}{0}

This section is devoted to the proof of Theorems \ref{main-1} and \ref{main-2}. For the simplicity of notation, we denote by $C$ a generic positive constant, which is independent of $n$, $h$ and $\tau$ but may depend on the physical parameters $\rho_s,\epsilon, \mu, \rho_f$ and the exact solution $(\u,p,\e)$. In addition, we denote by $a\lesssim b$ the statement ``$|a|\le b$''.

%
%
\subsection{Proof of Theorem \ref{main-1}} 
We rewrite \refe{s2} into 
\begin{align} 
& \rho_f  (D_{\tau} \u_h^n, \, \bv_h)   + a_f(\u_h^n, \, \bv_h) -b(p_h^n, \, \bv_h^n) 
+ b(q_h, \, \u_h^n) 
 + \rho_s \epsilon_s \left ( \frac{\u_h^n - \s_h^n}{\tau}, \, \bv_h \right  )_\Sigma 
 \label{s2-1} \\ 
& = ( \sig^{n-1}_h \cdot {\bf n}, \, \bv_h )_\Sigma  
-  ( \u_h^n - \s_h^n, \, \sig(\bv_h, q_h)\cdot \n  )_\Sigma  
-  \frac{\tau(1+\beta)}{\rho_s \epsilon_s}  ( (\sig^n_h-\sig^{n-1}_h) \cdot {\bf n}, \,  \sig(\bv_h, q_h)\cdot \n  )_\Sigma  \, . 
\nn
\end{align} 
Taking $\bv_h = \u_h^n, q_h = p_h^n$ in (\ref{s2-1}) and $\w_h=\s_h^n =D_\tau \e_h^n$ in (\ref{s1}), respectively, gives the following relations:
 \begin{align} 
 & \frac{\rho_f}{2\tau } \left ( \| \u_h^n \|^2 - \| \u_h^{n-1} \|^2  + \| \u_h^n - \u_h^{n-1} \|^2 
\right ) +2 \mu \| \u_h^n  \|_f^2 
+  \rho_s \epsilon_s \left ( \frac{\u_h^n - \s_h^n}{\tau}, \, \u_h^n 
\right  )_\Sigma 
\nn \\ 
& = ( \sig^{n-1}_h \cdot {\bf n}, \, \u^n_h )_\Sigma  
-  ( \u_h^n - \s_h^n, \, \sig_h^n\cdot \n  )_\Sigma  
-  \frac{\tau(1+\beta)}{\rho_s \epsilon_s}  ( (\sig^n_h-\sig^{n-1}_h) \cdot {\bf n}, \,  \sig_h^n \cdot \n  )_\Sigma  
\nn 
\end{align} 
and 
\begin{align} 
& \frac{1}{2\tau} \left ( a_s( \e_h^n, \e_h^n) - a_s( \e_h^{n-1}, \e_h^{n-1}) 
+ \tau^2 a_s( \s_h^n, \s_h^n) \right ) 
+ \rho_s \epsilon_s \left ( \frac{\s_h^n-\u_h^{n-1}}{\tau}, \, \s_h^n  
\right )_\Sigma 
= - ( \sig^{n-1}_h \cdot {\bf n}, \, \s_h^n )_\Sigma \, .
\nn 
\end{align} 
By summing up the last two equations, we have  
\begin{align} 
& \frac{\rho_f}{2} \left ( \| \u_h^n \|^2 - \| \u_h^{n-1} \|^2  + \| \u_h^n - \u_h^{n-1} \|^2 
\right ) +2 \mu \tau \| \u_h^n \|_f^2 + 
\frac{\rho_s \epsilon_s}{2} \left ( \| \s^n_h-\u_h^{n-1} \|_\Sigma^2 + \| \u_h^n - \s_h^n \|_\Sigma^2 \right ) 
\nn \\ 
& + \frac{1}{2} \left ( 
 a_s(\e_h^n, \e_h^n) - a_s(\e_h^{n-1}, \e_h^{n-1}) + \tau^2 a_s( \s_h^n, 
\s_h^n ) \right )+\frac{\rho_s\epsilon_s}{2}\left(\|\u_h^n\|^2_\Sigma-\|\u_h^{n-1}\|^2_\Sigma\right)
 \nn \\ 
 & = \tau ((\sig_h^{n-1}-\sig_h^n) \cdot \n, \u_h^n - \s_h^n)_\Sigma
  - \frac{\tau^2(1+\beta)}{\rho_s \epsilon_s} ((\sig^n_h-\sig^{n-1}_h) \cdot {\bf n}, \,  \sigma_h^n\cdot \n  )_{\Sigma}   
\nn \\ 
& \le  \frac{\tau^2(1+\beta-\beta_0)}{2\rho_s \epsilon_s} \| (\sig_h^n - \sig_h^{n-1}) \cdot \n \|_\Sigma^2 
+ \frac{\rho_s \epsilon_s}{2(1+\beta-\beta_0)} \| \u_h^n - \s_h^n \|_\Sigma^2  
\nn \\ 
& \quad - \frac{\tau^2(1+\beta)}{2\rho_s \epsilon_s} \left ( \| \sig^n_h \cdot \n\|_\Sigma^2 - \| \sig_h^{n-1} \cdot \n\|_\Sigma^2 
+ \| (\sig_h^n - \sig_h^{n-1})\cdot \n \|_\Sigma^2 \right ) 
\nn \\ 
& \le 
\frac{\rho_s \epsilon_s(1-\beta_0)}{2} \| \u_h^n - \s_h^n \|_\Sigma^2  
- \frac{\tau^2(1+\beta)}{2\rho_s \epsilon_s} \left ( \| \sig^n_h \cdot \n\|_\Sigma^2 - \| \sig_h^{n-1} \cdot \n\|_\Sigma^2  \right ) 
-  \frac{\tau^2\beta_0}{2\rho_s \epsilon_s} \| (\sig^n_h - \sig_h^{n-1}) \cdot \n\|_\Sigma^2   , 
\nn 
 \end{align} 
which leads to the following energy inequality: 
\begin{align} 
&  E_0(\u_h^n, p_h^n, \e_h^n) - E_0(\u_h^{n-1}, p_h^{n-1}, \e_h^{n-1}) 
+ E_1(\u_h^n, p_h^n, \e_h^n) \tau \le 0 \, . 
\label{stability-2} 
\end{align} 
This implies \refe{stability} and completes the proof of Theorem \ref{main-1}. 
\hfill \endproof 
\vskip0.1in

%
%
\subsection{A coupled non-stationary Ritz projection} \label{section:Ritz-def}
To establish $L^2$-norm optimal error estimate as given in Theorem \ref{main-2}, we need to introduce a new coupled Ritz projection. Since the thin fluid-structure model is governed by the Stokes type equation for fluid coupled with a hyperbolic type equation for solid, the coupled projection, which is non-stationary and much more complicated than  the standard Ritz projections, plays a key role in proving the optimal-order convergence of finite element solutions to the fluid-structure models. 

 
\begin{definition}[{{\bf Coupled non-stationary Ritz projection}}]\label{ritz-def} 
Let $(\u, p, \e) \in {\bf X} \times Q \times \S$ be a triple of functions smoothly depending on $t\in[0,T]$ and satisfying the condition $\u |_\Sigma = \partial_t \e$. 
For a given initial value $R_h \e(0)$, the coupled Stokes--Ritz projection $R_h(\u, p, \e): =(R_h\u, R_h p, R_h\e) \in \X_h^r\times Q_h^{r-1} \times \S_h^r$ is defined as a triple of functions satisfying $(R_h \u)|_\Sigma = \partial_t R_h \e$  and the following weak formulation for every $t\in[0,T]$: 
\begin{align}\label{evolution-ritz}
a_f(\u - R_h \u, \bv_h) & - b( p-R_hp, \bv_h) + b(q_h, \u - R_h \u)+(\u-R_h\u,\bv_h)
\nn \\ 
& + a_s(\e - R_h \e, \bv_h) + ( \e - R_h \e, \bv_h )_\Sigma = 0, 
 \qquad \forall (\bv_h, q_h) \in  \X_h^r\times Q_h^{r-1} . 
\end{align} 
\end{definition}
\begin{remark}\upshape
	Given an initial value $R_h\e(0)$, there exists a unique solution $(R_h\u,R_hp,R_h\e_h)$ for the finite element semi-discrete problem \eqref{evolution-ritz}. To see this, we firstly introduce a linear operator 
	$\mathcal{S}_h:(\X_h^{r})^*\times (Q_h^{r-1})^*\to \X_h^r\times Q_h^{r-1}$, where $(\X_h^{r})^*$ and $(Q_h^{r-1})^*$ denote the dual space of $\X_h^r$ and $Q_h^{r-1}$, respectively.  For a given $(\phi,\ell)\in (\X_h^{r})^*\times (Q_h^{r-1})^*$, denote by $(\u_h,p_h)\in \X_h^r\times Q_h^{r-1}$ the solution of  the following Neumann-type discrete Stokes equation 
	\begin{align*}
		&a_f(\u_h,\bv_h)-b(p_h,\bv_h)+(\u_h,\bv_h)=\phi(\bv_h)\quad \forall \bv_h\in \X_h^r,\\
		&b(q_h,\u_h)=\ell(q_h)\quad \forall q_h\in Q_h^{r-1},
	\end{align*}
	and define $\mathcal{S}_h(\phi,\ell)=(\mathcal{S}_h^v(\phi,\ell),\mathcal{S}_h^p(\phi,\ell)):=(\u_h,p_h)$. 
	The well-posedness of the above equation follows the inf-sup condition \eqref{inf-sup-condition}. 

Next, we denote 	
	\begin{align*}
	\phi_{(u,p,\eta)}(\bv_h)&:=a_f(\u,\bv_h)-b(p,\bv_h)+(\u,\bv_h)+a_s(\e,\bv_h)+(\e,\bv_h)_\Sigma,\\
	\phi_{R_h\eta}(\bv_h)&:=a_s(R_h\e,\bv_h)+(R_h\e,\bv_h)_\Sigma,\\
	\ell_{u}(q_h)&:= b(q_h,\u), \quad \phi_{R_h\eta}(\bv_h):=a_s(R_h\e,\bv_h)+(R_h\e,\bv_h)_\Sigma. 
	\end{align*}
Then $(R_h\u,R_hp,R_h\e)$ is a solution to \eqref{evolution-ritz} if and only if the following equations are satisfied 
	\begin{subequations}
	\begin{align}
	\partial_t R_h\e&= \mathcal{S}_h^v(\phi_{(u,p,\eta)}-\phi_{R_h\eta},\ell_{u})|_\Sigma\label{eta-evolution-ritz}\\
	R_h\u&=\mathcal{S}_h^v(\phi_{(u,p,\eta)}-\phi_{R_h\eta},\ell_{u}) \label{u-evolution-ritz},\quad R_hp= \mathcal{S}_h^p(\phi_{(u,p,\eta)}-\phi_{R_h\eta},\ell_{u})
	\end{align}	
	\end{subequations}
	Therefore, the uniqueness and existence of solution to \eqref{evolution-ritz} follows the uniqueness and existence of solution to \eqref{eta-evolution-ritz}. Since $\mathcal{S}_h^v$ is a linear operator on $(\X_h^{r})^*\times (Q_h^{r-1})^*$ and $\phi_{R_h\eta}$ is linear for $R_h\e$, \eqref{eta-evolution-ritz} is an in-homogeneous linear ordinary differential equation for  $R_h\e$ and thus admits a unique solution for given any initial value $R_h\e(0)$.
\end{remark}


In order to guarantee that the coupled non-stationary Ritz projection $R_h$ possesses optimal-order approximation properties, we need to define $R_h\e(0)$ in a rather technical way. Therefore, we present error estimates for this coupled non-stationary Ritz projection in Theorem \ref{main-3} and postpone the definition of $R_h\e(0)$ and the proof of Theorem \ref{main-3} to Section \ref{sec-ritz}. 

\begin{theorem}[{{\bf Error estimates for the coupled non-stationary Ritz projection}}]
\label{main-3} For sufficiently smooth functions $(\u, p, \e)$ satisfying $\u|_{\Sigma}=\partial_t\e$, there exists $\w_h \in \S_h^r$ such that when $R_h\e(0) = \w_h$, the following estimates hold uniformly for $t\in[0,T]$:  
\begin{align}
\label{SR-error-1} 
\max_{t\in[0,T]}
\big(  \| \e - R_h \e \|_{\Sigma} + \| \u - R_h \u \|+\| \u - R_h \u \|_{\Sigma} + h\|p-R_hp\|  \big) &\le C h^{r+1} 
\\ 
\label{SR-error-2} 
 \max_{t\in[0,T]} \big( \| \partial_t (\u - R_h \u) \|_{H^1}+\| \partial_t (\u - R_h \u) \|_{H^1}+\|\partial_t(p-R_hp)\| \big)
&\le C h^r
\\
 \| \partial_t (\u - R_h \u) \|_{L^2L^2(\Sigma)}+\|\partial_t(\u-R_h\u)\|_{L^2L^2} & \le C h^{r+1} \, . 
 \label{SR-error-3} 
\end{align} 
\end{theorem}

\subsection{Proof of Theorem \ref{main-2}} 

For the solution $(\u, p, \e)$ of problem \eqref{f-e}--\eqref{bc}, we define the following notations: 
\begin{align} \label{def-sn-un-en}
&\u^n = \u(t_{n}, \cdot) , \quad \e^n = \e(t_n, \cdot) ,\quad p^n = p(t_{n}, \cdot). 
\end{align} 
For the analysis of the kinematically coupled scheme, we introduce $\s^n\in H^1(\Sigma)$ and $R_h\s^n\in \S_h^r$ by 
$$
\s^n = \partial_t \e(t_n, \cdot) = \u(t_{n}, \cdot) \quad\mbox{and}\quad 
R_h\s^n:=(R_h\u)(t_n)=\partial_tR_h\e(t_n) 
\quad\mbox{on}\,\,\,\Sigma, 
$$ 
which satisfy the following estimates according to the estimates in Theorem \ref{main-3}: 
\begin{align} \label{SR-error-4} 
& \| \s^n - R_h \s^n \|_{\Sigma} \le C h^{r+1} \, . 
\end{align} 
By Taylor's expansion, we have $\e^n =  \e^{n-1} + \tau\s^{n}+\mathcal{T}_0^n$, with a truncation error $\mathcal{T}_0^n$ which has the following bound:
\begin{align}\label{Estimate-T0n}
\|\mathcal{T}_0^n\|_{H^1(\Sigma)}\leq C\tau^2\quad \forall\, n\geq 1\, . 
\end{align}
By \eqref{f-e}--\eqref{bc}, we can see that the sequence $( \u^n, p^n,  \e^n, \s^n)$ satisfies the following weak formulations 
\begin{align} 
&  \rho_s \epsilon_s \left ( \frac{ \s^n- \u^{n-1}}{\tau}, \, \w_h \right  )_\Sigma 
+ a_s(\e^{n}, \, \w_h) 
+ (  \sig^{n-1} \cdot {\bf n}, \, \w_h )_\Sigma 
= {\cal E}_s^n(\w_h), \qquad 
\forall \w_h \in \S_h^r 
\label{s1-3} 
\end{align} 
and 
\begin{align} 
& \rho_f  (D_{\tau} \u^n, \, \bv_h)   + a_f(\u^n, \, \bv_h) - b(p^n, \, \bv_h) 
+ b(q_h, \,  \u^n) 
 + \rho_s \epsilon_s \left ( \frac{\u^n -  \s^n}{\tau}, \, \bv_h \right  )_\Sigma 
 \nn \\ 
& = ( \sig^{n-1} \cdot {\bf n}, \, \bv_h )_\Sigma  
-  ( \u^n - \s^n, \, \sig(\bv_h, q_h)\cdot \n  )_\Sigma  
-  \frac{\tau(1+\beta)}{\rho_s \epsilon_s}  ( ( \sig^n- \sig^{n-1}) \cdot {\bf n}, \,  \sig(\bv_h, q_h)\cdot \n  )_\Sigma  
\nn \\ 
& \quad +  {\cal E}_{f}^n(\bv_h, q_h) , \hskip1in  \forall (\bv_h, q_h)\in \X_h^r\times Q_h^{r-1}
\label{s2-3} 
\end{align} 
where $ \sig^n = \sig(\u^n,  p^n)$ and 
the truncation error functions satisfy the following estimates: 
\begin{align} 
\begin{array}{ll} 
& |{\cal E}_{s}^n(\w_h)| \le C\tau \| \w_h \|_\Sigma , \\ [5pt]
&|  {\cal E}_{f}^n(\bv_h, \q_h) | \le  C  \tau(\| \bv_h \|_{\Sigma} +\|\bv_h\|)
+ C\tau^2 \| \sig(\bv_h, \q_h) \cdot \n \|_{\Sigma} \, .
\end{array} 
\label{EE} 
\end{align} 

For given $( \u^n, p^n,  \e^n, \s^n)$, we denote by $(R_h\u^n,R_hp^n,R_h\e^n,R_h\s^n)$ the corresponding coupled Ritz projection and define $R_h\mathcal{T}_0^n$ to be the defect satisfying
\begin{align} 
	& R_h\e^n =  R_h\e^{n-1} + \tau R_h\s^{n}+ R_h \mathcal{T}_0^n \qquad \forall n\geq 1\nn .    
\end{align} 
Then we introduce the following error decomposition:  
\begin{align} 
& e_u^n: =  \u^n - \u_h^n = \u^n - R_h \u^n + R_h \u^n - \u_h^n: = \theta_u^n + \delta_{u}^n, \hskip1in  \mbox{in } \Omega  
\nn \\ 
&e_p^n: =   p^n - p_h^n = p^n - R_h p^n + R_h p^n - p_h^n: = \theta_p^n + \delta_{p}^n, 
\hskip1.1in \, \mbox{in } \Omega  
\nn \\ 
&e_\sigma^n: =  \sig(\u^n, p^n) - \sig(\u_h^n, p_h^n) 
= \sig(\theta_u^n, \theta_p^n) + \sig(\delta_u^n, \delta_p^n) 
: = \theta_\sigma^n + \delta_\sigma^n  
\qquad \quad \, \mbox{in } \Omega  
\nn \\ 
& e_s^n: =  \s^n - \s_h^n = \s^n - R_h \s^n + R_h \s^n - \s_h^n: = \theta_s^n + \delta_s^n, \hskip1.1in  \, \, \mbox{on } \Sigma  \, . 
\nn \\ 
& e_\eta^n: =  \e^n - \e_h^n = \e^n - R_h \e^n + R_h \e^n - \e_h^n: = \theta_\eta^n + \delta_\eta^n, \hskip1in \, \, \, \mbox{on } \Sigma  \, . \nn
\end{align} 
Since $\u^n|_\Sigma = \s^n$, it follows that $\theta_u^n|_\Sigma=\theta_s^n$. Moreover, the following relations hold: 
\begin{align} 
(\u^n - \u^{n-1})-(\s^n_h - \u^{n-1}_h)  & = \theta_u^n + \delta_s^n  
 - \theta_u^{n-1} - \delta_u^{n-1} , 
\nn \\ 
(\u^n - \u^n)-(\u^n_h - \s^{n}_h)  
& = \theta_u^n +  \delta_u^n - \theta_u^{n}  - \delta_s^n  
=  \delta_u^n -  \delta_s^{n}
\quad\mbox{on $\Sigma$} .
\nn 
\end{align}  

By using \refe{s1}--\refe{s2} and \refe{s1-3}--\refe{s2-3}, we can write down the following error equations:
\begin{align} 
&  \rho_s \epsilon_s \left ( \frac{ \delta_s^n - \delta_u^{n-1}}{\tau}, \, \w_h \right )_\Sigma 
+ a_s( \delta_\eta^n, \, \w_h) 
+ ( \delta_\sigma^{n-1} \cdot {\bf n}, \, \w_h )_\Sigma 
 = 
  {\cal E}_s^n(\w_h) - F^n_s(\w_h), \quad \forall \w_h \in \S_h^r 
\label{err-e1} 
\\ 
& \delta_\eta^n = \delta_\eta^{n-1} 
+ \tau \delta_s^n+ R_h \mathcal{T}_0^n, \qquad  \mbox{on } \Sigma  
\label{err-e2} \\[15pt]
& \rho_f  \left ( \frac{\delta_u^n-\delta_s^n}{\tau}, \, \bv_h \right )   + a_f(\delta_u^n, \, \bv_h) - b(\delta_p^n, \, \bv_h^n) 
+ b(q_h, \, \delta_u^n) 
 + \rho_s \epsilon_s \left ( \frac{\delta_u^n - \delta_s^n}{\tau}, \, \bv_h \right )_\Sigma  
 \nn \\ 
 & = ( \delta_\sigma^{n-1}\cdot \n, \, \bv_h )_\Sigma 
 - ( \delta_u^n - \delta_s^n, \, \sig(\bv_h, \, q_h) )_\Sigma 
 - \frac{\tau(1+\beta)}{\rho_s \epsilon_s}  ( (\delta_\sigma^n - \delta_\sigma^{n-1})\cdot \n,  \, \sig(\bv_h, q_h) \cdot \n )_\Sigma 
\nn \\ 
& \quad + {\cal E}_f^n(\bv_h, q_h) - F_f^n(\bv_h, q_h), 
\hskip1in  \forall (\bv_h, q_h)\in \X_h^r\times Q_h^{r-1}
\label{err-e3} 
 \end{align} 
where 
\begin{align} 
F^n_s(\w_h)&  =   \rho_s \epsilon_s ( D_\tau \theta_u^n, \, \w_h )_\Sigma 
+  a_s(\theta_\eta^n, \, \w_h)  + ( \theta_\sigma^{n-1}\cdot {\bf n}, \, 
\w_h )_\Sigma 
\label{Fs} 
\\ 
 F_f^n(\bv_h, q_h)  & =  
 \rho_f(D_\tau \theta_u^n, \bv_h)+a_f(\theta_u^n, \, \bv_h) - b(\theta_p^n, \, \bv_h) 
 \nn \\ 
 &\quad\, - ( \theta_\sigma^{n-1}\cdot \n, \, \bv_h )_\Sigma 
 +\frac{\tau(1+\beta)}{\rho_s \epsilon_s}  ( (\theta_\sigma^n - \theta_\sigma^{n-1})\cdot \n, \sig(\bv_h, \, q_h) \cdot \n )_\Sigma 
 \label{Ff}
\end{align} 
Moreover, we have the following result:
\begin{align*}
 	\theta_\eta^n=\theta_\eta^{n-1}+\tau\theta_s^n+(\mathcal{T}_0^n-R_h \mathcal{T}_0^n), 
\end{align*} 
where the last term can be estimated by using \eqref{SR-error-2}, i.e., 
\begin{align} \label{T0n-RhT0n}
 &\|\mathcal{T}_0^n- R_h \mathcal{T}_0^n\|_{H^1(\Sigma)}\leq C\tau^2\|\partial_t(R_h\u-\u)\|_{L^\infty H^1(\Sigma)}\leq C\tau^2h^r . 
\end{align}
Therefore, by the triangle inequality with estimates \eqref{Estimate-T0n} and \eqref{T0n-RhT0n}, we have 
\begin{align} 
\| R_h \mathcal{T}_0^n\|_{H^1(\Sigma)} 
\le 
\| \mathcal{T}_0^n \|_{H^1(\Sigma)}
+\| \mathcal{T}_0^n-  R_h \mathcal{T}_0^n\|_{H^1(\Sigma)} \leq C\tau^2\quad\forall n\geq 1
\label{err-theta} 
\end{align} 
 
We take $(\bv_h, q_h) = (\delta_u^n, \delta_p^n) \in \X_h^r\times Q_h^{r-1}$ in \refe{err-e3} and $\w_h =\delta_s^{n} 
 \in \S_h^r$ in \refe{err-e1}, respectively, and then sum up the two results. Using the stability analysis in \refe{stability-2} and the relation 
 \begin{align*}
 	\delta_s^n=D_\tau\delta_\eta^n-\tau^{-1} R_h \mathcal{T}_0^n , 
 \end{align*}
we obtain 
\begin{align} 
&D_\tau E_0( \delta_u^n, \delta_p^n, \delta_\eta^n) 
+ E_1( \delta_u^n, \delta_s^n, \delta_\eta^n)\nn \\
&\le {\cal E}_s^n(\delta_s^n ) 
 - F_s^n(\delta_s^n) 
+ {\cal E}_f^n(\delta_u^n, \delta_p^n) - F_f^n(\delta_u^n,\delta_p^n)+\tau^{-1}a_s(\delta_\eta^n, R_h \mathcal{T}_0^n)\, . 
\label{error-1}
\end{align} 
To establish the error estimate,  we need to estimate each term on the right-hand side of \eqref{error-1}. 
From \refe{EE} and \eqref{err-theta} we can see that 
\begin{align} 
\begin{array}{ll} 
& |{\cal E}^n_s(\delta_s^n)| \le C\tau \|\delta_s^n \|_\Sigma 
 \\ [5pt]
&|  {\cal E}^n_f(\delta_u^n, \, \delta_p^n) | \le  C \tau (\| \delta_u^n \|_\Sigma+\|\delta_u^n\|) 
+ \tau^2 \| \delta_\sigma^n \cdot \n \|_\Sigma
\\ [5pt]
&|\tau^{-1}a_s(\delta_\eta^n, \, R_h\mathcal{T}_0^n)|\leq C\tau\|\delta_\eta^n\|_{s}
\end{array} 
\label{EE-1} 
\end{align} 
It remains to estimate $ F_s^n(\delta_s) + F_f^n(\delta_u,\delta_p)$ from the right-hand side of \eqref{error-1}. 

\begin{enumerate}
\item The second term in \eqref{Fs} plus the second and third terms in \eqref{Ff} can be estimated as follows. 
Let $\xi_h^n: = \delta_u^n - E_h(\delta_u^n - \delta_s^n)$, where $E_h(\delta_u^n-\delta_s^n)$ is an extension of $\delta_u^n-\delta_s^n$ to $\Omega$ satisfying estimate \eqref{Eh-estimate} 
	and $\xi_h^n|_{\Sigma} = \delta_s^n$. By choosing $v_h=\xi_h^n$ and $q_h=0$ in \eqref{evolution-ritz} (definition of the coupled Ritz projection), we obtain the following relation:  
	\begin{align}
		& a_f(\theta_u^{n}, \delta_u^{n}) - b(\theta_p^{n},\delta_u^{n})
		+ a_s(\theta_\eta^n,\delta_s^n)
		\nn \\
	 &=  a_f(\theta_u^n, E_h(\delta_u^n-\delta_s^n))  
	- b(\theta_p^n,E_h(\delta_u^n-\delta_s^n)) 
	-(\theta_u^n, \xi_h^n)-(\theta_s^n,\delta_s^n)_\Sigma
		\nn\\
		&\lesssim Ch^r\|E_h(\delta_u^n - \delta_s^n)\|_f 
		+ Ch^{r+1}(\|\xi_h^n\|+\|\delta_s^n \|_\Sigma)\nn\\
		&\leq  Ch^{r-1/2}\|\delta_u^n - \delta_s^n\|_\Sigma 
		+ Ch^{r+1}(\|\delta_u^n\|+\|\delta_s^n\|_\Sigma),
	\end{align}	
	where we have used estimate \eqref{SR-error-1}--\eqref{SR-error-2}.

	\item The third term in \eqref{Fs} plus the fourth term in \eqref{Ff} can be estimated as follows: 
	\begin{align}
		&(\theta_\sigma^{n-1}\cdot \n,\delta_s^n)_\Sigma-(\theta_\sigma^{n-1}\cdot\n, \delta_u^n )_\Sigma \nn\\
		\lesssim & \|\theta_\sigma^{n-1}\cdot \n\|_\Sigma \|\delta_s^n - \delta_u^n \|_{\Sigma}\nn\\
		\leq & C(h^{r-1/2}+h^{-1/2}(\|\theta_u^{n-1}\|_{H^1}
		+\|\theta^{n-1}_p\|))\|\delta_s^n - \delta_u^n \|_{\Sigma}\nn\\
		\leq & Ch^{r-1/2}\|\delta_s^n - \delta_u^n \|_{\Sigma},
	\end{align}
	where we used \eqref{sig-error-inverse} in the second inequality and 
	 \eqref{SR-error-1} in the last inequality.

	\item For the first term in \eqref{Fs} and \eqref{Ff}, respectively, we have 
	\begin{align}
			&\rho_s \epsilon_s
			 ( D_\tau \theta_u^n, \, \delta_s^n )_\Sigma  
			 \leq \frac{C}{\tau}\|\delta_s^n\|_\Sigma\int_{t_{n-1}}^{t_n}\|\partial_t \theta_u(t)\|_\Sigma dt \\
			& \rho_f(D_\tau \theta_u^n, \delta_u^n )\leq \frac{C}{\tau}
			\|\delta_u^n \|\int_{t_{n-1}}^{t_n}\|\partial_t \theta_u(t) \| dt
	\end{align}
	\item 
	The last term in \eqref{Ff} can be estimated by using \eqref{SR-error-2} and \eqref{sig-error-inverse}, i.e., 
	\begin{align}\label{EE-5}
		&\frac{\tau}{\rho_s \epsilon}  ( (\theta_\sigma^n - \theta_\sigma^{n-1})\cdot \n, \sig(\delta_u^n, \, \delta_p^n) \cdot \n )_\Sigma \nn\\
		\le & C\tau\left(\int_{t_{n-1}}^{t_n}\|\sig(\partial_t \theta_u, \, \partial_t \theta_p)(t)\n\|_\Sigma dt\right)\|\sig(\delta_u^n, \, \delta_p^n)\cdot \n\|_\Sigma\nn\\
	\leq & C\tau^2h^{r-1/2}\|\sig(\delta_u^n, \, \delta_p^n) \cdot \n\|_\Sigma 
	\end{align}
\end{enumerate} 
Now we can substitute estimates \eqref{EE-1}--\eqref{EE-5} into the energy inequality in \eqref{error-1}. This yields the following result: 
\begin{align}\label{final-estimate-1}
	&D_\tau E_0( \delta_u^n, \delta_p^n, \delta_\eta^n) 
	+ E_1( \delta_u^n, \delta_s^n, \delta_\eta^n)\nn \\
	&\le C\tau (\|\delta_s^n \|_\Sigma+\| \delta_u^n \|_\Sigma+\|\delta_u^n\|+\|\delta_\eta^n\|_s)+Ch^{r-1/2}\|\delta_u^n-\delta_s^n\|_\Sigma+Ch^{r+1}(\|\delta_u^n\|+\|\delta_s^n\|_\Sigma)\nn\\
	&\quad\, +\frac{C}{\tau}\|\delta_s^n\|_\Sigma\int_{t_{n-1}}^{t_n}\|\partial_t \theta_u(t) \|_\Sigma dt +\frac{C}{\tau}\|\delta_u^n\|\int_{t_{n-1}}^{t_n}\|\partial_t \theta_u(t) \| dt
	+C\tau^2\|\delta_\sigma^n \cdot \n\|_\Sigma	\, .
\end{align} 
Since $\|\delta_s^n \|_\Sigma \le \|\delta_s^n - \delta_u^n\|_\Sigma + \|\delta_u^n \|_\Sigma$, by using Young's inequality, we can re-arrange the right hand side of \eqref{final-estimate-1} to obtain
\begin{align}\label{final-estimate-2}
		&D_\tau E_0( \delta_u^n, \delta_p^n, \delta_\eta^n) 
	+ E_1( \delta_u^n, \delta_s^n, \delta_\eta^n)\nn \\
	&\le C\varepsilon^{-1}(\tau^2+Ch^{2(r+1)}+\tau h^{2r-1})+C \varepsilon (\| \delta_u^n \|^2_\Sigma+\|\delta_u^n\|^2+\|\delta_\eta^n\|^2_s)+\frac{C\varepsilon}{\tau}\|\delta_u^n-\delta_s^n\|^2_\Sigma\nn\\
	&\quad\, +\frac{C\varepsilon^{-1}}{\tau} 
	 \left ( 
	\int_{t_{n-1}}^{t_n}\|\partial_t \theta_u (t)\|^2_\Sigma dt + 
	\int_{t_{n-1}}^{t_n}\|\partial_t \theta_u(t)\|^2 dt \right ) 
	+C\tau^2\|\delta_\sigma^n \cdot \n\|^2_\Sigma	
\end{align}
where $0<\varepsilon<1$ is an arbitrary constant.


We can choose a sufficiently small $\varepsilon$ so that the term $\frac{C\varepsilon}{\tau}\|\delta_u^n-\delta_s^n\|^2_\Sigma$ can be absorbed by $E_1( \delta_u^n, \delta_s^n, \delta_\eta^n)$ on the left-hand side. Then, using the discrete Gronwall's inequality and the estimates of $\theta_u$ in \refe{SR-error-3}, as well as the definition of $E_0$ and $E_1$ in \refe{E0}--\refe{E1}, we obtain 
\begin{align}\label{end-estimate1}
	 E_0( \delta_u^n, \delta_p^n, \delta_\eta^n) 
	+\sum_{m=1}^n \tau E_1( \delta_u^m, \delta_s^m, \delta_\eta^m)\leq C E_0( \delta_u^0, \delta_p^0, \delta_\eta^0)+ C(\tau^2+Ch^{2(r+1)}+\tau h^{2r-1}),
\end{align}
Since the initial values satisfy the estimates in \eqref{initial-value}, the term $E_0( \delta_u^0, \delta_p^0, \delta_\eta^0)$ can be estimated to the optimal order. Thus inequality \eqref{end-estimate1} reduces to 
	\begin{align}\label{thm-3.2-tmp1}
		&\|\delta_u^n\|+\|\delta_u^n\|_{\Sigma}+\|\delta_\eta^n\|_s+\|\delta_u^n-\delta_s^n\|_\Sigma 
	\le 	C (h^{r - 1 / 2} \tau^{1 / 2} + \tau + h^{r + 1}).
	\end{align}
It follows from the relation $\delta_\eta^n = \delta_\eta^{n-1} + \tau \delta_s^n+R_h\mathcal{T}_0^n$, $n\ge 1$, that
\begin{align}
	\|\delta_\eta^n\|_\Sigma\leq &\|\delta_\eta^0\|_\Sigma+\sum_{m=1}^n\tau \|\delta_s^m\|_\Sigma+\sum_{m=1}^n\|R_h\mathcal{T}_0^n\|_\Sigma
	\leq 	C (h^{r - 1 / 2} \tau^{1 / 2} + \tau + h^{r + 1}),
\end{align}
where we have used \eqref{thm-3.2-tmp1} and \eqref{err-theta}. 
Then, combining the two estimates above with the following estimate for the projection error: 
\begin{align*}
	\|\theta_u^n\|+\|\theta_u^n\|_\Sigma+\|\theta_\eta^n\|_\Sigma\leq Ch^{r+1}\quad \forall n\geq 0,
\end{align*}
we obtain the following error bound:
\begin{align*}
	&\|e_u^n\|+\|e_u^n\|_{\Sigma}+\|e_\eta^n\|_\Sigma
	\le C (h^{r - 1 / 2} \tau^{1/2} + \tau + h^{r + 1}) 
	\le C(\tau + h^{r+1}) \, ,
\end{align*}
where the last inequality uses $h^{r - 1 / 2} \tau^{1/2} \le \tau + h^{2r-1}$ and $r\ge 2$. 
This completes the proof of Theorem \ref{main-2}. \hfill\endproof
%
%
\section{The proof of Theorem \ref{main-3}} \label{sec-ritz}


We present the proof of the Theorem \ref{main-3} step-by-step in the next three subsections.

\subsection{The definition of $R_h \e(0)$ in the coupled Ritz projection}  \label{sec-ritz-1}

In this subsection, we focus on designing the initial value $R_h \e(0)$ for our coupled non-stationary Ritz projection.

We first present two auxiliary Ritz projections $R_h^S$ and $R_h^D$ associated to the structure model and fluid model in Definitions \ref{ritz-solid}-\ref{ritz-fluid-dirichlet-stokes}, respectively.
Next, in terms of these two auxiliary Ritz projections, we define the initial value $R_h \e(0)$ in Definition \ref{def-Rhe0} which is only for our theoretical purpose. Finally, an alternative definition of $R_h \e(0)$ for practical computation is given in Definition \ref{def-Rshe0}.

\begin{definition}[Structure--Ritz projection $R_h^S$]\label{ritz-solid}
\upshape
	We define an auxiliary Ritz projection $R_h^S:\S\to \S_h^r$ for the elastic structure problem by 
	\begin{equation}\label{RhS-def}
		a_s(R_h^S\s-\s,\w_h)+(R_h^S\s-\s,\w_h)_\Sigma=0\quad\forall \w_h\in \S_h^r.
	\end{equation}
This is the standard Ritz projection on $\Sigma$, which satisfies the estimate 
$\|R_h^S\s-\s\|_\Sigma\leq Ch^{r+1}$ when $\s$ is sufficiently smooth. Moreover when $r\geq 2$, there holds negative norm estimate
	\begin{equation}\label{Rhs-neg}
		\|R_h^S\s-\s\|_{H^{-1}(\Sigma)}\leq Ch^{r+2}.
	\end{equation}
\end{definition} 

Let $\mathring \X^r_h := \{ \bv_h \in \X^r_h : \bv_h |_{\Sigma} = 0 \}$ and $Q_{h,0}^{r-1}:=\{q_h\in Q_h^{r-1}:q_h\in L^2_0(\Omega)\}$. We denote  $\widetilde{\S}_h^r:=\{\bv_h\in \S_h^r:(\bv_h,\n)_\Sigma=0\}$ and 
	 by $\widetilde{P}$ the $L^2(\Sigma)$-orthogonal projection from $\S_h^r$ to $\widetilde{\S}_h^r$.

\begin{definition}[Dirichlet Stokes--Ritz projection $R_h^D$]\label{ritz-fluid-dirichlet-stokes}
\upshape
Let $\hat{\X}:=\{\u\in \X:\u|_\Sigma\in \S\}$. 
	We define an auxiliary Dirichlet Stokes--Ritz projection $R^D_h : \hat{\X} \times Q
	\rightarrow \X^r_h \times Q^{r-1}_h$ by 
	\begin{subequations}\label{RhD-def}
		\begin{align}
			&a_f( \u - R^D_h\u, \bv_h) - b( p - R^D_hp, \bv_h) + (\u -
			R_h^D\u, \bv_h) =  0 \quad \forall \bv_h \in \mathring \X_h^r , \label{RhD-def1}\\
			&b(q_h, \u - R_h^D\u) =  0 \quad \forall q_h \in Q^{r-1}_{h,0};\quad 
			R_h^D \u  =  \widetilde{P} R_h^S(\u|_\Sigma) \quad \text{on $\Sigma$},\label{RhD-def2}
		\end{align}
	\end{subequations}
In addition, we choose $R_h^D p$ to satisfy $R_h^D p - p \in L^2_0 (\Omega)$. This uniquely determines a solution $(R_h^Du,R_h^Dp)\in \X^r_h \times Q^{r-1}_h$, as explained in the following Remark.
\end{definition}

\begin{remark}\upshape 
In order to see the existence and uniqueness of solution $(R_h^Du,R_h^Dp)$ defined by \eqref{RhD-def}, 
we let $\hat{\u}_h\in \X_h^r$ be an extension of $\widetilde{P}R_h^S\u$ to the bulk domain $\Omega$ and let $\hat{p}_h$ be the $L^2(\Omega)$-orthogonal projection of $p$ onto $Q_h^{r-1}$. Then $\hat{\u}_h-R_h^D\u\in \mathring{\X}_h^r$ and $\hat{p}_h-R_h^Dp\in Q^{r-1}_{h,0}$.
Replacing $(\u,p)$ and $(R_h^D\u, R_h^Dp)$ by  $(\u-\hat{\u}_h, p- \hat{p}_h)$ and  
$(R_h^D\u-\hat{\u}_h,R_h^Dp - \hat{p}_h)$ in \eqref{RhD-def1}-\eqref{RhD-def2} respectively,
we obtain a standard Stokes FE system with a homogeneous Dirichlet boundary condition for $(R_h^D\u-\hat{\u}_h,R_h^Dp - \hat{p}_h)$. The well-posedness directly follows the inf-sup condition \eqref{inf-sup-condition0}. 
\end{remark}

\begin{remark}\label{remark-4-2-div-free}\upshape 
The projection $\widetilde{P}$ in \eqref{RhD-def2} is introduced to guarantees that the $b(q_h, \u - R_h^D\u) =  0$ holds not only for $q_h\in Q_{h,0}^{r-1}$ but also for $q_h\in Q_h^{r-1}$. That is,
\begin{equation}\label{full-div-free}
		b(q_h, \u - R_h^D\u) =  0 \quad \forall q_h \in Q^{r-1}_{h}.
	\end{equation}
Since $Q_h^{r-1}=\{1\}\oplus Q^{r-1}_{h,0}$, this follows from the first relation in \eqref{RhD-def2} and the following relation: 
\begin{equation*}
		b(1,\u-R_h^D\u)=(R_h^D\u,\n)_\Sigma= (\widetilde{P}R_h^S\u,\n)_\Sigma=0 \, ,
	\end{equation*}
where $b(1,\u)=0$ for the exact solution $\u$ which satisfies $\nabla\cdot\u=0$. 

Especially, when $\u$ is replaced with $\partial_t\u(0)$, we have
\begin{equation}\label{RhD-partial_tu-q}
		 	b(q_h,(\partial_t\u-R_h^D\partial_t\u)(0))=0\quad\forall q_h\in Q_h^{r-1} . 
\end{equation}
Relation \eqref{RhD-partial_tu-q} is needed in estimating the error between $(\partial_tR_h\u(0),\partial_tR_hp(0))$ and $(\partial_t\u(0),\partial_tp(0))$ in Lemma \ref{time-derivative-ritz-u-p} below. Furthermore, in Definition \ref{def-Rhe0}, we defined $(R_h\u(0),R_hp(0))$ via a Dirichlet-type Stokes-Ritz projection with the boundary condition $R_h\u(0)|_\Sigma=\widetilde{P}R_{sh}\u(0)$. 

To facilitate further use of $\widetilde{P}$ in the following analysis, here we derive an explicit formula for $\widetilde{P}$. We denote by $\n_h \in \S^r_h$ the $L^2 (\Sigma)$-orthogonal projection of unit normal
	vector field $\n$ of $\Sigma$ to $\S^r_h$, i.e.,
	\begin{equation}\label{nh-def}
		(\n, \w_h)_{\Sigma} = (\n_h, \w_h)_{\Sigma} \quad \forall \w_h \in \S^r_h .
	\end{equation}
Then for any $\w_h \in \S^r_h$, we have
	\begin{equation}\label{P-tilde-formula0}
		\widetilde{P} \w_h = \w_h - \lambda (\w_h) \n_h \in \tilde{\S}^r_h \quad  \text{ with } \lambda (\w_h):= \frac{(\w_h, \n)_{\Sigma}}{\| \n_h \|^2_{\Sigma}}.
	\end{equation}
	From $\|\n-\n_h\|_\Sigma\leq \|\n-I_h\n\|_\Sigma\leq Ch^{r+1}$ (since $\n$ is smooth on $\Sigma$), especially we have $\|\n_h\|_\Sigma\sim C$ and  
	\begin{equation}\label{lambda-RhS-u}
	\lambda(R_h^S\u)=\frac{(R_h^S\u-\u,\n)_\Sigma}{\|\n_h\|^2_\Sigma} 
	\lesssim Ch^{r+1}\text{ and }\|\widetilde{P}R_h^S\u-R_h^S\u\|\leq Ch^{r+1}.
	\end{equation}
	Therefore we obtain the estimate  $\|R_h^D\u-\u\|_{\Sigma}\leq Ch^{r+1}$.
\end{remark}
The following lemma on the error estimates of the Dirichlet Stokes--Ritz projection is standard. We refer to \cite[Proposition 8, Proposition 9]{Gunzburger} for the proof of \eqref{RhD-error}. The negative norm estimate of pressure in \eqref{RhD-p-neg-err} requires a further duality argument, which is presented in the proof of 
Lemma B.3 in the Appendix B.
 We omit the details here.  
\begin{lemma}\label{lemma:Dirichlet-Ritz}
Under the regularity assumptions in Section \ref{section:regularity}, the Dirichlet Stokes--Ritz projection $R_h^D$ defined in \eqref{RhD-def} satisfies the following estimates: 
\begin{align}\label{RhD-error}
	&\|\u-R_h^D\u\|_{\Sigma} + \| \u - R_h^D \u \| + h \left ( \|\u - R_h^D \u\|_{H^1}
	+ \| p - R_h^D p \|\right ) \le C h^{r+1}  	,
%
\\
	\label{RhD-p-neg-err}
&		\| R_h^D p - p \|_{H^{- 1}} \leq C h^{r + 1} .
	\end{align}
\end{lemma}

We define an initial value $R_h\e(0)$ as follows in terms of the Dirichlet Ritz projection $R_h^D$. 
\begin{definition}[Initial value $R_h\e(0)$]\label{def-Rhe0}\upshape 

Firstly, assuming that the function $(R_h^D\partial_t \u(0),R_h^D\partial_t p(0))$ is known with operator $R_h^D$ defined by \eqref{RhD-def}, we define $R_{{sh}} \u (0) \in \S_h^r$ to be the solution of  the following weak formulation: 
	\begin{align}\label{ritz-initial1} 
		& a_s (_{} ( \u - R_{{sh}} \u) (0), \w_h)  + (( \u - R_{{sh}} \u)
		(0), \w_h)_{\Sigma} + a_f (( \partial_t \u - R_h^D\partial_t \u) (0), E_h\w_h)\nonumber\\
		&- b (( \partial_t p - R_h^D\partial_t p) (0), E_h\w_h) + ((
		\partial_t \u - R^D_h\partial_t \u) (0), E_h\w_h) = 0 \quad \forall \w_h \in {\bf S}^r_h,
	\end{align}
where $E_h\w_h$ denotes an extension of $\w_h$ to the bulk domain $\Omega$. From the definition of $R_h^D$ in \eqref{RhD-def} we can conclude that this definition is independent of the specific extension. Therefore, \eqref{ritz-initial1} still holds when replacing both  $\w_h$ and $E\w_h$ with $\bv_h \in {\bf X}^r_h$. 

Secondly, we denote by $(R_h \u (0),R_hp(0)) \in {\bf X}^r_h\times Q_h^{r-1}$ a Dirichlet-type Stokes--Ritz projection satisfying  
	\begin{subequations}\label{d3-2a2b}
		\begin{align}
			a_f(\u(0) - &R_h \u(0), \bv_h)  - b( p(0)-R_hp(0), \bv_h)  +(\u(0)-R_h\u(0),\bv_h)=0&& \forall \bv_h \in \mathring \X^r_h , 
			\label{d3-2a} \\
			&b (q_h, \u(0) - R_h \u(0)) = 0 \quad \forall q_h \in Q^{r-1}_{h,0};\quad 	R_h  \u (0) = \widetilde{P}  R_{s h} \u (0) && \text{on $\Sigma$} , 
			\label{d3-2b}
		\end{align}
	\end{subequations}
where we require $p (0) - R_h p (0) \in L^2_0 (\Omega)$. 

Finally, with the $R_h\u(0)$ and $R_hp(0)$ defined above, we define $R_h \e (0) \in \S_h^r$ to be the solution of the following weak formulation on $\Sigma$:
	\begin{align}\label{ritz-initial2}
a_f(\u(0) - R_h \u(0), E_h\w_h) & - b( p(0)-R_hp(0), E_h\w_h)  +(\u(0)-R_h\u(0), E_h\w_h)
\nn \\ 
& + a_s(\e(0) - R_h \e(0), \w_h) + ( \e(0) - R_h \e(0), \w_h )_\Sigma = 0 
		 \quad \forall \w_h \in \S^r_h.
	\end{align}	
Again  \eqref{ritz-initial2} also holds when replacing $\w_h$ and $E_h\w_h$  with $\bv_h \in {\bf X}^r_h$. 
\end{definition}

In addition, by differentiating \eqref{evolution-ritz} with respect to time, we have the following evolution equations:
	\begin{subequations}\label{evolution-ritz-t1}
		\begin{align}
			a_s ( \u - R_h\u, \bv_h)  + ( \u - R_h \u, \bv_h)_{\Sigma} + a_f (\partial_t ( \u
			- R_h\u), \bv_h) &\nn\\
			- b (\partial_t ( p - R_hp), \bv_h) + (\partial_t ( \u - R_h\u),
			\bv_h) &= 0 \quad \forall \bv_h \in \X^r_h , \label{evolution-ritz-t1-u-eta} \\[5pt]
	b (q_h, \partial_t ( \u - R_h\u)) &= 0 \quad \forall q_h \in Q^{r-1}_h \, . 
	\label{evolution-ritz-t1-p}
		\end{align}
	\end{subequations}
For the computation with the numerical scheme \eqref{s1}--\eqref{s2}, we can define the initial value $\e_h^0 =R_{sh}\e(0)\in \S^r_h$ in an alternative way below.

\begin{definition}[Ritz projection $R_{sh}\e(0)$]\label{def-Rshe0}\upshape
We define $\e_h^0 =R_{sh}\e(0)\in \S^r_h$ as the solution of the following weak formulation: 
\begin{align}
	&a_s ((R_{sh} \e - \e) (0), \w_h)  + ((R_{sh} \e - \e) (0),\w_h)_{\Sigma}
\qquad \forall \w_h \in \S^r_h \nn \\ 
	& = - a_f ((R_h^D \u - \u) (0), E_h \w_h)
	 + b ((R_h^D p - p) (0), E_h\w_h) - ((R_h^D\u - \u)(0), E_h \w_h) 
	  \, ,
	  \label{R_sh-eta} 
\end{align}
which does not require knowledge of $\partial_t\u(0)$ or $\partial_tp(0)$. Again, $E_h\w_h$ denotes an extension of $\w_h$ to the bulk domain $\Omega$, and this definition is independent of the specific extension. Therefore, \eqref{R_sh-eta} holds for
all $\bv_h \in {\bf X}^r_h$ with $\w_h$ and $E_h\w_h$ replaced by $\bv_h$ in the equation. For $r\geq 2$, the following result can be proved in the Appendix B
:
	\begin{equation}\label{super-approx}
		\|R_{sh}\e(0)- R_h \e(0)\|_{H^{1}(\Sigma)}  \leq Ch^{r+1}. 
	\end{equation}  

\end{definition}

%
%
\subsection{Error estimates for the coupled Ritz projection at $t=0$} \label{t:0 error-estimate}

Firstly, we consider the estimation of $R_{sh}\u(0)$ which occurs as an auxiliary function in the definition of $R_h\e(0)$ in Lemma \ref{Lemma:Rshu0}. Secondly, we present estimates for $\u(0)-R_h\u(0)$, $\e(0)-R_h\e(0)$ and $ p(0) - R_h p(0)$ in Lemma \ref{lemma:error-Rh-t=0}. Finally, we present estimates for the time derivatives $\partial_t (\u - R_h \u) (0)$ and $\partial_t ( p - R_hp) (0)$ in
Lemma \ref{time-derivative-ritz-u-p}.

\begin{lemma}\label{Lemma:Rshu0}
Under the assumptions in Sections \ref{section:regularity} and \ref{FE-space}, the following error estimate holds for the $R_{sh}\u(0)$ defined in \eqref{ritz-initial1}$:$ 
	\begin{align}
		\quad \| R_{{sh}} \u (0) - \u (0) \|_{\Sigma} + h \| R_{{sh}} \u(0) -
		\u (0) \|_s \leq C h^{r + 1} . \label{ritz-initial-estimate-Rsh}
	\end{align}
\end{lemma}
\begin{proof}
Since we can choose an extension $E_h{\bb \xi}_h$ to satisfy that  
 $\| E_h {\bb \xi}_h \|_{H^1 (\Omega)} \leq C \| {\bb \xi}_h \|_{H^1 (\Sigma)}$, 
 equation \eqref{ritz-initial1} implies that  
	\begin{align*}
		a_s (\u (0) - R_{{sh}}\u (0), {\bb \xi}_h)  + (\u (0) - R_{{sh}}\u (0),
		{\bb \xi}_h)_{\Sigma} \le Ch^r \| {\bb \xi}_h \|_{H^1 (\Sigma)} \, . 
	\end{align*} 
This leads to the following standard $H^1$-norm estimate: 
	\[ \| \u(0) - R_{{sh}} \u (0) \|_s + \|\u(0) - R_{{sh}} \u (0) \|_{\Sigma} \leq C h^r .\] 

In order to obtain an optimal-order $L^2$-norm estimate for $\u(0) - R_{{sh}} \u (0)$, we introduce the following dual 
problem: 
	\begin{equation}
	- \mathcal{L}_s \psi + \psi = R_{{sh}} \u (0) - \u (0) , \quad \text{$\psi$
		has periodic boundary condition on $\Sigma$}. 
	\end{equation}
The regularity assumption in \eqref{reg-ass3} implies that 
	\begin{equation*}
	a_s (\psi, {\bb \xi})  + (\psi, {\bb \xi})_{\Sigma} = ( \u(0) - R_{{sh}}\u (0),
	{\bb \xi})_{\Sigma} \quad \forall {\bb \xi} \in \S \quad \text{and } 
 \| \psi \|_{H^2(\Sigma)} \leq C\| \u(0) - R_{{sh}}\u (0) \|_{\Sigma} \, .
	\end{equation*}
	We can extend $\psi$ to be a function on $\Omega$, still denoted by $\psi$, satisfying the periodic boundary condition and $\| \psi \|_{H^2 (\Omega)} \leq C \| \psi\|_{H^2 (\Sigma)}$. Therefore, choosing ${\bb \xi}=\u(0) - R_{{sh}}\u (0)$ in the equation above leads to 
	\begin{align*}
		\| \u(0) - R_{{sh}}\u (0) \|_{\Sigma}^2 
		=\, & a_s (\u(0) - R_{{sh}}\u (0),\psi)  + (\u(0) - R_{{sh}}\u (0),\psi) _{\Sigma}\\
		=\,&a_s (\u(0) - R_{{sh}}\u (0),\psi - I_h \psi) + (\u(0) - R_{{sh}}\u (0),\psi - I_h \psi)_{\Sigma}\\
		& - a_f (\partial_t\u(0) - R_h^D\partial_t\u (0), I_h \psi) + 
		b(\partial_tp(0) - R_h^D\partial_tp (0), I_h \psi)\\
		& - (\partial_t\u(0) - R_h^D\partial_t\u (0), I_h \psi) \quad\mbox{(relation \eqref{ritz-initial1} is used)}\\
		\leq\, & C h^{r + 1} \| \psi \|_{H^2 (\Sigma)} + | a_f (\partial_t\u(0) - R_h^D\partial_t\u (0), \psi)| \\
		& + |b(\partial_tp(0) - R_h^D\partial_tp (0), \psi)| 
		 + |(\partial_t\u(0) - R_h^D\partial_t\u (0), \psi) |\, .
	\end{align*}
	Since 
	\begin{align*}
		&(\DD ( \partial_t \u(0) - R_h^D\partial_t \u(0)), \DD \psi) \\
		&= - ( \partial_t \u (0) -
		R_h^D\partial_t \u (0), \nabla \cdot \DD \psi) + (\partial_t \u (0) -
		R_h^D \partial_t \u (0), \DD \psi\cdot  \n)_{\Sigma}\\
		& \lesssim C h^{r + 1} \| \psi\|_{H^2(\Sigma)} , 
	\end{align*}
where the last inequality uses the estimate $\| \psi \|_{H^2 (\Omega)} \leq C \| \psi\|_{H^2 (\Sigma)}$ as well as the estimates of $\|\partial_t \u(0) - R_h^D\partial_t \u(0)\|$ and $ \|\partial_t \u(0) - R_h^D\partial_t \u(0)\|_\Sigma$ in \eqref{RhD-error} (with $\u(0)$ replaced by $\partial_t\u(0)$ therein). 
	Furthermore, using the $H^{-1}$ estimate in \eqref{RhD-p-neg-err}, we have  
	\begin{align} 
	 b( \partial_t p (0) - R_h^D\partial_t p (0), \psi)
 \lesssim C \|  \partial_t p (0) - R_h^D\partial_t p (0) \|_{H^{-
			1}_p} \| \psi \|_{H^2}  \le C h^{r+1} \| \psi \|_{H^2(\Sigma)} .
	\nn 
	\end{align} 
	Then, summing up the estimates above, we obtain  
	\[ \| \u (0) - R_{{sh}} \u (0) \|_{\Sigma}  \leq C h^{r + 1} \, . 
	\] 
The proof of Lemma \ref{Lemma:Rshu0} is complete. 
\hfill \end{proof}
\vskip0.1in 

\begin{lemma}\label{lemma:error-Rh-t=0}
Under the assumptions in Sections \ref{section:regularity} and \ref{FE-space}, the following error estimates hold (for the coupled Ritz projection in Definition \ref{def-Rhe0})$:$
	\begin{align}
	&	\| \e(0) - R_h\e (0)\|_{\Sigma}  + h\|  \e(0) - R_h\e (0) \|_s 
	+  \| \u(0) - R_h\u(0) \|_{\Sigma} \leq C h^{r + 1} , 
	 \label{ritz-initial-estimate1}
\\ 
& \|\u(0) - R_h\u(0)\| + h \| p(0) - R_hp(0) \|   \leq C h^{r + 1}  . 
		\label{ritz-initial-estimate2}
	\end{align}
\end{lemma}

\begin{proof}
From \eqref{P-tilde-formula0} we know that $R_h\u(0)=\widetilde PR_{sh}\u(0)=R_{sh}\u(0)-\lambda(R_{sh}\u(0))\n_h$ on $\Sigma$, with 
	\[ |\lambda (R_{{sh}}\u(0))|= \frac{|(R_{{sh}} \u (0),
		\n)_{\Sigma}|}{\| \n_h \|^2_{\Sigma}} = \frac{|(R_{{sh}}\u(0) - \u (0),
		\n)_{\Sigma}|}{\| \n_h \|_{\Sigma}^2} \le C \| R_{{sh}} \u (0) - \u (0)\|_{\Sigma} \le C h^{r + 1}. \]
Therefore, using the triangle inequality, we have 
$$
\| \u(0) - R_h\u(0) \|_\Sigma 
\le \| \u(0) - R_{sh}\u(0)\|_\Sigma + |\lambda(R_{sh}\u(0)) | \| \n_h \|_\Sigma 
\le Ch^{r+1} ,
$$
where the last inequality uses the estimate of $|\lambda(R_{sh}\u(0)) | $ above and the estimate in \eqref{ritz-initial-estimate-Rsh}. 

Since $(R_h\u(0),R_hp(0))$ is essentially a Dirichlet Ritz projection with a different boundary value, i.e., $\widetilde{P}  R_{s h} \u (0)$, the error estimates for $ \|\u(0) - R_h\u(0)\| $ and $ \| p(0) - R_hp(0) \|$ are the same as those in Lemma \ref{lemma:Dirichlet-Ritz}. With the optimal-order estimates of $\| \u(0) - R_h\u(0) \|_\Sigma $, $ \|\u(0) - R_h\u(0)\| $ and $ \| p(0) - R_hp(0) \|$, the estimation of $\| \e(0) - R_h\e (0)\|_{\Sigma}$ and $\|  \e(0) - R_h\e (0) \|_s $ would be the same as the proof of Lemma \ref{Lemma:Rshu0}. 
\hfill\end{proof}
\vskip0.1in

In addition, we present estimates for the time derivatives $\partial_t (\u - R_h \u) (0)$ and $\partial_t ( p - R_hp) (0)$. 
To this end, we use the following relation: 
\begin{equation}
	( \u - R_h\u) (0) = ( \u - R_{{sh}}\u) (0) + \lambda (R_{{sh}} \u
	(0)) \n_h \quad\mbox{on}\,\,\,\Sigma . 
\end{equation}
Replacing $( \u - R_{{sh}}\u) (0) $ by $( \u - R_h\u) (0) - \lambda (R_{{sh}} \u
	(0)) \n_h$ in \eqref{ritz-initial1}, we have  
\begin{align}\label{Rh-partial-t-u-1}
		& a_s (( \u - R_h\u) (0), \bv_h)  + (( \u - R_h\u) (0), \bv_h)_{\Sigma} +a_f (( \partial_t \u - R_h^D\partial_t \u) (0), \bv_h)\nn \\
		&\quad - b (( \partial_t p - R_h^D\partial_t p) (0), \bv_h) + ((
	\partial_t \u - R^D_h \partial_t \u) (0), \bv_h)\nn \\
	 & =  \lambda (R_{{sh}} \u (0)) (a_s (\n_h, \bv_h) + (\n_h, \bv_h)_{\Sigma}) 
	\quad \forall \bv_h \in \X^r_h . 
\end{align}
	Let $(\u^{\#}, p^{\#})\in \X\times Q$ be the weak solution of
	\begin{subequations}\label{u-hash-22}
			\begin{align}
			a_f (\u^{\#}, \bv ) - b (p^{\#}, \bv ) +(\u^{\#}, \bv ) &= a_s (\n , \bv ) + (\n, \bv)_{\Sigma}  && \forall \bv  \in \X  \label{u-hash-2} \\ 
			b (q , \u^{\#} ) &= 0 && \forall q  \in Q 
		\end{align}
	\end{subequations}
and $(\u_h^{\#}, p_h^{\#}) \in (\X^r_h, Q^{r-1}_h)$ denote the corresponding FE solution satisfying 
\begin{subequations}\label{def-u-hash22}
	\begin{align}\label{def-u-hash}
		a_f (\u^{\#}_h, \bv_h) - b (p^{\#}_h,
		\bv_h) + (\u^{\#}_h, \bv_h) &= a_s (\n_h, \bv_h) + (\n_h, \bv_h)_{\Sigma}  && \forall \bv_h \in \X^r_h \\
\label{def-u-hash-p}
		b (q_h, \u^{\#}_h) &= 0 && \forall q_h \in Q^{r-1}_h,
	\end{align}
\end{subequations}
where $\n_h$ is defined in \eqref{nh-def}. 

Note that \eqref{u-hash-22} is equivalent to the weak solution of
\begin{align}
	- & \nabla\cdot\sig(\u^\#,p^\#) + \u^\# = {\bf 0} \,\,\,\mbox{in}\,\,\,\Omega\quad\mbox{with}\,\,\, 
	\sig(\u^\#,p^\#){\bf n} = -\mathcal{L}_s{\bf n} +{\bf n} \,\,\,\mbox{on}\,\,\,\Sigma 
	\nn \\ 
	& \nabla \cdot \u^\# = 0 \,\,\,\mbox{in}\,\,\,\Omega . 
	\nn 
\end{align} 
Therefore, from the regularity estimate in \eqref{reg-ass1} (with $k=r-1/2$ therein) and assumption \eqref{Ls-second-order} on $\mathcal{L}_s$, we obtain the following regularity estimate for the solutions of \eqref{u-hash-22}: 
	\[ \| \u^{\#} \|_{H^{r+1}} + \| p^{\#} \|_{H^{r}} \leq C\|\n\|_{H^{r+3/2}(\Sigma)}\leq C . \]
By considering the difference between \eqref{u-hash-22} and \eqref{def-u-hash22}, the following estimates of ${\bf e}^{\#}_h:=I_h \u^{\#} - \u_h^{\#}$ and $m^{\#}_h:=I_h p^{\#} - p_h^{\#}$ can be derived for all $\bv_h
	\in \X^r_h$ and $q_h \in Q^{r-1}_h$: 
	\begin{align*}
	a_f ({\bf e}^{\#}_h, \bv_h) - b (m^{\#}_h, \bv_h) + ({\bf e}^{\#}_h, \bv_h) &\lesssim C {h^{r}}
	\| \bv_h \|_{H^1 (\Sigma)} + C h^{r} \| \bv_h \|_{H^1}\leq  C {h^{r-1/2}} \| \bv_h \|_{H^1} \\ 
	b (q_h, {\bf e}^{\#}_h) & \lesssim C h^{r} \| q_h \| ,
	\end{align*}
	where we have used the inverse estimate in \eqref{inverse-Sh} and the following trace inequality:
	\begin{equation*}
		\|\bv_h\|_{H^1(\Sigma)}\leq Ch^{-1/2}\|\bv_h\|_{H^{1/2}(\Sigma)}\leq C h^{-1/2} \| \bv_h \|_{H^1}.
	\end{equation*}
	From Korn's inequality and inf-sup condition \eqref{inf-sup-condition}, choosing $\bv_h={\bf e}_h^\#$ yields the following result: 
	\[ \| {\bf e}_h^{\#} \|_{H^1} + \| m_h^{\#} \| \leq C{h^{r-1/2}} ,\]
which also implies the following boundedness through the application of the triangle inequality:
\begin{equation*}
	\| \u_h^{\#} \|_{H^1} + \| p_h^{\#} \| \leq C .
\end{equation*} 
By using the boundedness of $H^1(\Omega)$-norm of $\u_h^{\#}$ and $L^2(\Omega)$-norm of $p_h^{\#}$, we can estimate $\partial_t (\u - R_h \u) (0)$ and $\partial_t ( p - R_hp) (0)$ as follows. 
 \begin{lemma} \label{time-derivative-ritz-u-p}
Under the assumptions in Sections \ref{section:regularity} and \ref{FE-space}, the following error estimates hold (for the time derivative of the coupled Ritz projection in Definition \ref{def-Rhe0})$:$
	\begin{align}\label{u-t-Linfty-H1-error}
	&  \| \partial_t (  \u - R_h\u) (0) \|  
		+ \| \partial_t ( \u - R_h\u) (0) \|_{\Sigma} + 
		h  \| \partial_t ( p - R_h p)  (0) \| \leq C h^{r +1} . 
	\end{align}
\end{lemma}
\begin{proof}
	Combining \eqref{Rh-partial-t-u-1}-\eqref{RhD-partial_tu-q} with \eqref{def-u-hash}-\eqref{def-u-hash-p} we observe
	\begin{align}
		&a_s (( \u - R_h\u) (0), \bv_h)  + ((\u - R_h\u) (0), \bv_h)_{\Sigma} \nn\\
		&\,\,\, +a_f(( \partial_t \u - R_h^D\partial_t \u) (0) - \lambda (R_{{sh}}\u (0))\u_h^{\#}, \bv_h)- b (( \partial_t p - R_h^D\partial_t p) (0)-\lambda(R_{{sh}} \u (0)) p_h^{\#}, \bv_h)\nn\\
		&\,\,\, +(( \partial_t \u - R^D_h\partial_t \u)(0) - \lambda (R_{{sh}} \u (0)) \u_h^{\#}, \bv_h) = 0 \quad \forall \bv_h \in \X^r_h \label{modified-RhD}\\
		&b \left(q_{h},(\partial_t \u - R^D_h \partial_t \u)(0) -
		\lambda (R_{{sh}} u (0)) \u_h^{\#} \right) = 0 \quad\,\,\,\,\,\, \forall q_h \in Q^{r-1}_h \label{modified-RhD-p}
	\end{align}
	By comparing \eqref{modified-RhD}-\eqref{modified-RhD-p} with \eqref{evolution-ritz-t1-u-eta}-\eqref{evolution-ritz-t1-p} we find the following relations: 
	\begin{align*} 
	\partial_t ( \u - R_h\u) (0) &= \left( { \partial_t \u - R^D_h} \partial_t \u
	\right) (0) - \lambda (R_{{sh}} \u (0)) \u_h^{\#} , \\
	 \partial_t ( p - R_h p) (0) &= ( \partial_t p - R_h^D\partial_t p) (0) -
	\lambda (R_{{sh}} \u (0)) p_h^{\#} . 
	\end{align*}
	Since $|\lambda (R_{{sh}}\u(0))|\leq Ch^{r+1}$ and $\|\u_h^\#\|+\|\u_h^\#\|_\Sigma+\|p^\#_h\|\leq C$, the result of this lemma follows from the estimates of the Dirichlet Stokes--Ritz projection in Lemma \ref{lemma:Dirichlet-Ritz} (with $\u$ and $p$ replaced by $\partial_t\u$ and $\partial_tp$ therein).
\hfill \end{proof}

%
%
\subsection{Error estimates of the coupled Ritz projection for $t>0$} 
In this subsection, using the results in subsection \ref{t:0 error-estimate}, we mainly present the proof of the $H^1$-error estimates \eqref{H1-norm-error2} and $L^2$-error estimates results in Theorem \ref{main-3}.

We first present $H^1$-norm error estimates for the the coupled Ritz projection by employing the auxiliary Ritz projections $R_h^S$ and $R_h^D$ defined in \eqref{RhS-def} and \eqref{RhD-def}, respectively. From \eqref{RhD-def2} we see that $$ R_h^D\u-R_h^S\u = \widetilde{P} R_h^S\u - R_h^S\u = - \lambda (R_h^S\u) \n_h \quad\mbox{with}\,\,\, \lambda (R_h^S\u)\in\mathbb{R}, $$ where the last equality follows from relation \eqref{P-tilde-formula0}. Therefore, with the relation above we have
\begin{align}
	&a_s(\u-R_h^D\u,\bv_h)+(\u-R_h^D\u,\bv_h)_\Sigma\nonumber\\
	=&\, a_s(\u-R_h^S\u,\bv_h)+(\u-R_h^S\u,\bv_h)_\Sigma + \lambda(R_h^S\u)\left(a_s(\n_h,\bv_h)+(\n_h,\bv_h)_\Sigma\right)\nonumber\\
	\lesssim&\, Ch^{r+1}\|\bv_h\|_{H^1(\Sigma)}\leq Ch^{r+1/2}\|\bv_h\|_{H^{1/2}(\Sigma)}\leq Ch^{r+1/2}\|\bv_h\|_{H^1} \quad \forall\bv_h\in \X_h^r,
\end{align}
where we have used the inverse inequality in \eqref{inverse-Sh} and the trace inequality in the derivation of the last two inequalities. Moreover, since the auxiliary Ritz projection $R_h^D$ defined in \eqref{RhD-def} is time-independent, it follows that $(\partial_t R_h^D u,\partial_t R_h^D p) = (R_h^D \partial_t u, R_h^D\partial_t  p)$. Therefore, in view of estimate \eqref{RhD-error} for the Dirichlet Stokes--Ritz projection, the following estimate can be found: 
\begin{align}\label{RhD-use-1} 
	&a_s ( \u - R_h^D\u, \bv_h)  + ( \u - R_h^D\u, \bv_h)_{\Sigma} + a_f
	(\partial_t ( \u - R_h^D\u), \bv_h)\nn\\
	& - b (\partial_t ( p - R_h^Dp), \bv_h) + 
	(\partial_t ( \u - R_h^D\u), \bv_h) \lesssim C h^r \| \bv_h \|_{H^1}\quad\forall\bv_h\in \X_h^r.
\end{align}
By considering the difference between \eqref{evolution-ritz-t1-u-eta} and \eqref{RhD-use-1}, we can derive the following inequality: 
	\begin{align}\label{L2H1-estimate-tmp1} 
		&a_s (R_h \u - R_h^D \u, \bv_h)  + (R_h \u - R_h^D \u, \bv_h)_{\Sigma} + a_f
		(\partial_t (R_h \u - R_h^D \u), \bv_h)\nn\\
		& - b (\partial_t (R_h p - R_h^D p), \bv_h) +
		(\partial_t (R_h \u - R_h^D \u), \bv_h) \lesssim C h^r \| \bv_h \|_{H^1} \quad \forall \bv_h\in \X_h^r. 
	\end{align} 
Then, choosing $\bv_h = \partial_t (R_h \u -  R_h^D \u)$ in \eqref{L2H1-estimate-tmp1} and using relation $b(\partial_t (R_h p - R_h^D p), \partial_t (R_h \u - R_h^D \u))=0 $ (which follows from \eqref{full-div-free} and \eqref{evolution-ritz-t1-p}),  
using Young's inequality 
$$
Ch^r \|\partial_t(R_h\u-R_h^D\u) \|_{H^1} 
\le C\varepsilon^{-1} h^{2r} + \varepsilon\|\partial_t(R_h\u-R_h^D\u) \|_{H^1}^2 
$$
with a small constant $\varepsilon$ so that $\varepsilon\|\partial_t(R_h\u-R_h^D\u) \|_{H^1}^2$ can be absorbed by the left hand side of \eqref{L2H1-estimate-tmp1}, we obtain 
	\begin{align}\label{Rhu-Linfty-H1-Sigma}
		&\| R_h \u - R_h^D \u \|_{L^\infty H^1(\Sigma)} + \| \partial_t ( R_h \u - R_h^D \u) \| _{L^2 H^1} \nn\\
		&\leq C h^r + C\|
		(R_h \u - R_h^D \u) (0) \|_{s}  +C \| (R_h \u - R_h^D \u) (0) \|_{\Sigma}  \leq
		C h^r , 
	\end{align}
where the last inequality uses the estimates in Lemma \ref{lemma:error-Rh-t=0} and Lemma \ref{lemma:Dirichlet-Ritz}.  Then, by applying the inf-sup condition in \eqref{inf-sup-condition} (which involves $\| \bv_h \|_{H^1(\Sigma)}$ in the denominator), we can obtain the following estimate from \eqref{L2H1-estimate-tmp1}:
\begin{equation}\label{dt-Rhp-error}
	\| \partial_t (R_h p - R_h^D p) \| \leq C\| R_h \u - R_h^D \u \|_{H^1(\Sigma)}+C\|\partial_t( R_h \u - R_h^D \u)\|_{H^1}+Ch^r,
\end{equation}
which combined with the estimate in \eqref{Rhu-Linfty-H1-Sigma}, leads to the following estimate: 
\begin{align} \label{dt-Rhp-error-L2L2}
\| \partial_t (R_h p - R_h^D p) \|_{L^2 L^2} \leq C h^r . 
\end{align}
Therefore, using an additional triangle inequality, the estimates in \eqref{Rhu-Linfty-H1-Sigma}--\eqref{dt-Rhp-error-L2L2} can be written as follows: 
\begin{align} 
	\| \partial_t(R_h \u - \u)\|_{L^2 H^1} +  \| R_h \u - \u \|_{L^\infty H^1 (\Sigma)} + \|\partial_t( R_h p - p) \|_{L^2 L^2} \leq C h^r . 
	\label{H1-norm-error1} 
\end{align} 
With the initial estimates in Lemma \ref{lemma:error-Rh-t=0}, the estimate of $\| \partial_t ( R_h \u - \u) \| _{L^2 H^1}$ above further implies that 
\begin{align}\label{Rhu-Ihu-Linfty-H1}
	\| R_h \u - \u \|_{L^\infty H^1} \leq \| (R_h \u - \u) (0)
	\|_{H^1} + C\| \partial_t ( R_h \u - \u) \| _{L^2 H^1} \leq C h^r \, . 
\end{align}
Since $ \partial_t (R_h \e - \e) = R_h \u - \u$ on the boundary $\Sigma$, by using the Newton--Leibniz formula with respect to $t\in[0,T]$, the estimate in \eqref{H1-norm-error1} and initial estimates in Lemma \ref{lemma:error-Rh-t=0}, we have 
\begin{align}\label{Rheta-Iheta-Linfty-H1}
	\| R_h \e - \e \|_{L^\infty H^1 (\Sigma)} 
	& \leq \| ( R_h \e - \e)
	(0) \|_{H^1(\Sigma)} + C \| \partial_t (R_h \e -  \e) \|_{L^2 H^1 (\Sigma)} \notag\\
	& \leq \| ( R_h \e - \e)
	(0) \|_{H^1(\Sigma)} + C \| R_h \u - \u \|_{L^2 H^1 (\Sigma)} 
	\leq C h^r . 
\end{align}
In the same way, from \eqref{H1-norm-error1} and initial estimates in Lemma \ref{lemma:error-Rh-t=0} we have
\begin{align}\label{Linfty-L2-Rhp}
	\| R_h p - p \|_{L^\infty L^2} 
	&\leq C\| (R_h p - p)(0)\| + C\|R_h p - \u \|_{L^2 L^2} \leq Ch^r.
\end{align}
Moreover, by differentiating  \eqref{evolution-ritz-t1} with respect to time, we have
	\begin{subequations}\label{evolution-ritz-3} 
		\begin{align}
			a_s& (\partial_t (R_h \u - \u), \bv_h)   + (\partial_t (R_h \u - \u),
			\bv_h)_{\Sigma} + a_f (\partial^2_t (R_h \u - \u), \bv_h)  \nn\\
			& - b (\partial_t^2 (R_h
			p - p), \bv_h) + (\partial_t^2 (R_h \u - \u), \bv_h) = 0 && \forall \bv_h \in\X_h^r , 
			 \label{evolution-ritz-3-u-eta} \\ 
			 b&(q_h,\partial_{t}^2(R_h\u-\u)) =0&& \forall q_h\in Q_h^{r-1} .\label{evolution-ritz-3-p}
		\end{align}
	\end{subequations} 
Similarly, by choosing $\bv_h = \partial_t^2 (R_h \u - R_h^D \u)$ in \eqref{evolution-ritz-3-u-eta} and using the same approach as above with the initial value estimates in \eqref{u-t-Linfty-H1-error}, we can obtain the following estimate (the details are omitted): 
\begin{align} 
 \| \partial_t (R_h \u -  \u) \|_{L^\infty H^1} 
  +  \| \partial_t (R_h \u -  \u) \|_{L^\infty H^1(\Sigma)}   + \| \partial_t (R_h p - p) \|_{L^\infty L^2} &\nn \\
 + \| \partial_t^2 ( R_h \u - \u) \|_{L^2 H^1}+\|\partial_t^2(R_hp-p)\|_{L^2L^2}
& \leq C h^r \, . 
\label{H1-norm-error2} 
\end{align} 
This establishes the $H^1$-norm error estimates for the coupled non-stationary Ritz projection defined in \eqref{evolution-ritz}. 

We then present $L^2$-norm error estimates for the coupled non-stationary Ritz projection. 
To this end, we introduce the following dual problem: 
	\begin{subequations}\label{dual-problem-Ritz}
	\begin{align}
	- \mathcal{L}_s {\bb \phi} + {\bb \phi} &=  \partial_t \sig ({\bb \phi}, q)\n + {\bf f}
			 && \text{in } 
			\Sigma\label{dual-phi-1}\\
		-\nabla \cdot \sigma ({\bb \phi}, q) + {\bb \phi} &=  0 && \text{in } \Omega\label{dual-phi-2}\\
		\nabla \cdot {\bb \phi} &=  0 && \text{in } \Omega , 
		\label{dual-phi-3}
	\end{align} 
	\end{subequations}
with the initial condition $\sig({\bb \phi},q)\n =0$ at $t=T$. Problem \eqref{dual-problem-Ritz} can be equivalently written as a backward evolution equation of ${\bb \xi}=\sig ({\bb \phi}, q) \n$, i.e., 
	\begin{equation}\label{xi-strong-eq-1}
		- \mathcal{L}_s \mathcal{N} {\bb \xi} + \mathcal{N} {\bb \xi} - \partial_t \mathcal{}
		{\bb \xi} = \f \,\,\, \text{on}\,\,\, \Sigma\times[0,T), \,\,\,\mbox{with initial condition}\,\,\, {\bb \xi} (T) = 0 , 
	\end{equation}
where $\mathcal{N}: H^{-\frac12}(\Sigma)^d\rightarrow H^{\frac12}(\Sigma)^d$ is the Neumann-to-Dirichlet map associated to the Stokes equations. 
The existence, uniqueness and regularity of solutions to \eqref{dual-problem-Ritz} are presented in the following lemma, for which the proof is given in the Appendix A 
 by utilizing and analyzing \eqref{xi-strong-eq-1}.

\begin{lemma}\label{dual-reg} 
Problem \eqref{dual-problem-Ritz} has a unique solution which satisfies the following estimate: 
\begin{equation}\label{dual-reg-estimate}
		\| {\bb \phi} \|_{L^2 H^2} + \| {\bb \phi} \|_{L^2 H^2 (\Sigma)} + \| q \|_{L^2 H^1} + \| \sig ({\bb \phi}, q) (0) \n \|_{\Sigma}\leq C\| {\bf f} 
		 \|_{L^2 L^2 (\Sigma)}.
	\end{equation}
\end{lemma}

	
By choosing ${\bf f} = R_h \e - \e$ and, testing equations \eqref{dual-phi-1} and \eqref{dual-phi-2} with $R_h \e - \e$ and $R_h\u-\u$, respectively, and using relation $\partial_t(R_h \e - \e)=R_h\u-\u$ on $\Sigma$, we have
	\begin{align*}
		&a_s ({\bb \phi}, R_h \e - \e)  + ({\bb \phi}, R_h \e - \e)_{\Sigma} + a_f
		({\bb \phi}, R_h \u - \u) - b (q, R_h \u - \u) + ({\bb \phi}, R_h \u - \u)\\
		&= \frac{d}{d t}
		(\sig ({\bb \phi}, q)\cdot \n, R_h \e - \e)_{\Sigma} + \| R_h \e - \e
		\|_{\Sigma}^2.
	\end{align*}
In view of the definition of the non-stationary Ritz projection in \eqref{evolution-ritz}, we can subtract $I_h{\bb\phi}$ from ${\bb\phi}$ in the inequality above by generating an additional remainder $b(R_hp-p,{\bb \phi}-I_h{\bb \phi})$. This leads to the following result in view of the estimate in \eqref{H1-norm-error1}: 
	\begin{align} 
	& \frac{d}{d t} (\sig ({\bb \phi}, q) \n, R_h \e -  \e)_{\Sigma} + \| R_h
	\e - \e \|_{\Sigma}^2 
	 = a_s ( {\bb \phi} - I_h {\bb \phi} , R_h \e - \e)  + ({\bb \phi} - I_h {\bb \phi}, R_h \e -  \e)_{\Sigma}  
	\nn \\ 
	&\quad\, + a_f
		({\bb \phi}-I_h{\bb \phi}, R_h \u - \u) - b (q - I_h q, R_h \u - \u) + ({\bb \phi}-I_h{\bb \phi}, R_h \u - \u)-b(R_hp-p,{\bb \phi}-I_h{\bb \phi})
\nn \\ 	
	& \le C h^{r + 1} (\| {\bb \phi} \|_{H^2} + \| {\bb \phi} \|_{H^2 (\Sigma)} + \| q \|_{H^1})  .
	\nn
		\end{align} 
Since $\| ( R_h \e - \e) (0) \|_{\Sigma} \le C h^{r+1}$ (see Lemma \ref{lemma:error-Rh-t=0}), the inequality above leads to the following result: 
\begin{align*} 
&\| R_h \e - \e \|_{L^2 L^2 (\Sigma)}^2 \\
&\leq 
 Ch^{r + 1}  \| R_h \e - \e \|_{L^2 L^2 (\Sigma)}
 +  \| R_h \e(0) - \e(0) \|_{L^2 (\Sigma)} \| (\sig ({\bb \phi}, q) \n) (0) \|_{L^2(\Sigma)} \\
&\leq 
 Ch^{r + 1}  \| R_h \e - \e \|_{L^2 L^2 (\Sigma)}
 + C h^{r+1} \| R_h \e - \e \|_{L^2 L^2 (\Sigma)} , 
\end{align*} 
and therefore 
\begin{align} 
 \| R_h \e - \e \|_{L^2 L^2 (\Sigma)} \leq Ch^{r + 1} . 
 \label{L22-1} 
\end{align} 

By using the same approach, choosing ${\bf f} = R_h \u - \u$ and ${\bf f} = \partial_t(R_h \u - \u)$ in \eqref{dual-phi-1}, respectively, the following result can be shown (the details are omitted):  
\begin{align} 
\| R_h \u - \u \|_{L^2 L^2 (\Sigma)} + 
\| \partial_t ( R_h \u - \u) \|_{L^2 L^2 (\Sigma)} \leq C h^{r +1} . 
 \label{L22-2} 
 \end{align} 
This also implies, via the Newton--Leibniz formula in time,   
\begin{align} \label{L-infty-2}
\|  R_h \u - \u \|_{L^\infty L^2 (\Sigma)} \leq C h^{r +1} \, .
\end{align} 
Furthermore, we consider a dual problem defined by 

	\begin{equation}\label{fluid-dual-final}
		\left\{\begin{aligned}
			&- \nabla \cdot \sigma ({\bb \phi}, q) + {\bb \phi} = R_h \u - \u && \text{in }\Omega\\[3pt] 
			&\nabla \cdot {\bb \phi} = 0 && \text{in } \Omega\\
			&{\bb \phi}|_{\Sigma} = 0,\quad q\in L^2_0(\Omega) , 
		\end{aligned}\right.
	\end{equation}
which satisfies the following standard $H^2$ regularity estimate 
	\[ \| {\bb \phi} \|_{H^2} + \| q \|_{H^1} + \| \sigma({\bb \phi} ,q) \n \|_{L^2(\Sigma)}  \leq C \| R_h \u - \u\| , 
	\]
where the term $\| \sigma({\bb \phi} ,q) \n \|_{L^2(\Sigma)} $ is included on the left-hand side because it is actually bounded by $\| {\bb \phi} \|_{H^2} + \| q \|_{H^1}$. Then, testing \eqref{fluid-dual-final} with $R_h \u - \u$, we have 
	\begin{align*}
		&\| R_h \u - \u \|^2 \\
		& = a_f ({\bb \phi}, R_h \u - \u) - b (q, R_h \u - \u) + ({\bb \phi},R_h \u - \u) - (\sig ({\bb \phi}, q) \n, R_h \u - \u)_{\Sigma}\\
		& =  a_f ({\bb \phi} - I_h {\bb \phi}, R_h \u - \u) - b (q - I_h q, R_h \u - \u) -
		(\sig ({\bb \phi}, q) \n, R_h \u - \u)_{\Sigma}\\
		 &\quad +({\bb \phi}-I_h{\bb \phi},R_h \u - \u) - b (R_h p - p,{\bb \phi}- I_h  {\bb \phi}) \quad \mbox{(as a result of \eqref{evolution-ritz} with $\bv_h=I_h{\bb \phi},q_h=I_hq$)} \\
		& \le  C h (\| {\bb \phi} \|_{H^2} +\|q\|_{H^1}) (\| R_h \u
		- \u \|_{H^1} +\|R_h p - p\|)\\
		& \quad + \| \sig({\bb \phi}, q)\cdot \n \|_{\Sigma} \| R_h \u - \u \|_{\Sigma}\\
		& \le C h^{r + 1} \| R_h \u - \u \| + C \| R_h \u -\u \| \| R_h \u - \u \|_{\Sigma} . 
	\end{align*}
The last inequality implies, in combination with \refe{L-infty-2}, the following result: 
	\begin{align} 
	\| R_h \u - \u \| \leq Ch^{r + 1}\, . 
	\label{Linfty-L2}
	\end{align} 
By using the same approach, replacing $R_h\u-\u$ by $\partial_t (R_h\u-\u)$ in \eqref{fluid-dual-final}, the following estimate can be shown (the details are omitted): 
	\begin{equation}
	\| \partial_t(R_h \u - \u) \|_{L^2L^2} \leq Ch^{r + 1}\, .
	\end{equation}
The proof of Theorem \ref{main-3} is complete. 
\hfill\endproof


%
%
%
\section{Numerical examples}
In this section, we present numerical tests to support the theoretical analysis in this article and to show the effectivenss of the proposed algorithm. For  2D numerical examples, the operator ${\cal L}_s\e = C_0 \partial_{xx} \e-  C_1 \e$ on the interface $\Sigma$ is considered. 
All computations are performed by the finite element package NGSolve; see \cite{Ngs}. 
\vskip0.1in 

\begin{example}\upshape 
To test the convergence rate of the algorithm, we consider an artificial example of a two-dimensional thin structure model  given in \refe{f-e}--\refe{s-e} with extra sources such that the exact solution is given by
\begin{align}\label{exact-solutions}
u_1 &= 4\sin(2\pi x)\sin(2 \pi y)\sin(t),\nn 
\\
u_2 &= 4(\cos(2\pi x)\cos(2 \pi y))\sin(t),\nn 
\\
p &=  8(\cos(4\pi x) - \cos(4\pi y))\sin(t),\nn 
\\
\eta_1 &= 0, \,\,\,\,\,\, \eta_2 = -4\cos(2 \pi x)\cos(t).\nn 
\end{align}

First, we examine a problem involving left/right-side periodic boundary conditions and top/bottom interfaces on the domain $\Omega = [0,2] \times [0,1]$. A uniform triangular partition is employed, featuring $M+1$ vertices in the $y$-direction and $2M+1$ vertices in the $x$-direction, where $h = 1/M$. The classical lowest-order Taylor--Hood element is utilized for spatial discretization. To verify the $L^2$-norm error estimates, we set all involved parameters to $1$. Our algorithm is applied to solve the system with $M = 8, 16, 32$, $\tau = h^3$, and the terminal time $T = 0.1$. The numerical results are presented in Table \ref{table1}, which shows that the algorithm has third-order accuracy for velocity and displacement in the $L^2$-norm, as well as second-order accuracy for pressure in the $L^2$-norm and displacement in the energy-norm. These numerical results align with our theoretical analysis.

\setcounter{table}{0}
\begin{table}[h!]
    \caption{The convergence order of the algorithm under periodic boundary conditions}
    \centering
    \begin{tabular}{lcccc} 
  \hline 
     Taylor--Hood elements ($ \tau=h^3$)  & $\|\u^N-\u_h^N\| $  
     & $ \|p^N-p_h^N \| $ 
     &  $\|\e^N-\e_h^N\|_\Sigma $   
     & $\| \e^N-\e_h^N \|_s$   \\
     \hline 
    $h = 1/8$      &6.852e-3  & 1.403e-1 & 1.324e-2  & 8.075e-1  \\
    $h = 1/16$    &6.848e-4  & 2.691e-2 & 1.644e-3  & 2.029e-1  \\
    $h = 1/32$     &7.937e-5 & 6.297e-3 & 2.052e-4  & 5.079e-2   \\
    \hline 
 \mbox{order} & 3.10       & 2.10  & 3.00  & 2.00 \\ 
 \hline 
    \end{tabular} 
    \label{table1}
\end{table} 

Next, we test our algorithm for a model with Dirichlet boundary condition on the left and right boundaries, using the same configuration as previously described. Both the lowest-order Taylor--Hood element and the MINI element are employed for spatial discretization. For the Taylor--Hood element and the MINI element, we adopt $\tau = h^3$ and $\tau = h^2$ in the computation, respectively. The numerical results are displayed in Table \ref{table2}. As observed in Table \ref{table2}, the algorithm, when paired with both the Taylor--Hood element and the MINI element, yields numerical results exhibiting optimal convergence orders for $\u$ and $\e$. 

\begin{table}[h!]
    \caption{The convergence order of the algorithm under Dirichlet boundary conditions}
    \centering
    \begin{tabular}{lcccc} 
\hline 
     Taylor--Hood elements ($ \tau =h^3 $) & $\|u^N-u_h^N\|$ & $\|p^N-p_h^N\|$ 
     &  $\| \e^N-\e_h^N \|_\Sigma$   & $\| \e^N-\e_h^N \|_s$ \\
   \hline 
    $h = 1/8$\,\,\,    &4.553e-3  &  1.354e-1 & 1.313e-2 & 8.069e-1 \\
    $h = 1/16$     &6.009e-4  & 2.775e-2  & 1.645e-3 & 2.029e-1 \\ 
    $h = 1/32$     &7.693e-5  & 6.470e-3  &  2.055e-4 & 5.079e-2 \\ 
    \hline 
 \mbox{order} & 2.97 & 2.10 & 3.00 & 2.00 \\ \hline \hline 
  MINI elements ($ \tau =h^2 $) & $\|u^N-u_h^N\|$ & $\|p^N-p_h^N\|$ 
     &  $\| \e^N-\e_h^N \|_\Sigma$   & $\| \e^N-\e_h^N \|_s$ \\
   \hline 
    $h = 1/16$    &1.324e-2  & 3.186e-1 & 7.971e-2 & 4.001e0 \\
    $h = 1/32$     &3.349e-3  &  1.192e-1  & 1.999e-2 & 2.003e0 \\ 
    $h = 1/64$     &8.327e-4  &  4.641e-2  & 5.001e-3 & 1.002e0 \\ 
    \hline 
 \mbox{order} & 2.00 & 1.36 & 2.00 & 1.00 \\ \hline 
  \end{tabular} 
  \label{table2}
\end{table} 

\end{example}

\begin{example}\upshape 
We consider a benchmark model which was studied 
by many people \cite{BM-2016SINUM,  BukacT-2022-thin, Fer-2013NumMath, Formaggia-CMAME-2001, GGCC-2009, LiuJ-JCP-2014, Bukac2018SINUM}. All the quantities will be given in the CGS system of units \cite{Fer-2013NumMath}.
The model is described by \refe{f-e}--\refe{s-e} in $\Omega = (0,5)\times (0, 0.5)$ 
with the physical parameters: 
fluid density $\rho_f = 1$, fluid viscosity $\mu = 0.035$, solid density $\rho_s = 1.1$, the thickness of wall $\epsilon_s = 0.1$, Young's modulus $E = 0.75\times 10^6$, Poisson's ratio $\sigma = 0.5$
and 
\begin{align} 
C_0 = \frac{E\epsilon_s}{2(1+\sigma)},\quad
C_1 = \frac{E\epsilon_s}{R^2(1-\sigma^2)} , 
\nn 
\end{align} 
where $R = 0.5$ is the width of the domain $\Omega$.
The boundary conditions on the in/out-flow sides ($x=0, x=5$) are defined by $\sigma(\u,p)\n = -p_{\rm in/out}\n$
where 
$$
p_{\text {in }}(t)=\left\{\begin{array}{ll}
\dfrac{p_{\max }}{2}\Big[1-\cos \Big(\dfrac{2 \pi t}{t_{\max }}\Big)\Big] & \text { if } t \leq t_{\max } \\
0 & \text { if } t>t_{\text {max }}
\end{array} \quad, \quad p_{\text {out }}(t)=0 \,\,\, \forall t \in(0, T)\right. .
$$
with  $p_{\rm max} = 1.3333\times 10^4$ and $t_{\rm max} = 0.003$. The top and bottom sides of $\Omega$ are thin structures, and the fluid is initially at rest. We take a uniform triangular partition with $M+1$ vertices in $y$-direction and $10M+1$ vertices in $x$-direction ($h=1/M$), and solve the system by our algorithm where the lowest-order 
Taylor--Hood finite element approximation is used with the spatial mesh size $h=1/64$ ($M=64$), the temporal step size $\tau =h^3$ and the parameter $\beta = 0.5$. 
We present the contour of pressure $p$ in Figure \ref{figpressure} at $t=0.003, 0.009, 0.016, 0.026$  (from top to bottom).
We can see a forward moving pressure wave(red), which reaches the right-end of the domain and gets reflected. The reflected wave is characterized by the different color(blue) of the pressure, which was also observed in \cite{Fer-2013NumMath, Formaggia-CMAME-2001, GGCC-2009}. 
\vskip0.1in 
\begin{figure}[htp!]
\centering
\includegraphics[scale=0.085]{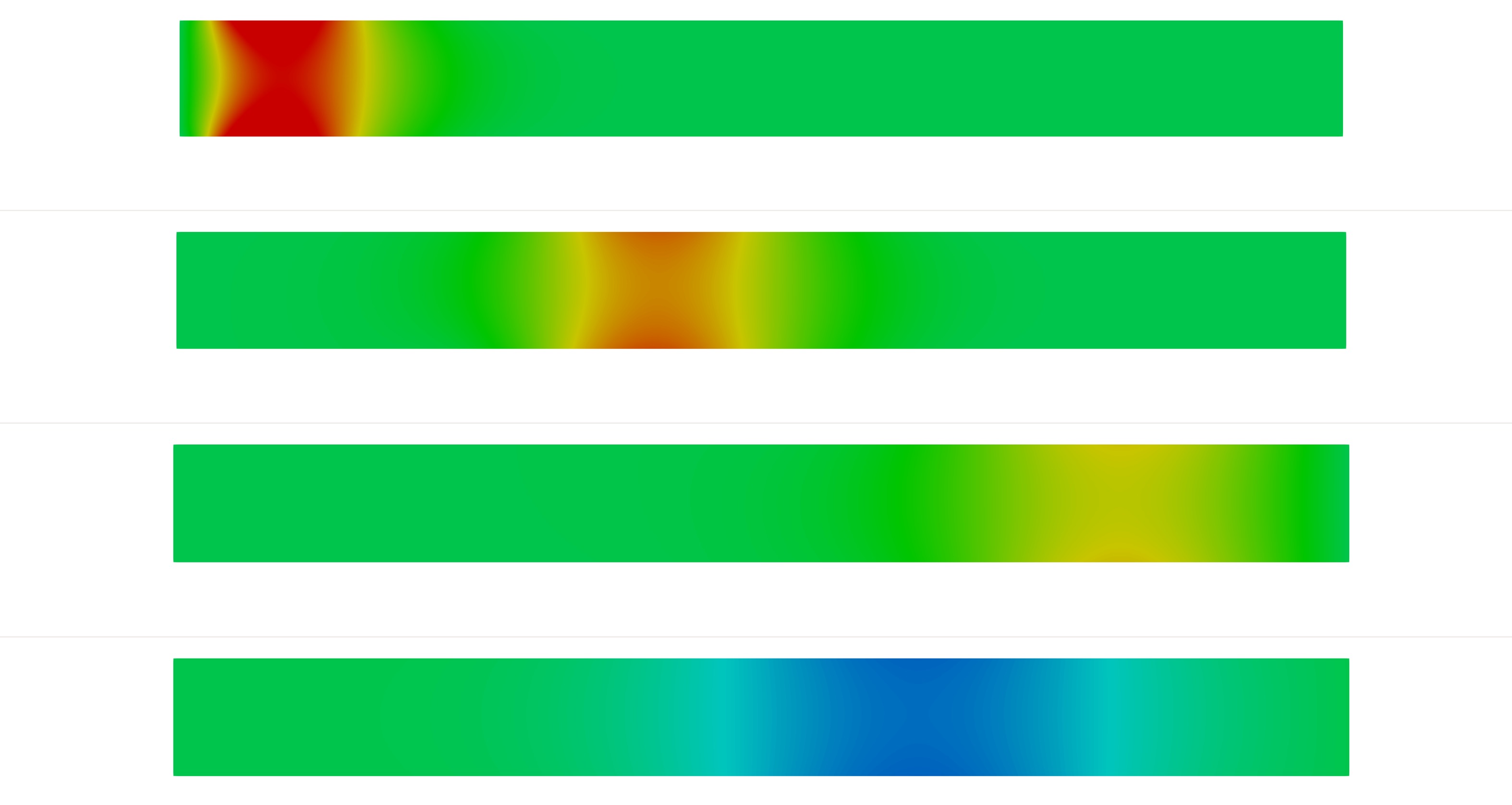}
\caption{The contour of the pressure when $t=0.003, 0.009, 0.016, 0.026$  (from top to bottom)}
\label{figpressure}
\end{figure}

\end{example}

\vskip-0.2in

\begin{example}\upshape 

Finally, we consider the 3D blood flow simulation in common carotid arteries studied in 
\cite{Bukac2018SINUM}. The blood flow is modeled by the Navier-Stokes equation, while  the analysis was presented only for 
the model with Stokes equation. 
The weak form of the arterial wall model is:
\begin{align*}
\rho_s\epsilon_s (\e_{tt},\w)_{\Sigma} + D_1(\e,\w)_{\Sigma} +D_2(\e_t,\w)_{\Sigma} + \epsilon_s(\boldsymbol{\Pi}_s(\e), \nabla_{s}\w)_{\Sigma} = (-\sigma(\u,p)\n,\w)_{\Sigma}
\end{align*}
for any $ \w \in \S$, where  $\nabla_s$ denote the surface gradient on the interface $\Sigma$ and 
$$
\boldsymbol{\Pi}_s({\e})=\frac{E}{1+\sigma^2} \frac{\nabla_s {\e}+\nabla_s^T {\e}}{2}+\frac{E \sigma}{1-\sigma^2} \nabla_s \cdot {\e} \boldsymbol{\mathrm{I}}
$$
for a linearly elastic isotropic structure.
The geometrical domain is a straight cylinder of length 4 cm and radius 0.3 cm, see Fugure \ref{fig_3d}. The hemodynamical parameters used in this model are given in the Table \ref{tab_3d_para}.

\begin{figure}[ht!]
\centering
\begin{tabular}{cc} 
\hskip-1in 
\includegraphics[scale=0.12]{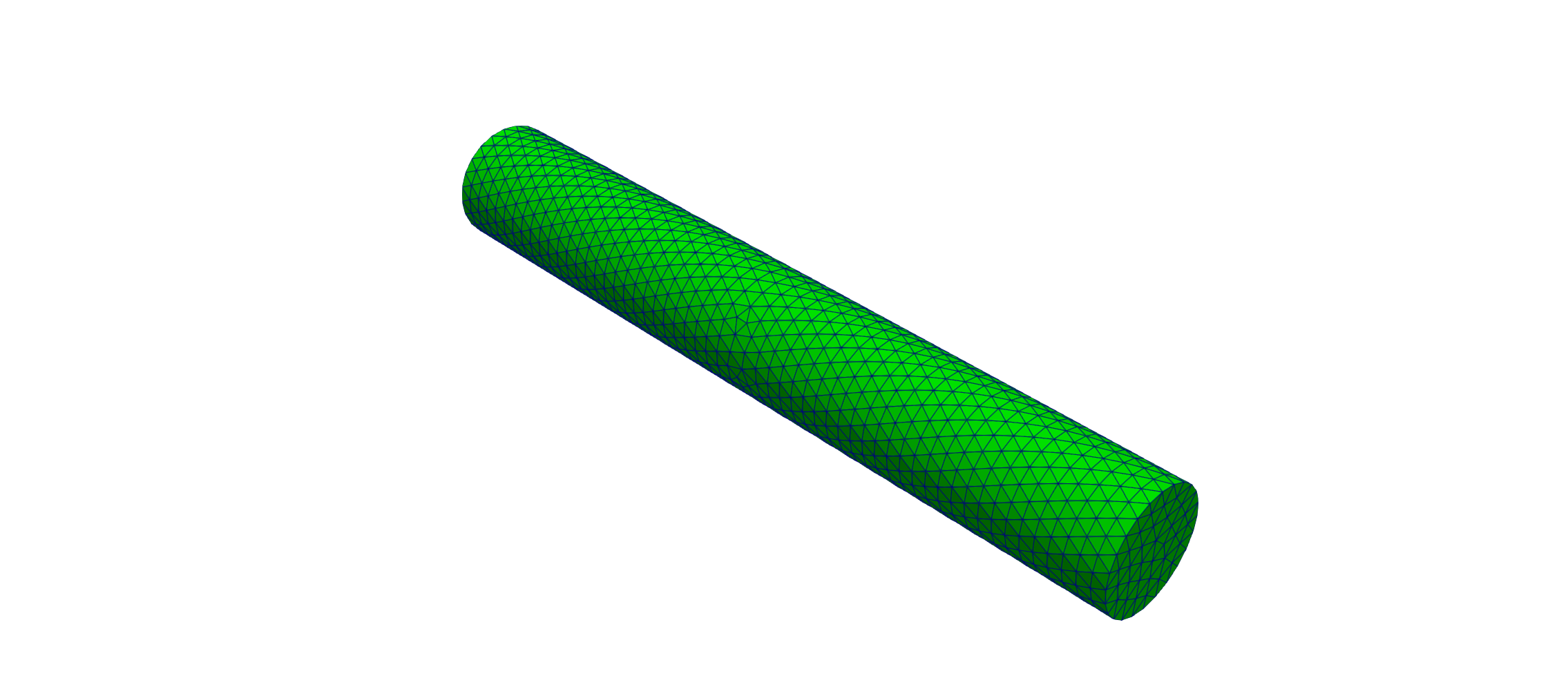} 
& 
\hskip-0.5in 
\includegraphics[scale=0.42]{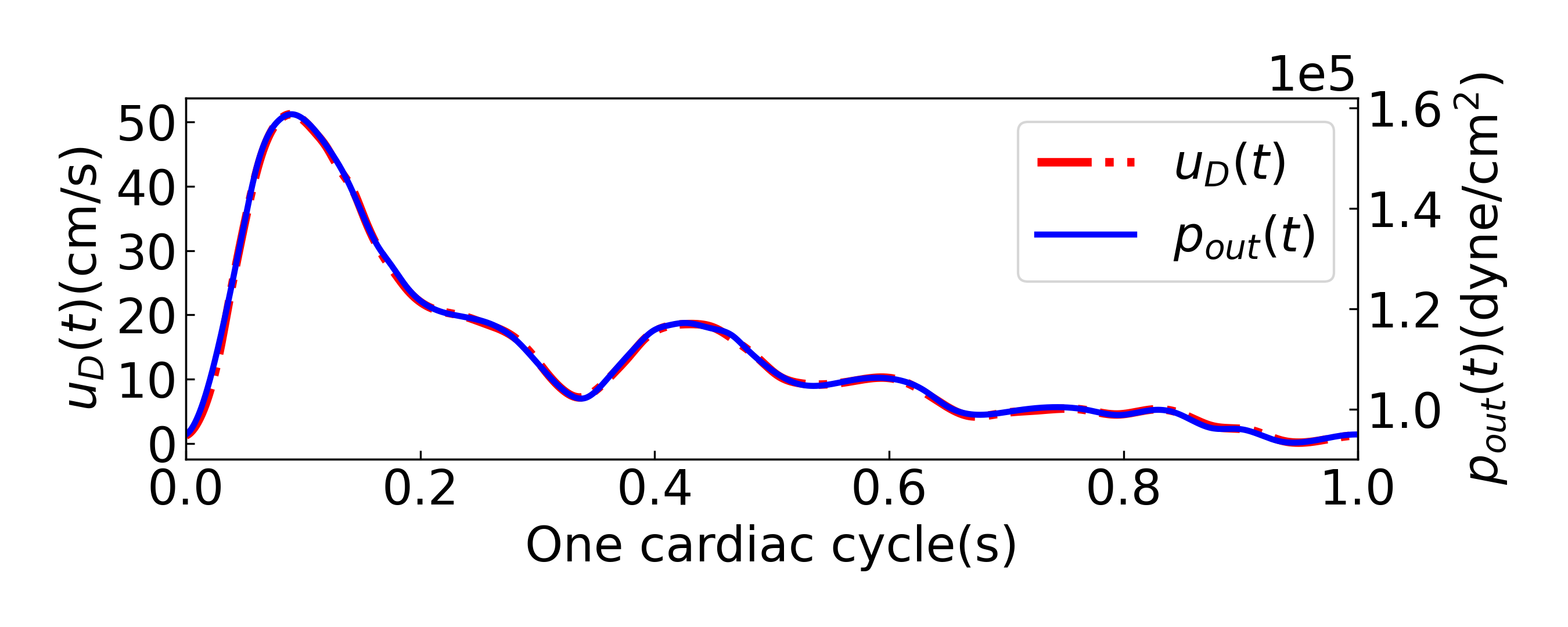}
\end{tabular}
\caption{The geometrical domain and the prescribed inlet velocity and outlet pressure}
\label{fig_3d}
\end{figure}



\begin{table}[ht!]
\centering
\caption{The hemodynamical parameters in the PDE model}
\begin{tabular}{lc|lc}
\hline 
Parameter & Value  & Parameter & Value \\
\hline 
 Wall thickness   $ \epsilon_s (\mathrm{cm})$                                          & 0.06 &
 Poisson's ratio   $\sigma$                                                                        & 0.5 \\
 Fluid viscosity    $\mu (\mathrm{g}/\mathrm{cm}\,\mathrm{s})$              & 0.04   &
 Young's modulo $E (\mathrm{dyne} / \mathrm{cm}^2) $                          & $2.6 \cdot 10^6$ \\
Fluid density       $\rho_f  (\mathrm{g} / \mathrm{cm}^3)$                         & 1 & 
Coefficient          $D_1 (\mathrm{dyne}/ \mathrm{cm}^3)$                        & $6 \cdot 10^5$ \\
Wall density        $\rho_s (\mathrm{g} / \mathrm{cm}^3)$                         & 1.1 & 
Coefficient          $D_2 (\mathrm{dyne} \, \mathrm{s} / \mathrm{cm}^3)$ & $2 \cdot 10^5$ \\
\hline
\label{tab_3d_para}
\end{tabular} 
\end{table}

For the inlet and outlet boundary conditions, we set 
\begin{align*}
\u = (u_D(t)\frac{R^2-r^2}{R^2},0,0) \,\,\,\,  \mbox{on $\Sigma_{in}$}   &&\mbox{and}  
\,\,\,\,\,\,\sigma(\u,p)\n = -p_{out}(t)\n  \,\,\,\,     \mbox{on $\Sigma_{out}$}.
\end{align*}
The waveforms of $u_D(t)$ and $p_{out}(t) $ are given in Figure \ref{fig_3d}.
The pulsatile periodic shapes of waveforms are studied  in \cite{Figueroa-CMAME-2006,Holdsworth-1999}.


The fluid mesh used in this example consists of 11745 tetrahedral elements, while the structure mesh consists of 3786 triangles. We used the $P2-P1$ finite element approximation for the velocity and pressure of the fluid, the P2 finite element approximation for the displacement of the structure. 
For comparison, both classical monolithic scheme and the proposed partitioned scheme are used for solving this example, where 
the parameter $\beta = 0.5$. The initial velocity/pressure is the smooth constant extension of the inlet/outlet boundary data at $t = 0$ for both schemes. 
The terminal time $T = 3$ s  which corresponds to 3 cardiac cycles. We have observed that the periodic pattern was established after 1 cardiac cycle. 
The comparison of the radial displacement magnitude of the artery wall obtained by our partitioned scheme and monolithic scheme at the interface point $(2,0.3,0)$ in the whole 3 cardiac cycles is shown in Figure \ref{fig_3d_displacement}. 
The axial velocity and pressure 
at the center point (2,0,0) in the third cardiac cycle are presented in Figure \ref{fig_3d_velocity}- \ref{fig_3d_press}.

\vskip0.1in 
\begin{figure}[htp!]
\centering
\includegraphics[scale=0.35]{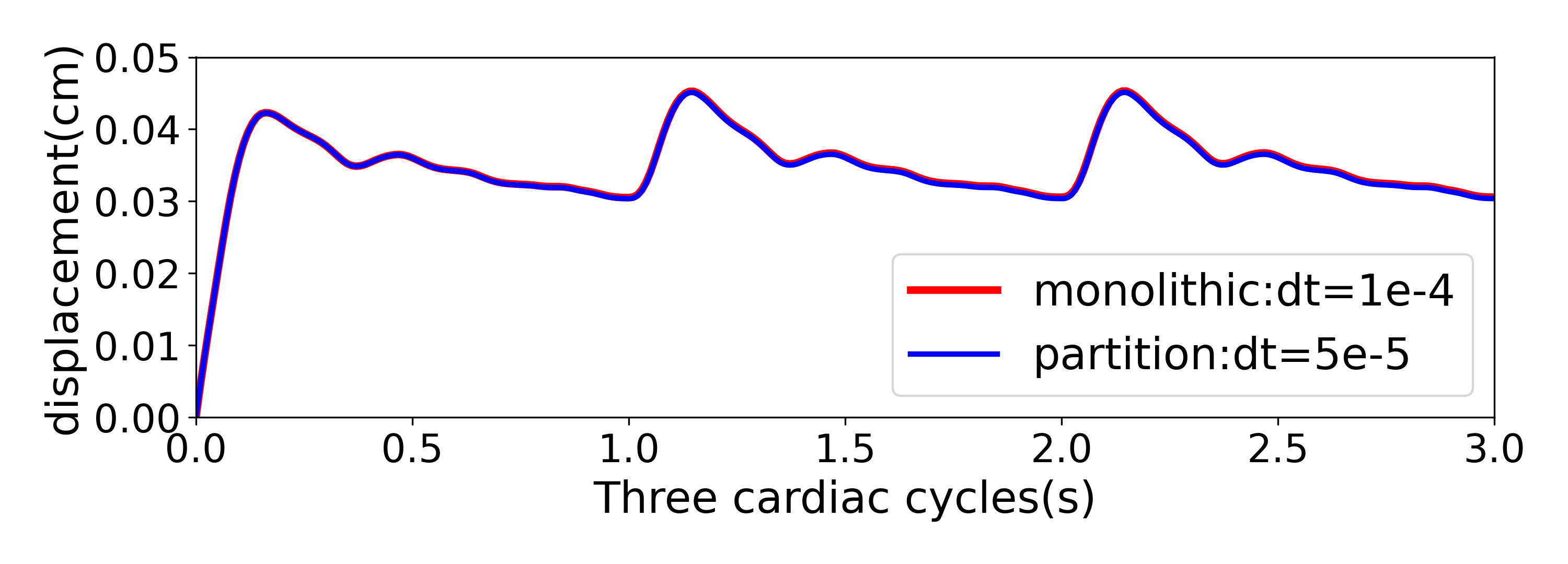}\vspace{-10pt}
\caption{Comparison of the radial displacement}
\label{fig_3d_displacement}
\end{figure}\vspace{-10pt}

\vskip0.1in 
\begin{figure}[htp!]
\centering
\includegraphics[scale=0.4]{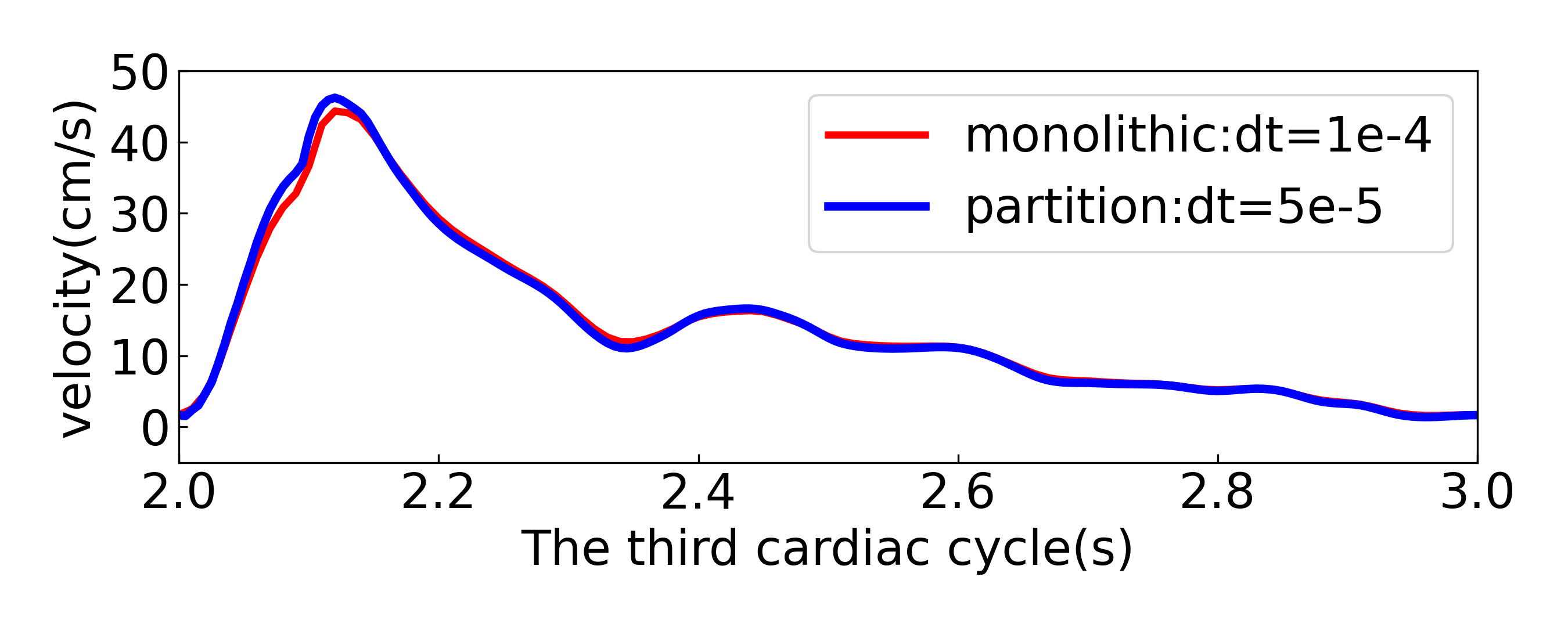}
\vspace{-10pt}
\caption{Comparison of the axial velocity}
\label{fig_3d_velocity}
\end{figure}\vspace{-10pt}

\vskip0.1in 
\begin{figure}[htp!]
\centering
\includegraphics[scale=0.4]{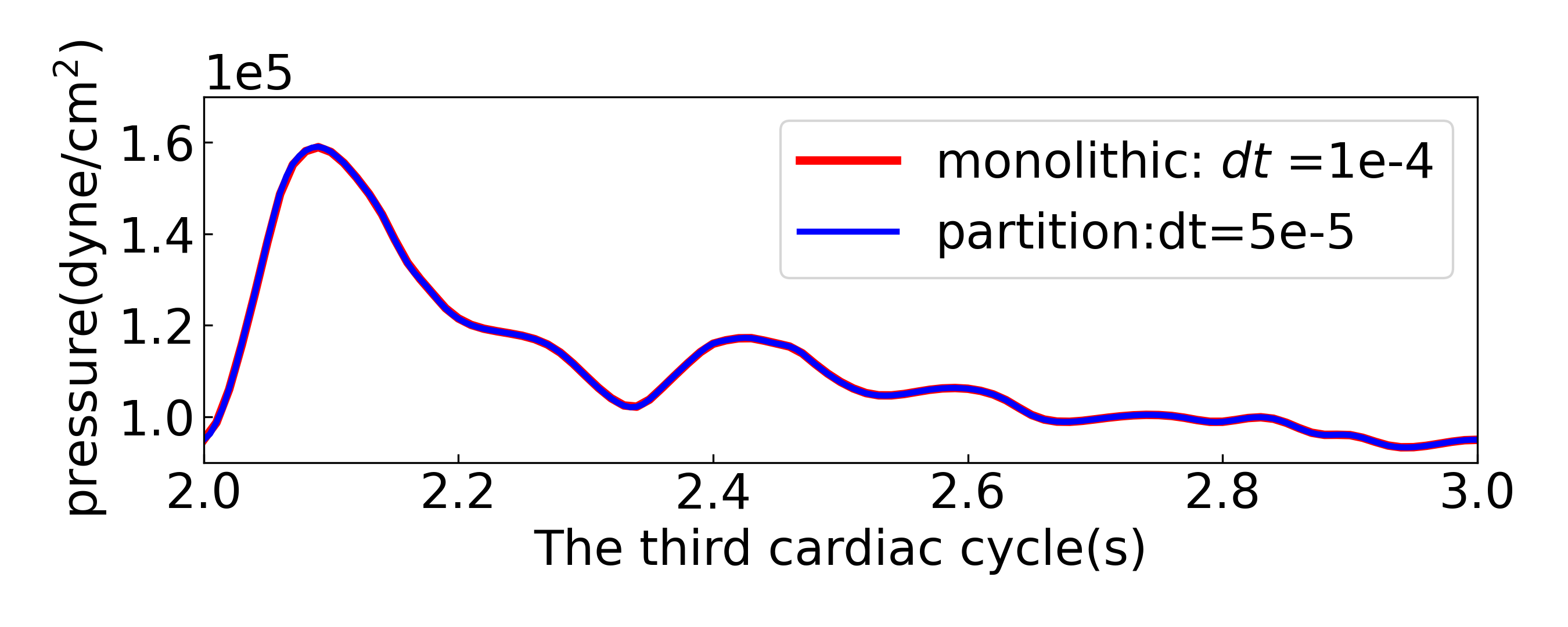}
\vspace{-10pt}
\caption{Comparison of the pressure}
\label{fig_3d_press}
\end{figure}

\end{example}

\section{Conclusion}
We have proposed a new stable fully discrete kinematically coupled scheme which decouples fluid velocity from the structure displacement for solving a thin-structure interaction problem described by \eqref{f-e}--\eqref{bc}. To the best of our knowledge, the optimal-order convergence in $L^2$ norm of spatially finite element methods for such problems has not been established in the previous works. 
Our scheme in \eqref{s1}--\eqref{s2} contains two stabilization terms 
$$
\rho_s \epsilon_s \left ( \frac{\u_h^n - \s_h^n}{\tau}, \,  \frac{\tau}{\rho_s \epsilon_s} \sig(\bv_h, q_h)\cdot \n  \right  )_\Sigma 
\quad\mbox{and}\quad 
\left ( (\sig^n_h-\sig^{n-1}_h) \cdot {\bf n}, \,  \frac{\tau(1+\beta)}{\rho_s \epsilon_s} \sig(\bv_h, q_h)\cdot \n  \right )_\Sigma 
$$
which guarantee the unconditional stability of the method, and an additional parameter $\beta>0$ which is helpful for us to prove optimal-order convergence in the $L^2$ norm for the fully discrete finite element scheme. 
Moreover, we have developed a new approach for the numerical analysis of such thin-structure interaction problems in terms of a newly introduced coupled non-stationary Ritz projection, with rigorous analysis for its approximation properties through analyzing its dual problem, which turns out to be equivalent to a backward evolution equation on the boundary $\Sigma$, i.e., 
	\begin{equation*}
		- \mathcal{L}_s \mathcal{N} {\bb \xi} + \mathcal{N} {\bb \xi} - \partial_t \mathcal{}
		{\bb \xi} = \f \,\,\, \text{on}\,\,\, \Sigma\times[0,T), \,\,\,\mbox{with initial condition}\,\,\, {\bb \xi} (T) = 0 , 
	\end{equation*}
in terms of the Neumann-to-Dirichlet map $\mathcal{N}: H^{-\frac12}(\Sigma)^d\rightarrow H^{\frac12}(\Sigma)^d$ associated to the Stokes equations. Although we have focused on the analysis for the specific kinematically coupled scheme proposed in this article for a thin-structure interaction problem, the new approach developed in this article, including the non-stationary Ritz projection and its approximation properties, may be extended to many other fully discrete monolithic and partitioned coupled algorithms and to more general fluid-structure interaction models.

\section*{Acknowledgement} 
The authors would like to thank the anonymous referees for their valuable comments and suggestions.
This work is supported in part by the NSFC key program (project no. 12231003), NSFC general program (project no. 12071020), Guangdong Provincial Key Laboratory IRADS (2022B1212010006, UIC-R0400001-22) and Guangdong Higher Education Upgrading Plan (UIC-R0400024-21), the Hong Kong Research Grants Council (GRF project no. PolyU15301321), and an internal grant of The Hong Kong Polytechnic University (Work Programme: ZVX7).

\renewcommand{\theequation}{A.\arabic{equation}}
\renewcommand{\thetheorem}{A.\arabic{theorem}}
\renewcommand{\thefigure}{A}
\renewcommand{\thesection}{A}
\setcounter{equation}{0}
\setcounter{theorem}{0}
\section*{Appendix A: Proof of Lemma \ref{dual-reg}}

In this appendix, we prove Lemma \ref{dual-reg} via the following proposition, where equation \eqref{dual-phi-eq} differs from \eqref{dual-problem-Ritz} via a change of variable $t\rightarrow T-t$ in time. 
\begin{proposition}\label{dual-reg-lemma1}
	The initial-boundary value problem
	\begin{subequations}\label{dual-phi-eq}
		\begin{align}
			- \mathcal{L}_s {\bb \phi} + {\bb \phi} &= -\partial_t \sig ({\bb \phi}, q) \n +\f && \text{on}\,\,\, \Sigma\times (0,T]
			 &&\mbox{\rm(the boundary condition)}
			\\
			-\nabla \cdot \sigma ({\bb \phi}, q) + {\bb \phi} &=  0 && \text{in}\,\,\, \Omega\times (0,T] \\
			\nabla \cdot {\bb \phi} &=  0 && \text{in}\,\,\, \Omega \times (0,T] \\
			\sig ({\bb \phi}, q) \n &=0 &&\mbox{at}\,\,\, t=0 &&\mbox{\rm(the initial condition)},
		\end{align} 
	\end{subequations}
has a unique solution $({\bb \phi},q)$ which satisfies the following regularity estimate:
	\begin{align}
		\| {\bb \phi} \|_{L^2 H^2} + \| {\bb \phi} \|_{L^2 H^2 (\Sigma)} + \| q
		\|_{L^2 H^1} + \| \sig ({\bb \phi}, q) \n \|_{L^{\infty} L^2 (\Sigma)}\leq C \| \f \|_{L^2 L^2 (\Sigma)}
	\end{align}	
\end{proposition}
\begin{proof}
We divide the proof into three parts. In the first part, we introduce the Neumann-to-Dirichlet operator and reformulate \eqref{dual-phi-eq} into an evolution equation \eqref{xi-strong-eq} on the boundary $\Sigma$ with the aid of Neumann-to-Dirichlet operator, and then establish some mapping properties of the Neumann-to-Dirichlet operator to be used in the proof of Proposition \ref{dual-reg-lemma1}. In the second part, we establish the existence, uniqueness and regularity of solutions to an equivalent formulation of \eqref{dual-phi-eq}, i.e., equation \eqref{xi-strong-eq} below. Finally, in the third part, we establish regularity estimates for the solutions to \eqref{dual-phi-eq}. 
	
	{\it Part 1}. We can define the Neumann-to-Dirichlet operator $\mathcal{N} : H^{- 1 / 2}(\Sigma)^d \rightarrow H^{1 / 2} (\Sigma)^d$ as ${\bb \zeta} \mapsto (\mathcal{N}^v{\bb \zeta}) |_{\Sigma} $, with $(\mathcal{N}^v {\bb \zeta}, \mathcal{N }^p
	{\bb \zeta})$ being the solution of the following Stokes equation:
	\begin{subequations}\label{NtD-weak-form}
			\begin{align}
			a_f (\mathcal{N }^v {\bb \zeta} , \bv) - b (\mathcal{N}^p {\bb \zeta}, \bv ) +
			(\mathcal{N}^v {\bb \zeta}, \bv) &= ({\bb \zeta} , \bv)_{\Sigma} && \forall \bv \in H^1(\Omega)^d \\
			b (q, \mathcal{N}^v {\bb \zeta}) &= 0 && \forall q \in L^2(\Omega) . 
		\end{align}
	\end{subequations}
Therefore, 
$$
- \nabla\cdot \sigma(\mathcal{N }^v {\bb \zeta},\mathcal{N }^p {\bb \zeta}) + \mathcal{N }^v {\bb \zeta} = 0 \,\,\,\mbox{in}\,\,\,\Omega 
\quad\mbox{and}\quad
\sigma(\mathcal{N }^v {\bb \zeta},\mathcal{N }^p {\bb \zeta})\n = {\bb \zeta} \,\,\,\mbox{on}\,\,\,\Sigma . 
$$

	Let ${\bb \xi} = \sig({\bb \phi}, q)\n$. Then it is easy to see that problem \eqref{dual-phi-eq} can be equivalently formulated as follows: Find ${\bb \xi} (t) \in H^1(\Sigma)^d $ for $t\in [0,T]$ satisfying the following evolution equation: 
	\begin{equation}\label{xi-strong-eq}
		- \mathcal{L}_s \mathcal{N} {\bb \xi} + \mathcal{N} {\bb \xi} + \partial_t \mathcal{}
		{\bb \xi} = \f \,\,\, \text{on}\,\,\, \Sigma\times(0,T], \,\,\,\mbox{with initial condition}\,\,\, {\bb \xi} (0) = 0.
	\end{equation}
By choosing $\bv=\mathcal{N }^v \bvarphi$ in \eqref{NtD-weak-form} and using relation $b (\mathcal{N}^p {\bb \zeta}, \mathcal{N }^v \bvarphi )=0$ (due to the definition of $\mathcal{N }^v \bvarphi$), we obtain 
\begin{equation}\label{NtD-symmetric}
	 ({\bb \zeta}, \mathcal{N}\bvarphi)_{\Sigma} = a_f (\mathcal{N}^v {\bb \zeta},
	\mathcal{N}^v \bvarphi) + (\mathcal{N}^v {\bb \zeta}, \mathcal{N}^v \bvarphi) \quad
	\forall \bzeta,\bvarphi \in H^{- 1 / 2} (\Sigma)^d .  
\end{equation}
Especially, this implies that 
	\begin{equation}\label{NtD-garding}
		 ({\bb \zeta}, \mathcal{N}{\bb \zeta})_{\Sigma} =2\mu\| \mathcal{N}^v {\bb \zeta} \|^2_f+ \|
		\mathcal{N }^v {\bb \zeta} \|^2 \sim \| \mathcal{N}^v {\bb \zeta} \|_{H^1
			}^2\sim\|\mathcal{N}{\bb \zeta}\|^2_{H^{1/2}(\Sigma)}
			\quad
	\forall {\bb \zeta}  \in H^{- 1 / 2} (\Sigma)^d .
	\end{equation}
By choosing $k=s$ in the regularity result in \eqref{reg-ass1} with $s\geq -1/2,s\in \mathbb{R}$ and noting the trace inequality, we can establish the following mapping property of the Neumann-to-Dirichlet operator:
	\begin{align}\label{NtD-regularity}
		 \| \mathcal{N} {\bb \zeta} \|_{H^{s+1} (\Sigma)}\leq C\|\mathcal{N}^v{\bb \zeta}\|_{H^{s+3/2}(\Omega)} \leq C \| {\bb \zeta} \|_{H^{s} (\Sigma)} \quad \forall s\geq -1/2,s\in \mathbb{R},
	\end{align}
	Note that $\mathcal{N} {\bb \zeta} = 0$ if and only if ${\bb \zeta} = \lambda \n$ for some scalar constant $\lambda\in\mathbb{R}$. This motivates us to define the following subspace of $H^s(\Sigma)^d$ for $s\in \mathbb{R}$:
	\begin{equation*}
		\widetilde{H}^s(\Sigma)^d:=\{{\bb \zeta}\in H^s(\Sigma)^d: ({\bb \zeta},\n)_\Sigma=0 \} . 
	\end{equation*}
	Then we define the Dirichlet-to-Neumann operator $\mathcal{D}:\widetilde{H}^{1/2}(\Sigma)^d\to \widetilde{H}^{-1/2}(\Sigma)^d$ as follows: For ${\bb \zeta}\in \widetilde{H}^{1/2}(\Sigma)^d$, let $(\mathcal{D}^v{\bb \zeta},\mathcal{D}^p{\bb \zeta})$ be the weak solution of 
	\begin{align}\label{DtN-strong-form}
		\begin{aligned}
			-\nabla\cdot\sig(\mathcal{D}^v{\bb \zeta},\mathcal{D}^p{\bb \zeta})+\mathcal{D}^v{\bb \zeta}&=0 && \text{in $\Omega$}\\
			\nabla\cdot \mathcal{D}^v{\bb \zeta}&=0&&\text{in $\Omega$} \\
			 (\mathcal{D}^v{\bb \zeta})|_\Sigma &={\bb \zeta} &&\text{on $\Sigma$} , 
		\end{aligned}
	\end{align}
		and then define $\mathcal{D}{\bb \zeta}\in \widetilde{H}^{-1/2}(\Sigma)^d$ by the following equation
	\begin{equation}\label{DtN-weak-form}
			a_f (\mathcal{D}^v {\bb \zeta} , \bv) - b (\mathcal{D}^p {\bb \zeta}, \bv ) +
			(\mathcal{D}^v {\bb \zeta}, \bv) = (\mathcal{D}{\bb \zeta} , \bv)_{\Sigma} \quad \forall \bv \in H^1(\Omega)^d . 
	\end{equation}
Since the function $\mathcal{D}^p{\bb \zeta}$ in equation \eqref{DtN-strong-form} is only determined up to a constant, we can choose this constant in such a way that the function $\mathcal{D}{\bb \zeta}$ defined by \eqref{DtN-weak-form} lies in $\widetilde{H}^{-1/2}(\Sigma)^d$.  Using trace theorem and Bogovoski's map (cf. \cite[Corollary
1.5]{farwig}) there exists $\bv\in H^1(\Omega)^d$ such that $\bv|_\Sigma=\n$, $\nabla\cdot \bv=\frac{\|\n\|_{\Sigma}^2}{|\Omega|}$ with $\|\bv\|_{H^1}\leq C$, testing \eqref{DtN-weak-form} with such $\bv$, noting the assumption that $(\mathcal{D}{\bb \zeta},\n)_\Sigma=0$, we obtain
\begin{equation}\label{mean-value-p}
	|\overline{\mathcal{D}^p{\bb \zeta}}|\leq C\|\mathcal{D}^v{\bb \zeta}\|_{H^1},
\end{equation}
where $\overline{\mathcal{D}^p{\bb \zeta}}$ is the mean value of $\mathcal{D}^p{\bb \zeta}$ over $\Omega$. Therefore, choosing $k=s$ with $s\geq 1/2,s\in \mathbb{R}$ in \eqref{reg-ass2} and combining \eqref{mean-value-p} leads to the following estimates
\begin{equation}\label{dirichlet-reg-full}
	\|\mathcal{D}^v{\bb \zeta}\|_{H^{s+1/2}}+\|\mathcal{D}^p{\bb \zeta}\|_{H^{s-1/2}}\leq C\|{\bb \zeta}\|_{H^s(\Sigma)} \quad \forall s\geq 1/2,s\in \mathbb{R}.
\end{equation}
From the weak form \eqref{DtN-weak-form}, it follows that
\begin{equation}\label{min-reg-estimate-traction}
	\|\mathcal{D}{\bb \zeta}\|_{H^{-1/2}(\Sigma)}\leq C\left(\|\mathcal{D}^v{\bb \zeta}\|_{H^1}+\|\mathcal{D}^p{\bb \zeta}\|\right).
\end{equation}
Meanwhile when $s\geq 3/2$, by trace inequality we have
\begin{equation}\label{grad-trace}
	\|\mathcal{D}{\bb \zeta}\|_{H^{s-1}(\Sigma)}\leq C\left(\|\mathcal{D}^v{\bb \zeta}\|_{H^{s+1/2}}+\|\mathcal{D}^p{\bb \zeta}\|_{s-1/2}\right)\quad \forall s\geq 3/2,s\in \mathbb{R}.
\end{equation} 
Combining \eqref{min-reg-estimate-traction}, \eqref{grad-trace} and \eqref{dirichlet-reg-full} leads to the following estimates of the Neumann value $\mathcal{D}{\bb \zeta}$ in terms of the Dirichlet value ${\bb \zeta}$: 
	\begin{align}\label{DtN-estimate}
	\begin{aligned}
		 \|\mathcal{D}{\bb \zeta}\|_{H^{-1/2}(\Sigma)} &\leq C\|{\bb \zeta}\|_{H^{1/2}(\Sigma)} &&\forall {\bb \zeta}\in \widetilde{H}^{1/2}(\Sigma)^d , \\
		 \|\mathcal{D}{\bb \zeta}\|_{H^{s-1}(\Sigma)} &\leq C\|{\bb \zeta}\|_{H^{s}(\Sigma)} &&\forall {\bb \zeta}\in \widetilde{H}^{s}(\Sigma)^d.\quad\mbox{(whenever $s\geq 3/2,s\in \mathbb{R}$)}
	\end{aligned}
	\end{align}
	The following complex interpolation of Sobolev spaces hold: 
	\begin{subequations}
		\begin{align}
			&[H^{k}(\Sigma)^d,H^{s}(\Sigma)^d]_{\theta}=H^{\theta s+(1-\theta)k}(\Sigma)^d\quad\forall k,s\in \mathbb{R},\theta\in [0,1];\label{interpolation-1}\\
			& [\widetilde{H}^{k}(\Sigma)^d,\widetilde{H}^{s}(\Sigma)^d]_{\theta}=\widetilde{H}^{\theta s+(1-\theta)k}(\Sigma)^d\quad \forall k,s\in \mathbb{R},\theta\in [0,1]; \label{interpolation-2}
		\end{align}
	\end{subequations}
	where \eqref{interpolation-1} follows from \cite[Proposition 3.1-3.2 of Chapter 4]{Taylor} and \eqref{interpolation-2} follows from \eqref{interpolation-1} because $\widetilde{H}^{s}(\Sigma)^d$ is a retract of $H^s(\Sigma)^d$ for $s\in \mathbb{R}$ via projection $\pi:H^s(\Sigma)^d\to \widetilde{H}^{s}(\Sigma)^d$, with 
	\begin{equation}\label{tilde-projection}
		\pi({\bb \zeta}):={\bb \zeta}-\frac{({\bb \zeta},\n)_\Sigma}{\|\n\|^2_\Sigma}\n.
	\end{equation}  
	Therefore, the following result follows from the complex interpolation between the two estimates in \eqref{DtN-estimate}:  
	 \begin{equation}\label{DtN-regularity}
	 	\|\mathcal{D}{\bb \zeta}\|_{H^{s-1}(\Sigma)}\leq C\|{\bb \zeta}\|_{H^{s}(\Sigma)}\quad \forall {\bb \zeta}\in \widetilde{H}^s(\Sigma)^d\quad\forall s\geq 1/2,s\in \mathbb{R}.
	\end{equation}
	If we restrict the domain of $\mathcal{N}$ to $\widetilde{H}^{-1/2}(\Sigma)^d$, then $\mathcal{N}:\widetilde{H}^{-1/2}(\Sigma)^d\to \widetilde{H}^{1/2}(\Sigma)^d$ and $\mathcal{D}:\widetilde{H}^{1/2}(\Sigma)^d\to \widetilde{H}^{-1/2}(\Sigma)^d$ are inverse maps of each other. This leads to the following norm equivalence: 
   \begin{equation*}
   	\|{\bb \zeta}\|_{H^{-1/2}(\Sigma)}\sim \|\mathcal{N}{\bb \zeta}\|_{H^{1/2}(\Sigma)}\quad \forall {\bb \zeta}\in \widetilde{H}^{-1/2}(\Sigma)^d.
   \end{equation*}
	Similarly, from identity $\mathcal{D}\mathcal{N}{\bb \zeta}=\mathcal{N}\mathcal{D}{\bb \zeta}={\bb \zeta}$ for ${\bb \zeta}\in \widetilde{H}^{1/2}(\Sigma)^d$ 
	and the mapping property in \eqref{DtN-regularity} and \eqref{NtD-regularity}, we conclude that the maps $\mathcal{N}:\widetilde{H}^{s}(\Sigma)^d\to \widetilde{H}^{s+1}(\Sigma)^d$ and $\mathcal{D}:\widetilde{H}^{s+1}(\Sigma)^d\to \widetilde{H}^s(\Sigma)^d$ are also inverse to each other for all $s\geq -1/2,s\in \mathbb{R}$. This implies the following norm equivalence for $s\geq -1/2,s\in \mathbb{R}$: 
	\begin{equation}\label{norm-equiv}
	\|{\bb \zeta}\|_{H^{s}(\Sigma)}\sim \|\mathcal{N}{\bb \zeta}\|_{H^{s+1}(\Sigma)};\;\;\|{\bb \zeta}\|_{H^{s+1}(\Sigma)}\sim \|\mathcal{D}{\bb \zeta}\|_{H^{s}(\Sigma)} \quad \forall \bzeta\in \widetilde{H}^{-1/2}(\Sigma)^d
	\end{equation}
	To facilitate further use, we summarize the properties of the NtD (Neumann to Dirichlet) operator and DtN (Dirichlet to Neumann) operator in the following lemma:
	\begin{lemma}
		~
		\begin{enumerate}
			\item For $s\geq -1/2,s\in \mathbb{R}$, the NtD operator $\mathcal{N}:\widetilde{H}^s(\Sigma)^d\to \widetilde{H}^{s+1}(\Sigma)^d$ and DtN operator $\mathcal{D}:\widetilde{H}^{s+1}(\Sigma)^d\to \widetilde{H}^s(\Sigma)^d$ are bounded and inverse to each other.
			\item With domain ${\rm dom}(\mathcal{D}):=\widetilde{H}^1(\Sigma)^d\subseteq \widetilde{L}^2(\Sigma)^d$, the DtN operator $\mathcal{D}$ is a self-adjoint positive-definite operator on $\widetilde{L}^2(\Sigma)^d$. The NtD operator $\mathcal{N}:\widetilde{L}^2(\Sigma)^d\to \widetilde{L}^2(\Sigma)^d$ is a compact self-adjoint positive-definite operator on $\widetilde{L}^2(\Sigma)^d$.
			\item The square root operators $\mathcal{D}^{1/2}$ and $\mathcal{N}^{1/2}$ are well defined. Moreover, for $s\geq -1/2,s\in \mathbb{R}$, operators $\mathcal{N}^{1/2}:\widetilde{H}^s(\Sigma)^d\to \widetilde{H}^{s+1/2}(\Sigma)^d$ and $\mathcal{D}^{1/2}:\widetilde{H}^{s+1/2}(\Sigma)^d\to \widetilde{H}^s(\Sigma)^d$ are bounded and inverse to each other.
		\end{enumerate}
	\end{lemma}
	\begin{proof}
		The three statements are proved as follows.  
		\begin{enumerate}
			\item The first statement has been proved in \eqref{norm-equiv}.
			\item From \eqref{NtD-symmetric} and \eqref{NtD-garding} it follows that $\mathcal{N}$ is self-adjoint positive-definite operator on $\widetilde{L}^2(\Sigma)^d$. Since $\widetilde{H}^1(\Sigma)^d\to \widetilde{L}^2(\Sigma)^d$ is a compact embedding by Rellich-Kondrachov theorem (cf. \cite[Proposition 3.4 of Chapter 4]{Taylor}), from mapping property \eqref{NtD-regularity} of $\mathcal{N}$ it follows that $\mathcal{N}$ is a compact operator. To verify $\mathcal{D}:{\rm dom}(\mathcal{D})\to \widetilde{L}^2(\Sigma)^d$ is self-adjoint, it suffices to show that if ${\bb \zeta}$ satisfies 
			\begin{equation}\label{self-adj-proof-tmp}
				|({\bb \zeta},\mathcal{D}\bvarphi)_\Sigma|\leq C\|\bvarphi\|_\Sigma\quad \forall \bvarphi\in \widetilde{H}^1(\Sigma)^d,
			\end{equation}
			then ${\bb \zeta}\in \widetilde{H}^1(\Sigma)^d$. From \eqref{self-adj-proof-tmp}, by Riesz representation theorem there exists $\bg \in \widetilde{L}^2(\Sigma)^d$ such that
			\begin{equation*}
				({\bb \zeta},\mathcal{D}\bvarphi)_\Sigma=(\bg,\bvarphi)_\Sigma\quad \forall \bvarphi\in \widetilde{H}^1(\Sigma)^d.
			\end{equation*}
			Especially, taking $\bvarphi=\mathcal{N}\bxi$, it follows that
			\begin{equation*}
				({\bb \zeta},\bxi)_\Sigma=(\bg,\mathcal{N}\bxi)_\Sigma=(\mathcal{N}\bg,\bxi)_\Sigma\quad \forall \bxi\in \widetilde{L}^2(\Sigma)^d.
			\end{equation*}
			Therefore ${\bb \zeta}=\mathcal{N}\bg\in \widetilde{H}^1(\Sigma)^d$, proof of the second statement is complete.
			\item By the spectrum theory of compact self-adjoint operator (cf. \cite[Theorem 5.3.16]{salamon}), $\widetilde{L}^2(\Sigma)^d$ admits an orthornormal basis of eigenvectors $\{\bomega_i\}_{i\in \mathbb{N}}$ of $\mathcal{N}$ and $\mathcal{N}$ has the following expression
			\begin{equation*}
				\mathcal{N}\bzeta=\sum_{i=1}^\infty (\mathcal{N}\bzeta,\bomega_i)_\Sigma\bomega_i=\sum_{i=1}^\infty \lambda_i(\bzeta,\bomega_i)_\Sigma\bomega_i\quad \forall \bzeta\in \widetilde{H}^{-1/2}(\Sigma)^d,
			\end{equation*}
			where $\lambda_i>0$ is the eigenvalue associated with $\bomega_i$. From norm equivalence \eqref{norm-equiv}, we can deduce that for $s\in \mathbb{N}$ there holds
			\begin{equation}\label{to-sequence-norm1}
				\|\bzeta\|_{H^s(\Sigma)}\sim \|\mathcal{D}^{s}\bzeta\|_\Sigma= \left(\sum_{i=1}^\infty \lambda_i^{-2s}|(\bzeta,\bomega_i)_\Sigma|^2 \right)^{1/2}\quad \forall s\in \mathbb{N}.
			\end{equation}
			In view of complex interpolation result of weighted $\ell^2$-sequence spaces (cf. \cite[Theorem 5.5.3]{bergh}), in fact  \eqref{to-sequence-norm1} is valid for all $s\geq 0,s\in \mathbb{R}$ by complex interpolation method. Moreover for $ -1/2\leq s<0$, using norm equivalence \eqref{norm-equiv} and \eqref{to-sequence-norm1} (which is valid for $s\geq 0,s\in \mathbb{R}$) we have
			\begin{equation}\label{to-sequence-norm2}
				\|\bzeta\|_{H^{s}(\Sigma)}\sim \|\mathcal{N}\bzeta\|_{H^{s+1}(\Sigma)}\sim \left(\sum_{i=1}^\infty \lambda_i^{-2s}|(\bzeta,\bomega_i)_\Sigma|^2 \right)^{1/2} \quad \forall s\in \mathbb{R},-1/2\leq s<0.
			\end{equation}
			Combining \eqref{to-sequence-norm1} and \eqref{to-sequence-norm2}, we arrive at
			\begin{equation}\label{to-sequence-norm3}
				\|\bzeta\|_{H^s(\Sigma)}\sim \left(\sum_{i=1}^\infty \lambda_i^{-2s}|(\bzeta,\bomega_i)_\Sigma|^2 \right)^{1/2}\quad \forall s\geq -1/2,s\in \mathbb{R}.
			\end{equation}
			We can define square root operators $\mathcal{N}^{1/2}$ and $\mathcal{D}^{1/2}$ by formula
			\begin{align*}
				\mathcal{N}^{1/2}\bzeta:&=\sum_{i=1}^\infty \lambda_i^{1/2}(\bzeta,\bomega_i)_\Sigma\bomega_i\quad &&\forall \bzeta\in \widetilde{H}^{-1/2}(\Sigma)^d\quad \mbox{(this series converges in $\widetilde{L}^{2}(\Sigma)^d$)}\\
				\mathcal{D}^{1/2}\bzeta:&=\sum_{i=1}^\infty \lambda_i^{-1/2}(\bzeta,\bomega_i)_\Sigma\bomega_i\quad &&\forall \bzeta\in \widetilde{L}^{2}(\Sigma)^d\quad \mbox{(this series converges in $\widetilde{H}^{-1/2}(\Sigma)^d$)},
			\end{align*}
			from the norm equivalence in \eqref{to-sequence-norm3}, it is direct to verify that operators $\mathcal{N}^{1/2}$ and $\mathcal{D}^{1/2}$ are inverse to each other and satisfy the following mapping property for $s\geq -1/2, s\in \mathbb{R}$
			\begin{equation}\label{norm-equiv-1/2}
				\|{\bb \zeta}\|_{H^{s}(\Sigma)}\sim \|\mathcal{N}^{1/2}{\bb \zeta}\|_{H^{s+1/2}(\Sigma)};\;\;\|{\bb \zeta}\|_{H^{s+1/2}(\Sigma)}\sim \|\mathcal{D}^{1/2}{\bb \zeta}\|_{H^{s}(\Sigma)}\quad \forall \bzeta\in \widetilde{H}^{-1/2}(\Sigma)^d.
			\end{equation}
			The proof of third statement is complete.
		\end{enumerate}
	\hfill\end{proof}

	{\it Part 2}. Taking into account of the fact that $\mathcal{N}$ is not injective on $L^2(\Sigma)^d$, for convenience of our further construction we first take the $L^2$-orthogonal projection $\pi:L^2(\Sigma)^d\to \widetilde{L}^2(\Sigma)^d$ defined as in \eqref{tilde-projection} on the both side of \eqref{xi-strong-eq}, and obtain the following equation with solution space contained in $\widetilde{L}^2(\Sigma)^d$: seek $\widetilde{\bxi}\in L^2\widetilde{H}^{1}(\Sigma)^d$ with $\partial_t\widetilde{\bxi}\in L^2\widetilde{L}^2(\Sigma)^d$ satisfying 
	\begin{equation}\label{xi-eq-projed}
		\partial_t\widetilde{\bxi}+\mathcal{A}\widetilde{\bxi}=\widetilde{\f};\quad \widetilde{\bxi}(0)=0,
	\end{equation}
	where $\mathcal{A}=\pi(I-\mathcal{L}_s)\mathcal{N}$ and $\widetilde{\f}=\pi \f$. One difficulty in proving existence and regularity of solution to \eqref{xi-eq-projed} is that the operator $\mathcal{A}:\widetilde{H}^1(\Sigma)^d\to \widetilde{L}^2(\Sigma)^d$ is not a self-adjoint operator in $\widetilde{L}^2(\Sigma)^d$. To overcome this difficulty, we consider the following change of variable $\bomega=\mathcal{N}^{1/2}\widetilde{\bxi}$, and reformulate \eqref{xi-eq-projed} into an abstract Cauchy problem on $\bomega$: seek $\bomega\in L^2\widetilde{H}^{1}(\Sigma)^d$ with $\partial_t\omega\in L^2\widetilde{L}^2(\Sigma)^d$ satisfying 
	\begin{equation}\label{omega-eq}
		\partial_t \bomega+\mathcal{B}\bomega=\bg;\quad \bomega(0)=0,
	\end{equation}
	where $\mathcal{B}:=\mathcal{N}^{1/2}\pi(I-\mathcal{L}_s)\mathcal{N}^{1/2}$ and $\mathbf{g}:=\mathcal{N}^{1/2}\widetilde{\f}$. We summarize some useful properties on the operators $\mathcal{A}$ and $\mathcal{B}$ in the following lemma:
	\begin{lemma}
		~
		\begin{enumerate}
			\item There holds norm equivalence for $0\leq s\leq 1,s\in \mathbb{R}$
			\begin{equation}\label{AB-norm-equiv}
				\|\mathcal{A}\bzeta\|_{H^s(\Sigma)}\sim \|\bzeta\|_{H^{s+1}(\Sigma)};\;\;\;\|\mathcal{B}\bzeta\|_{H^{s}(\Sigma)}\sim \|\bzeta\|_{H^{s+1}(\Sigma)} \quad \forall \bzeta\in \widetilde{H}^{-1/2}(\Sigma)^d.
			\end{equation}
			\item $\mathcal{B}$ is a self-adjoint positive-definite operator on $\widetilde{L}^2(\Sigma)^d$ with domain ${\rm dom}(\mathcal{B}):=\widetilde{H}^1(\Sigma)^d$.
		\end{enumerate}
	\end{lemma} 
	\begin{proof}
		The two statements are proved as follows. 
		\begin{enumerate}
			\item In view of the norm equivalence relations in \eqref{norm-equiv} and \eqref{norm-equiv-1/2}, it suffices to show the following norm equivalence for $-1\leq s\leq 1,s\in \mathbb{R}$
			\begin{equation}\label{pi-Ls-norm-equiv}
				\|\pi (I-\mathcal{L}_s)\bzeta\|_{H^s(\Sigma)} \sim \|\bzeta\|_{H^{s+2}(\Sigma)}\quad \forall \bzeta\in \widetilde{H}^{-1/2}(\Sigma)^d.
			\end{equation} 
			Note that one direction of the norm equivalence in \eqref{pi-Ls-norm-equiv} is given by assumption \eqref{Ls-second-order}. To prove the opposite direction, observe first that
			\begin{equation*}
				\|\pi(I-\mathcal{L}_s)\bzeta\|_{H^{-1}(\Sigma)}\|\bzeta\|_{H^1(\Sigma)}\geq (\pi(I-\mathcal{L}_s)\bzeta,\bzeta)_\Sigma=((I-\mathcal{L}_s)\bzeta,\bzeta)_\Sigma\geq C\|\bzeta\|^2_{H^1(\Sigma)}.
			\end{equation*}
			It follows that \eqref{pi-Ls-norm-equiv} is valid for $s=-1$. Next we note that, by definition \eqref{tilde-projection} of projection $\pi:H^s(\Sigma)^d\to \widetilde{H}^s(\Sigma)^d$, 
			\begin{equation}\label{tmp-diff-estimate}
				\|\pi(I-\mathcal{L}_s)\bzeta-(I-\mathcal{L}_s)\bzeta\|_{H^s(\Sigma)}\leq C|(\bzeta,(I-\mathcal{L}_s)\n)_\Sigma|\leq C\|\bzeta\|_{H^1(\Sigma)}
			\end{equation}
			For $-1\leq s\leq 1,s\in \mathbb{R}$, in view of regularity assumption \eqref{reg-ass3} and the estimate \eqref{tmp-diff-estimate} above, we have
			\begin{align*}
				\|\bzeta\|_{H^{s+2}(\Sigma)}&\leq C\|(I-\mathcal{L}_s)\bzeta\|_{H^s(\Sigma)}\\
				&\leq C\|\pi(I-\mathcal{L}_s)\bzeta\|_{H^s(\Sigma)}+C\|\bzeta\|_{H^1(\Sigma)}\\
				&\leq C\|\pi(I-\mathcal{L}_s)\bzeta\|_{H^s(\Sigma)}\quad\mbox{(this is because \eqref{pi-Ls-norm-equiv} is valid for $s=-1$)}.
			\end{align*}
			Thus \eqref{pi-Ls-norm-equiv} is proved and the first statement follows directly. 
			\item Since $\mathcal{B}$ is obviously symmetric and positive definite on its domain ${\rm dom}(B)=\widetilde{H}^1(\Sigma)^d$. To prove that $\mathcal{B}$ is self-adjoint, it remains to show that the domain of the dual operator $\mathcal{B}'$ defined by 
			$$
			{\rm dom}(\mathcal{B}') = \{\w\in \widetilde L^2(\Sigma)^d: \exists\,\, \bg\in \widetilde{L}^2(\Sigma)^d\,\,\mbox{such that}\,\, (\bomega,\mathcal{B}\bzeta)_\Sigma=(\bg,\bzeta)_\Sigma\quad\forall \bzeta\in \widetilde{H}^1(\Sigma)^d \} , 
			$$
coincides with the domain of $\mathcal{B}$. 
			Therefore, we need to prove that if $\bomega\in \widetilde{L}^2(\Sigma)^d$ satisfies
			\begin{equation}\label{A3-tmp1}
				(\bomega,\mathcal{B}\bzeta)_\Sigma=(\bg,\bzeta)_\Sigma\quad\forall \bzeta\in \widetilde{H}^1(\Sigma)^d,
			\end{equation}
		 	for some $\bg\in \widetilde{L}^2(\Sigma)^d$, then $\bomega\in \widetilde{H}^1(\Sigma)^d$. To this end, we define $\bvarphi\in \widetilde{H}^1(\Sigma)^d$ to be the weak solution of equation
		 	\begin{equation}\label{B-equation-aux}
		 		a_s(\bvarphi,\bzeta)+(\bvarphi,\bzeta)_\Sigma=(\mathcal{D}^{1/2}\bg,\bzeta)_\Sigma\quad \forall\bzeta\in \widetilde{H}^1(\Sigma)^d,
		 	\end{equation}
	 		where the existence and uniqueness of solution to \eqref{B-equation-aux} is due to coercive property: $\|\bzeta\|^2_s+\|\bzeta\|^2_\Sigma\sim \|\bzeta\|_{H^1(\Sigma)}^2,\forall\bzeta\in \widetilde{H}^1(\Sigma)^d$. Equation \eqref{B-equation-aux} means 
	 		\begin{equation*}
	 			\pi(I-\mathcal{L}_s)\bvarphi=\mathcal{D}^{1/2}\bg\in \widetilde{H}^{-1/2}(\Sigma)^d,
	 		\end{equation*}
 			thus by norm equivalence \eqref{pi-Ls-norm-equiv} we have $\bvarphi\in \widetilde{H}^{3/2}(\Sigma)^d$. Now we observe
 			\begin{equation}\label{A3-tmp2}
 				(\mathcal{D}^{1/2}\bvarphi,\mathcal{B}\bzeta)_\Sigma=\left(\pi(I-\mathcal{L}_s)\bvarphi,\mathcal{N}^{1/2}\bzeta\right)_\Sigma=(\mathcal{D}^{1/2}\bg,\mathcal{N}^{1/2}\bzeta)_\Sigma=(\bg,\bzeta)_\Sigma \quad \forall \bzeta\in \widetilde{H}^1(\Sigma)^d
 			\end{equation}
 			By comparing \eqref{A3-tmp1} with \eqref{A3-tmp2} we obtain $\bomega=\mathcal{D}^{1/2}\bvarphi\in \widetilde{H}^{1}(\Sigma)^d$. This completes the proof. 
		\end{enumerate}
	\hfill\end{proof}

	Especially, since $\mathcal{B}$ is a self-adjoint positive-definite operator on $\widetilde{L}^2(\Sigma)^d$ with domain ${\rm dom}(\mathcal{B}):=\widetilde{H}^1(\Sigma)^d$, $-\mathcal{B}$ generates an analytic semigroup $E(t):\widetilde{L}^2(\Sigma)^d\to \widetilde{L}^2(\Sigma)^d$ for $t\geq 0$ (cf. \cite[Example 7.4.5]{salamon}), and the unique solution to \eqref{omega-eq} is given by
	\begin{equation*}
		\bomega(t)=\int_{0}^tE(t-s)\bg(s)ds.
	\end{equation*} 
	Moreover, for self-adjoint semigroup on a Hilbert space, the following $L^2$-maximal regularity estimate holds (cf. \cite[Theorem 7.6.11]{salamon}): 
	\begin{equation}\label{reg-omega-1}
		\|\partial_t \bomega\|_{L^2L^2(\Sigma)}+\|\mathcal{B}\bomega\|_{L^2L^2(\Sigma)}\leq C\|\bg\|_{L^2L^2(\Sigma)} , 
	\end{equation}
	which can be obtained by testing \eqref{omega-eq} with $\partial_t\bomega$. 
	If the source term $\bg$ in \eqref{omega-eq} possesses higher spacial regularity, the solution $\bomega$ also inherits higher spacial regularity. To see this, assume $\bg\in L^2\widetilde{H}^1(\Sigma)^d$, then since
	\begin{equation*}
		\mathcal{B}\bomega(t)=\int_0^t E(t-s)\mathcal{B}\bg(s)ds
	\end{equation*}
	is the solution to \eqref{omega-eq} with the source term replaced by $\mathcal{B}\bg$. Thus again by maximal $L^2$-regularity estimate, we have
	{
	\begin{equation}
		\|\mathcal{B}\partial_t \bomega\|_{L^2L^2(\Sigma)}+\|\mathcal{B}^2\bomega\|_{L^2L^2(\Sigma)}\leq C\|\mathcal{B}\bg\|_{L^2L^2(\Sigma)}.
	\end{equation}}
	By norm equivalence in \eqref{AB-norm-equiv}, it follows that
	\begin{equation}\label{reg-omega-2}
		\|\partial_t \bomega\|_{L^2H^1(\Sigma)}+\|\mathcal{B}\bomega\|_{L^2H^1(\Sigma)}\leq C\|\bg\|_{L^2H^1(\Sigma)}.
	\end{equation}
	Complex interpolation of \eqref{reg-omega-1} and \ref{reg-omega-2} gives
	\begin{equation}
		\|\partial_t \bomega\|_{L^2H^{1/2}(\Sigma)}+\|\mathcal{B}\bomega\|_{L^2H^{1/2}(\Sigma)}\leq C\|\bg\|_{L^2H^{1/2}(\Sigma)}.
	\end{equation}
	Now we take $\bg=\mathcal{N}^{1/2}\widetilde{\f}$, then it is direct to verify that $\widetilde{\bxi}:=\mathcal{D}^{1/2}\bomega$ is the solution to \eqref{xi-eq-projed} and satisfies estimate
	\begin{equation}\label{reg-xi-tilde}
		\|\partial_t\widetilde{\bxi}\|_{L^2L^2(\Sigma)}+\|\widetilde{\bxi}\|_{L^2H^1(\Sigma)}\leq C\|\widetilde{\f}\|_{L^2L^2(\Sigma)},
	\end{equation}
	where we have used norm equivalences in \eqref{AB-norm-equiv} and \eqref{norm-equiv-1/2}. Having obtained the solution $\widetilde{\bxi}$ to equation \eqref{xi-eq-projed}, if we write $\f(t)=\widetilde{\f}(t)+c(t)\n$ then it is direct to verify that $\bxi(t)=\widetilde{\bxi}(t)+k(t)\n$ is the solution to \eqref{xi-strong-eq}, where $k(t)$ is given by 
	\begin{align*}
		\partial_t k&=c-r(\widetilde{\bxi});\quad k(0)=0\\
		r(\widetilde{\bxi})&:=\frac{\big((I-\mathcal{L}_s)\mathcal{N}\widetilde{\bxi},\n\big)_\Sigma}{\|\n\|^2_\Sigma},
	\end{align*} 
	it follows that
	\begin{equation}\label{estimate-k-scalar}
		\|\partial_t k\|_{L^2(0,T)}\leq C(\|\f\|_{L^2L^2(\Sigma)}+\|\widetilde{\bxi}\|_{L^2H^1(\Sigma)})\leq C\|\f\|_{L^2L^2(\Sigma)}.
	\end{equation}
	Therefore, combining \eqref{estimate-k-scalar} and \eqref{reg-xi-tilde} we obtain
	\begin{equation}\label{reg-xi}
		\|\partial_t \bxi\|_{L^2L^2(\Sigma)}+\|\bxi\|_{L^2H^1(\Sigma)}\leq C\|\f\|_{L^2L^2(\Sigma)}
	\end{equation}

	{\it Part 3}. Given the solution ${\bb \xi}$ to equation \eqref{xi-strong-eq}, we define $({\bb \phi}, q) = (\mathcal{N}^v {\bb \xi},
	\mathcal{N}^p {\bb \xi})$. Then ${\bb \xi} = \sig ({\bb \phi}, q)
	\n$ and $\mathcal{N} {\bb \xi} = {\bb \phi} |_{\Sigma}$. Therefore, equation \eqref{xi-strong-eq} can be written as 
	
	\[ - \mathcal{L}_s {\bb \phi} + {\bb \phi} = - \partial_t \sig ({\bb \phi}, q)\n + \f \quad
	\mbox{on}\,\,\, \Sigma\times (0, T],\quad\mbox{with initial condition}\,\,\,  ({\bb \phi}, q) (0) = (0, 0) .  \]
	Thus $({\bb \phi}, q)$ is a solution of equation \eqref{dual-phi-eq}. Since $\sig({\bb \phi},q)\n(0)=0$, it follows that 
	\begin{equation*}
		\|\sig ({\bb \phi}, q) \n \|_{L^\infty L^2 (\Sigma)}\leq C\|\partial_t\sig ({\bb \phi}, q) \n \|_{L^2L^2(\Sigma)}.
	\end{equation*}
	Therefore, $({\bb \phi}, q)$ satsifies the following estimate according to \eqref{reg-xi} and the inequality above: 
	\begin{equation}
		\|\sig ({\bb \phi}, q) \n \|_{L^\infty L^2 (\Sigma)} + \| {\bb \phi} \|_{L^2 H^2
			(\Sigma)} \leq C \| \f \|_{L^2 L^2(\Sigma)}.
		\label{dual-proof-tmp1}
	\end{equation}
	Moreover, since $(\phi,q)$ is the solution of the homogeneous Stokes equation (with boundary value $\phi|_\Sigma=\mathcal{N} {\bb \xi}$), the following two estimates follow from the regularity results of the Stokes equation in  \eqref{reg-ass1}--\eqref{reg-ass2}: 
	\begin{equation}
		\| {\bb \phi} \|_{H^2} + \| \nabla q \| \leq C \| {\bb \phi}
		|_{\Sigma} \|_{H^2 (\Sigma)} ; \quad  \| {\bb \phi} \|_{H^1}
		+ \| q \| \leq C \| \sig({\bb \phi}, q)\n \|_{H^{- 1 / 2}
			(\Sigma)} \label{dual-proof-tmp2}
	\end{equation}
	Combining the estimates in \eqref{dual-proof-tmp1} and \eqref{dual-proof-tmp2}, we obtain the result of Proposition \ref{dual-phi-eq}. 
\hfill\end{proof}

\begin{proof}[Proof of \eqref{dual-reg-estimate}]
Let $\widehat{{\bb \phi}}={\bb \phi}(T-t)$ and $\widehat{q}=q(T-t)$, then $\widehat{{\bb \phi}}$ and $\widehat{q}$ is a solution of \eqref{dual-phi-eq} in Proposition \ref{dual-reg-lemma1}, with source term $\widehat{\f}(t)=\f(T-t)$, thus we have 
\begin{align*}
	&\| {\bb \phi} \|_{L^2 H^2 (\Omega)} + \| {\bb \phi} \|_{L^2 H^2 (\Sigma)} + \| q\|_{L^2 H^1 (\Omega)} + \| \sig ({\bb \phi}, q) \n \|_{L^{\infty} L^2 (\Sigma)}\leq C\|\f\|_{L^2L^2(\Sigma)}.
\end{align*}
The proof of Lemma \ref{dual-reg} is complete.
\hfill\end{proof}

\renewcommand{\theequation}{B.\arabic{equation}}
\renewcommand{\thetheorem}{B.\arabic{theorem}}
\renewcommand{\thefigure}{B}
\renewcommand{\thesection}{B}
\setcounter{equation}{0}
\setcounter{theorem}{0}

\section*{Appendix B: Proof of (\ref{super-approx})}\label{AppendixB}
In this subsection, we assume that $r\geq 2$. Under this assumption, we establish a negative-norm estimate for the Dirichlet Stokes--Ritz projection $R_h^D$ in the following lemma.
\begin{lemma}\label{Lemma-B3-RhD}
For the Dirichlet Stokes--Ritz projection $R_h^D$ defined in \eqref{RhD-def}, the following error estimate holds: 
	\begin{align}\label{RhD-neg-error}
		\|\u - R_h^D \u \|_{H^{-1}} + \|\u - R_h^D \u\|_{H^{-1}(\Sigma)}+ h\| p - R_h^D p \|_{H^{-1}} \le C h^{r+2}  	. 
	\end{align}
\end{lemma}
\begin{proof}[Skectch of Proof]
	From the definition of $R_h^D\u$ in \eqref{RhD-def} we can see that the following relation holds on the boundary $\Sigma$: 
	\begin{equation*}
		R_h^S\u-R_h^D\u=\frac{(R_h^S\u,\n)_\Sigma}{\|\n_h\|^2}\n_h=\frac{(R_h^S\u-\u,\n)_\Sigma}{\|\n_h\|^2}\n_h . 
	\end{equation*}
	Since  
	\begin{equation*}
		\frac{(R_h^S\u-\u,\n)_\Sigma}{\|\n_h\|^2}\lesssim C\|R_h^S\u-\u\|_{H^{-1}(\Sigma)}\leq Ch^{r+2},
	\end{equation*}
	it follows that $\|\u-R_h^D\u\|_{H^{-1}(\Sigma)}\leq Ch^{r+2}$. Then \eqref{RhD-neg-error} follows from the same routine of duality argument for the Dirichlet Stokes--Ritz projection.
\hfill\end{proof}

Next, we note that $\|(R_{sh}\u-\u)(0)\|$ also satisfies negative norm estimate below. 
\begin{lemma}\label{lemma-B4-sh}
	For the $R_{{sh}} \u (0)$ defined in \eqref{ritz-initial1}, the following negative-norm estimate  holds: 
	\begin{align}\label{Rsh-u-neg}
		\| (R_{{sh}} \u - \u) (0) \|_{H^{-1}(\Sigma)}\leq C h^{r + 2} . 
	\end{align}
\end{lemma}
\begin{proof}
We introduce a dual equation 
	\begin{equation}
		- \mathcal{L}_s \psi + \psi = \varphi \quad \text{$\psi$
			has periodic boundary condition on $\Sigma$}. 
	\end{equation}
	The regularity assumption in \eqref{reg-ass3} implies that $\| \psi \|_{H^3(\Sigma)} \leq \| \varphi \|_{H^1(\Sigma)}$. We can extend $\psi$ to be a function (still denoted by $\psi$) which is defined in $\Omega$ with periodic boundary condition and satisfies $\| \psi \|_{H^3} \leq C \| \psi\|_{H^3 (\Sigma)}$. Then the following relation can be derived: 
	\begin{align*}
		((R_{{sh}} \u - \u) (0),\varphi)_{\Sigma} = & a_s (_{} (R_{{sh}} \u -
		\u) (0), \psi)  + ((R_{{sh}} \u - \u) (0), \psi)_{\Sigma}\\
		=&a_s ((R_{{sh}} \u
		- \u) (0), \psi - I_h \psi)_{\Sigma} + ((R_{{sh}} \u - \u) (0), (\psi -
		I_h \psi))_{\Sigma}\\
		& - a_f ((R_h^D \partial_t \u - \partial_t \u) (0), I_h \psi) + 
		b((R_h^D \partial_t p - \partial_t p) (0), I_h \psi)\\
		& - ((R^D_h\partial_t \u - \partial_t \u) (0), I_h \psi)\\
		\leq & C h^{r + 2} \| \psi \|_{H^3 (\Sigma)} + | a_f ((R_h^D
		\partial_t \u - \partial_t \u) (0), \psi)|+ |b((R_h^D \partial_t p
		\partial_t  - p) (0), \psi)|\\
		& + |((R^D_h \partial_t
		\u - \partial_t \u) (0), \psi) | . 
	\end{align*}
	Since 
	\begin{align*}
		(\DD (R_h^D \partial_t \u - \partial_t \u) (0), \DD \psi)&= - ((R_h^D \partial_t \u -
		\partial_t \u) (0), \nabla \cdot \DD \psi) + ((R_h^D \partial_t \u -
		\partial_t \u) (0), \n \cdot \DD \psi)_{\Sigma}\\
		& \lesssim \left(\|(R_h^D \partial_t \u -
		\partial_t \u) (0)\|_{H^{-1}(\Omega)}+\|(R_h^D \partial_t \u -
		\partial_t \u) (0)\|_{H^{-1}(\Sigma)}\right)\|\psi\|_{H^3(\Omega)}\\
		&\leq Ch^{r+2}\|\varphi\|_{H^1(\Sigma)} 
	\end{align*}
	and
	\begin{align*} 
			&b( (R_h^D \partial_t p - \partial_t p) (0), \psi)
			\lesssim C \| (R_h^D \partial_t p - \partial_t p) (0) \|_{H^{-
					2}} \| \psi \|_{H^3}\leq Ch^{r+2}\|\varphi\|_{H^1(\Sigma)}\\
			&((R^D_h \partial_t
			\u - \partial_t \u) (0), \psi)\lesssim \|(R^D_h \partial_t
			\u - \partial_t \u) (0)\|_{H^{-1}}\|\psi\|_{H^1}\leq Ch^{r+2}\|\varphi\|_{H^1(\Sigma)} , 
	\end{align*} 
	summing up the estimates above yields the result in \eqref{Rsh-u-neg}.  
	The proof is complete.  
\hfill\end{proof}

\begin{lemma}\label{lemma:B5}
	Let $e^u_h := (R_h \u - R_h^D \u) (0)$ and $e_h^p:= (R_h p - R_h^D
	p)(0)$. Then the following estimates hold: 
	\begin{align}
		& \| e_h^u \|_{H^1} + \| e_h^p \| \leq \text{Ch}^{r + 1 / 2} , \label{B4-H1}\\
		& \| e_h^u \|_{H^{- 1} (\Sigma)}+\| e_h^u \|_{H^{- 1 / 2}} +\|e_h^p \|_{H^{- 3 / 2}}
		\leq C h^{r + 2}  \label{neg-norm-appendix} . 
	\end{align}
\end{lemma}
\begin{proof}
	To prove the first inequality in Lemma \ref{lemma:B5}, we note that 
	\begin{equation}\label{6.6-lemma-tmp}
		 a_f (e^u_h, \bv_h) + (e_h^u, \bv_h) - b (e_h^p, \bv_h) = 0 \quad \forall \bv_h
		\in \mathring \X^r_h. 
	\end{equation}
	Let $\u_h = E_h (e_h^u |_{\Sigma})$, where $E_h$ is an extension operator as in \eqref{Eh-estimate}. Then $e_h^u - \u_h \in \mathring{\X}^r_h$ and $\| \u_h \|_{H^1} \leq Ch^{- 1 / 2} \| e_h^u \|_{\Sigma} \leq
	C h^{r + 1 / 2}$. This estimate of $\| \u_h \|_{H^1}$ and relation \eqref{6.6-lemma-tmp} imply that 
	\[ a_f (e^u_h - \u_h, \bv_h) + (e_h^u - \u_h, \bv_h) - b (e_h^p, \bv_h) \lesssim C
	h^{r + 1 / 2} \| \bv_h \|_{H^1} \quad \forall \bv_h \in \mathring{\X}^r_h . \]
	Now, choosing $\bv_h = e_h^u - \u_h$ in the inequality above, we obtain \eqref{B4-H1}
	\[ \| e_h^u \|_{H^1 (\Omega)}  + \| e_{h }^p \| \leq C h^{r + 1 / 2} . 
	 \]
	
	 We prove the second inequality in Lemma \ref{lemma:B5} now.
	 On the boundary $\Sigma$, relations \eqref{d3-2b} and \eqref{P-tilde-formula0} imply that $(R_h \u - \u) (0) = (R_{{sh}} \u - \u) (0) - \lambda (R_{{sh}} \u (0)) \n_h$. Since 
\[ \lambda (R_{{sh}}\u(0))= \frac{(R_{{sh}} \u (0),
	\n)_{\Sigma}}{\| \n_h \|^2_{\Sigma}} = \frac{(R_{{sh}}\u(0) - \u (0),
	\n)_{\Sigma}}{\| \n_h \|_{\Sigma}^2} \lesssim C \| R_{{sh}} \u (0) - \u (0)\|_{H^{-1}(\Sigma)}, 
	\]
it follows from \eqref{Rsh-u-neg} in Lemma \ref {lemma-B4-sh} that $\|(R_h\u-\u)(0)\|_{H^{-1}(\Sigma)}\leq Ch^{r+2}$.  Using  inequality $\|(R_h^D\u-\u)(0)\|_{H^{-1}(\Sigma)}\leq Ch^{r+2}$ in \eqref{RhD-neg-error} of Lemma \ref{Lemma-B3-RhD}, we obtain
\begin{equation}
	\|(R_h \u - R_h^D \u) (0) \|_{H^{- 1} (\Sigma)} \leq C h^{r + 2} . 
\end{equation}

	Next, we consider a dual problem: For given $\f\in H^{1 / 2}(\Omega)^d$, we construct $({\bb \phi}, q)$ to be the solution of
	\[ - \nabla \cdot \sigma ({\bb \phi}, q) + {\bb \phi} = \f ; \quad \nabla \cdot {\bb \phi} = 0
	; \quad {\bb \phi} |_{\Sigma} = 0 ; \; q \in L^2_0 (\Omega) . \]
	By the regularity assumptions in \eqref{reg-ass2}, the following estimate of $\phi$ and $q$ can be written down:  
	\[ \| {\bb \phi}\|_{H^{5 / 2}} + \| p \|_{H^{3 / 2}} \leq C \| \f \|_{H^{1 / 2}} . \]
	From equation \eqref{6.6-lemma-tmp} one can see that 
	\begin{align*}
		(e_h^u, \f) = & a_f (e_h^u, {\bb \phi} - I_h {\bb \phi}) + (e_h^u, {\bb \phi} - I_h {\bb \phi}) -(\sig({\bb \phi}, q) \n, e_h^u)_{\Sigma}\\
	& + b (e_h^p, I_h {\bb \phi}- {\bb \phi})\\
	 \lesssim & C h^{r + 2} \| \f \|_{H^{1 / 2}} + \| e_h^u \|_{H^{-
			1}_p (\Sigma)} \|\f\|_{H^{1/2}} \leq Ch^{r + 2} \| \f
	\|_{H^{1 / 2}}.	
	\end{align*}
	Therefore, the following result is proved: 
	\[ | (e_h^u, \f) | \leq  C h^{r + 2} \| \f \|_{H^{1 / 2}} ; \quad \|e_h^u \|_{H^{- 1 / 2}} \leq C h^{r + 2}. \]
	We move on to consider the dual problem for pressure: For given $f \in H^{3 / 2} (\Omega)$, since $e_h^p \in L^2_0 (\Omega)$ it follows that 
	\[ (e_h^p, f) = (e_h^p, f - \overline{f}) . \]
	Thus it suffices to assume that $\int_{\Omega} f = 0$. Then, using Bogovoski's
	map {(see details in \cite[Corollary
		1.5]{farwig} and \cite[Theorem 4]{girault-94})}, there exists $\bv \in H^{5 / 2} (\Omega)$
	such that
	\[ \nabla\cdot\bv = f, \quad  \| \bv \|_{H^{5 / 2}} \leq C \| f
	\|_{H^{3 / 2}}, \quad \bv |_{\Sigma} = 0  . \]
	From equation \eqref{6.6-lemma-tmp}, we can find that  
	\[ (e_h^p, f) = b (e_h^p, \bv) = b (e_h^p, \bv - I_h \bv) + a_f (e_h^u, I_h \bv - \bv)
	+ (e_h^u, I_h \bv - \bv) + a_f (e_h^u, \bv) + (e_{h }^u, \bv) . \]
	Thus, combining the known estimate $\|e_h^u\|_{H^{-1/2}}\leq Ch^{r+2}$ and $\|e_h^u\|_{H^1}+\|e_h^p\|\leq Ch^r$, we have 
	\[ | (e_h^p, f) | \leq C h^{r + 2} \|f\|_{H^{3 / 2}} 
		+ | a_f (e_h^u, \bv) | .\]
	Using integration by parts, we derive that	
	\[ a_f (e_h^u, \bv) = 2\mu(\DD e_h^u, \DD \bv) = - 2\mu (e_h^u, \nabla \cdot \DD \bv) + 2\mu(e_h^u, \DD \bv\cdot \n)_{\Sigma} \lesssim C h^{r + 2} \| f \|_{H^{3/2}} . \]
	Therefore, we have proved the following result: 
	\[ | (e_h^p, f) | \leq  C h^{r + 2} \| f \|_{H^{3 / 2}} ; \quad \|e_h^p \|_{H^{- 3 / 2}} \leq C h^{r + 2} .  \]
	This completes the proof of Lemma \ref{lemma:B5}.
\hfill\end{proof}

\begin{lemma}\label{super-approx-lem}
For the projection operators $R_{h}$ and $R_{sh}$ defined in \eqref{def-Rhe0} and \eqref{def-Rshe0}, respectively, the following estimate holds: 
	\[ \|R_{\text{sh}} \e (0) - R_h \e (0) \|_{H^1  (\Sigma)} \leq C h^{r + 1} . \]
\end{lemma}

\begin{proof}
	By denoting $\delta_h : = R_h \e(0)-R_{\text{sh}} \e(0)$, $e^u_h :=
	(R_h \u - R_h^D \u) (0)$ and $e_h^p := (R_h p - R_h^D p) (0)$, we can write down the following equation according to the definitions of the two projection operators: 
	\[ a_s (_{} \delta_h, \bv_h)  + (\delta_h, \bv_h)_{\Sigma} + a_f (e_h^u, \bv_h) -
	b (e_h^p, \bv_h) + (e_h^u, \bv_h) = 0 \quad \forall \bv_h \in \X^r_h . \]
	Then, choosing $\bv_h = E_h \delta_h \in \X^r_h$ in the relation above and note that $\| \bv_h \|_{H^1} \leq C h^{- 1 / 2} \| \delta_h \|_{\Sigma}$, we derive that 
	\begin{equation}\label{sup-approx-lem-tmp1}
		\| \delta_h \|^2_{H^1 (\Sigma)} \leq C h^{- 1 / 2} \| \delta_h
		\|_{\Sigma} (\| e_h^u \|_{H^1} + \| e_h^p \|) \leq C h^r \|
		\delta_h \|_{\Sigma} . 
	\end{equation}
	Next, we consider the following dual problem: Let $\psi$ be the solution of 
	\[ - \mathcal{L}_s \psi + \psi = \delta_h \quad \text{$\psi$ has periodic boundary condition on $\Sigma$} . \]
	Then 
	\[ a_s (\psi, {\bb \xi})  + (\psi, {\bb \xi})_{\Sigma} = (\delta_h, {\bb \xi})_{\Sigma} \quad
	\forall {\bb \xi} \in \S \quad \text{ and } \|\psi\|_{H^2 (\Sigma)} \leq \|
	\delta_h \|_{\Sigma} .  \]
	We can extend $\psi$ to a function (still denoted by $\psi$) defined on
	$\Omega$ with periodic boundary condition and $\| \psi \|_{H^{5 / 2}} \leq C \|\psi \|_{H^2}$. Therefore
	\begin{align}\label{sup-approx-lem-tm2}
			{\| \delta_h \|^2_{\Sigma}}   = & a_s (\delta_h, \psi - I_h
		\psi)_{\Sigma} + (\delta_h, (\psi - I_h \psi))_{\Sigma}\nn\\
		& - a_f (e_h^u, I_h \psi) + b (e_h^p, I_h \psi) - (e_h^u, I_h \psi)\nn\\
		 \leq & C h  \| \delta_h \|_{H^1(\Sigma)} \| \psi\|_{H^2(\Sigma)} + C
		h^{3 / 2} (\| e_h^u \|_{H^1 } + \| e_h^p \|)\|\psi\|_{H^{5/2}}\nn\\
		  & + | a_f (e_h^u, \psi) - b (e_h^p, \psi) + (e_h^u, \psi) | . 
	\end{align}
	Using integrating by parts and negative-norm estimates in Lemma \ref{lemma:B5}, the following estimates can obtain: 
	\begin{align}\label{sup-approx-lem-tmp3}
		b (e_h^p, \psi) &\lesssim C \| \psi \|_{H^{5 / 2}} \|e_h^p
		\|_{H^{- 3 / 2}} \leq C h^{r + 2} \| \delta_h \|_{\Sigma} , \\
\label{sup-approx-lem-tmp4}
	a_f (e_h^u, \psi) + (e_h^u, \psi) &\lesssim C \| e_h^u \|_{H^{- 1 / 2}} \| \psi \|_{H^{5 / 2}} + C \| e_h^u \|_{H^{- 1}
		(\Sigma)} \| \psi \|_{H^{5 / 2}} \leq C h^{r + 2} \| \delta_h	\|_{\Sigma}	.
	\end{align}
	Therefore, combining the estimates in \eqref{sup-approx-lem-tmp1}--\eqref{sup-approx-lem-tmp4}, we obtain the following error estimate for $\delta_h$: 
	\[ \| \delta_h \|_{\Sigma}  \leq C h^{r / 2 + 1} \| \delta_h \|_{\Sigma}^{1
		/ 2} + C h^{r + 2} \Rightarrow \| \delta_h \|_{\Sigma} \leq C h^{r + 2} . 
	\]
	The inverse inequality implies $\| \delta_h \|_{H^1 (\Sigma)} \leq C
	h^{r + 1}$. This completes the proof of Lemma \ref{super-approx-lem}. 
\hfill\end{proof}

\renewcommand{\theequation}{C.\arabic{equation}}
\renewcommand{\thetheorem}{C.\arabic{theorem}}
\renewcommand{\thefigure}{C}
\renewcommand{\thesection}{C}
\setcounter{equation}{0}
\setcounter{theorem}{0}

\section*{Appendix C: Proof of \eqref{inf-sup-condition}}\label{AppendixC} 

In this appendix, we prove \eqref{inf-sup-condition} in the following lemma. 
\begin{lemma}\label{lemma-appendix-C}
Under assumptions {\bf (A1)}--{\bf (A4)} on the finite element spaces, the following type of  inf-sup condition holds (the $H^1(\Sigma)$-norm is involved on the right-hand side of the inequality){\rm:}
\begin{equation*}
	\| p_h \| {\leq C \sup_{0 \neq \bv_h \in \X^r_h}}  \frac{({\rm div} \bv_h,
		p_h)}{\| \bv_h \|_{H^1 }+\|\bv\|_{H^1(\Sigma)}} \quad
	\forall p_h \in Q^{r-1}_{h},
\end{equation*}
where $C>0$ is a constant independent of $p_h$ and the mesh size $h$. 
\end{lemma}
\begin{proof}
	Each $p_h\in Q_h^{r-1}$ can be decomposed into $p_h=\widetilde{p}_h+\bar{p}_h$, with $\widetilde{p}_h\in Q^{r-1}_{h,0}$ and $\bar{p}_h=\frac{1}{|\Omega|} \int_\Omega p_h {\rm d} x$. Since we have assumed that inf-sup condition \eqref{inf-sup-condition0} holds, there exists $\widetilde{\bv}_h\in \mathring{\X}_h^r$ such that 
	\begin{equation}\label{b-tilde-ph-vh} 
		\|\widetilde{\bv}_h\|_{H^1}\leq \|\widetilde{p}_h\| \quad \text{and}\quad b(\widetilde{p}_h,\widetilde{\bv}_h)\geq C_1\|\widetilde{p}_h\|^2
		\quad\mbox{for some constant $C_1>0$}.
	\end{equation}
	For the constant $\bar{p}_h\in\mathbb{R}$, we note that
	\begin{equation*}
		b(\bar{p}_h,\bv_h)=\bar{p}_h b(1,\bv_h)=\bar{p}_h(\bv_h,\n)_\Sigma.
	\end{equation*}
	Let $\bv_h^*\in \X_h^r$ be defined as $\bv_h^*=E_h(\n_h)$, where $\n_h$ is defined in \eqref{nh-def} and $E_h$ is the extension operator defined in item 4 of Remark \ref{remark-FE-assm}, i.e., $\bv_h^*=I_h^X \bv\in \X^r_h$ with $\bv\in H^1(\Omega)^d$ being an extension of $\n_h$ such that $\bv|_\Sigma=\n_h$. By the definition of $\bv_h^*$, we have 
	\begin{align*}
		\|\bv_h^*\|_{H^1} &= \|I_h^X \bv\|_{H^1} \leq C\|\bv\|_{H^1}\leq C\|\n_h\|_{H^{1/2}(\Sigma)}\leq C \\
		\|\bv_h^*\|_{H^1(\Sigma)} &=\|\n_h\|_{H^1(\Sigma)}\leq C.
	\end{align*}
Moreover, the following relation holds: 
	\begin{equation*}
		b(1,\bv_h^*)=(\bv_h^*,\n)_\Sigma=(\n_h,\n)_\Sigma=\|\n_h\|_\Sigma^2\geq C>0.
	\end{equation*}
	Therefore, the function $\bv_h^1:=\bar{p}_h\bv_h^*$ has the following property: 
	\begin{equation*}
		\|\bv_h^1\|_{H^1}+\|\bv_h^1\|_{H^1(\Sigma)}\leq C|\bar{p}_h|\leq C_0\|\bar{p}_h\|
		\quad\mbox{for some constant $C_0>0$} . 
	\end{equation*}
We can re-scale $\bv_h^1$ to $\bv_h^e=\frac{1}{C_0}\bv_h^1$ so that the following inequalities hold for some constant $C_2>0$: 
	\begin{equation}\label{vhe-H1}
		\|\bv_h^e\|_{H^1}+\|\bv_h^e\|_{H^1(\Sigma)}\leq \|\bar{p}_h\|\quad\text{and}\quad b(\bar{p}_h,\bv_h^e)= \frac{|\bar{p}_h|^2}{C_0}b(1,\bv_h^*)\geq C_2\|\bar{p}_h\|^2 . 
	\end{equation}
By considering $\bv_h=\widetilde{\bv}_h+\epsilon\bv_h^e$, with a parameter $\epsilon>0$ to be determined later, and using the relation $b(\bar{p}_h,\widetilde{\bv}_h)=\bar{p}_h( \widetilde{\bv}_h ,\n_h)_\Sigma = 0$, we have   
	\begin{align*}
		b(p_h,\bv_h)&=b(\widetilde{p}_h+\bar{p}_h,\widetilde{\bv}_h+\epsilon\bv_h^e)\\
		&=b(\widetilde{p}_h,\widetilde{\bv}_h)+\epsilon b(\widetilde{p}_h,\bv_h^e)+\epsilon b(\bar{p}_h,\bv_h^e)\\
		&\geq C_1\|\widetilde{p}_h\|^2+\epsilon b(\widetilde{p}_h,\bv_h^e) +\epsilon C_2\|\bar{p}_h\|^2
		&&\mbox{(here \eqref{b-tilde-ph-vh} and \eqref{vhe-H1} are used)}\\
		&\geq C_1\|\widetilde{p}_h\|^2-C\epsilon\|\widetilde{p}_h\|\|\bar{p}_h\| +\epsilon C_2\|\bar{p}_h\|^2 
		&&\mbox{(the first inequality in \eqref{vhe-H1} is used)} . 
	\end{align*}
By using Young's inequality, we can reduce the last inequality to the following one: 
	\begin{equation*}
		b(p_h,\bv_h)\geq C_1\|\widetilde{p}_h\|^2+\epsilon C_2\|\bar{p}_h\|^2-\left(\frac{C_1}{2}\|\widetilde{p}_h\|^2+\frac{C^2\epsilon^2}{2C_1}\|\bar{p}_h\|^2\right) . 
	\end{equation*}
	Then, choosing $\epsilon=\frac{C_1C_2}{C^2}$, we derive that 
	\begin{equation}\label{b-ph-vh}
		b(p_h,\bv_h)\geq \frac{C_1}{2}\|\widetilde{p}_h\|^2+\frac{C_1C_2^2}{2C^2}\|\bar{p}_h\|^2\geq C_3\|p_h\|^2 \quad\mbox{with $C_3:=\min \{\frac{C_1}{2},\frac{C_1C_2^2}{2C^2}\}$}. 
	\end{equation}
Since $\bv_h=\widetilde{\bv}_h+\epsilon\bv_h^e$ with $\widetilde{\bv}_h=0$ on $\Sigma$, it follows from the triangle inequality and \eqref{b-tilde-ph-vh}--\eqref{vhe-H1} that 
	\begin{align}\label{vh-H1-H1Sigma}
		\|\bv_h\|_{H^1}+\|\bv_h\|_{H^1(\Sigma)}&\leq \|\widetilde{\bv}_h\|_{H^1}+\epsilon(\|\bv_h^e\|_{H^1}+\|\bv_h^e\|_{H^1(\Sigma)})\\
		&\leq \|\widetilde{p}_h\|+\frac{C_1C_2}{C^2}\|\bar{p}_h\|\leq \Big(1+\frac{C_1C_2}{C^2}\Big)\|p_h\|. 
	\end{align}
 	Therefore, \eqref{b-ph-vh} and \eqref{vh-H1-H1Sigma} imply that 
 	\begin{equation*}
	\frac{b(p_h,\bv_h)}{\|\bv_h\|_{H^1}+\|\bv_h\|_{H^1(\Sigma)}}
	\ge \frac{C_3\|p_h\|^2}{\big(1+\frac{C_1C_2}{C^2}\big)\|p_h\|} 
	=\frac{C_3\|p_h\|}{\big(1+\frac{C_1C_2}{C^2}\big)} .
 	\end{equation*}
This proves that 
$$
\|p_h\|\leq \frac{1}{C_3}\Big(1+\frac{C_1C_2}{C^2}\Big)\frac{b(p_h,\bv_h)}{\|\bv_h\|_{H^1}+\|\bv_h\|_{H^1(\Sigma)}} . 
$$
and therefore completes the proof of \eqref{inf-sup-condition}.
\hfill\end{proof}

\end{document}